\documentclass[
11pt,  reqno]{amsart}
\pdfoutput=1

\usepackage[margin=1.2in,marginparwidth=1.5cm, marginparsep=0.5cm]{geometry}

\usepackage{booktabs} 
\usepackage{microtype}
\usepackage{amssymb}
\usepackage{mathrsfs}

\usepackage{color}
\usepackage[implicit=true]{hyperref}

\usepackage{xsavebox}

\usepackage{cases}

\allowdisplaybreaks[2]

\definecolor{gr}{rgb}   {0.,   0.69,   0.23 }
\definecolor{bl}{rgb}   {0.,   0.5,   1. }
\definecolor{mg}{rgb}   {0.85,  0.,    0.85}
\definecolor{yl}{rgb}   {0.8,  0.7,   0.}
\definecolor{or}{rgb}  {0.7,0.2,0.2}

\usepackage{tikz} 
\usepackage{marginnote}
\usepackage{scalerel} 

\usetikzlibrary{shapes.misc}
\usetikzlibrary{shapes.symbols}
\usetikzlibrary{decorations}
\usetikzlibrary{decorations.markings}

\tikzset{
	dot/.style={circle,fill=black,draw=black,inner sep=0pt,minimum size=0.5mm},
	>=stealth,
	}

\tikzset{
	dot2/.style={circle,fill=black,draw=black,inner sep=0pt,minimum size=0.2mm},
	>=stealth,
	}

\tikzset{
	ddot/.style={circle,fill=white,draw=black,inner sep=0pt,minimum size=0.8mm},
	>=stealth,
	}

\tikzset{decision/.style={ 
        draw,
        diamond,
        aspect=1.5
    }}

\tikzset{dia2/.style
={diamond,fill=white,draw=black,inner sep=0pt,minimum size=1mm},
	>=stealth,
	}

\tikzset{dia/.style
={star,fill=black,draw=black,inner sep=0pt,minimum size=1mm},
	>=stealth,
	}

\tikzset{dia/.style
={diamond,fill=black,draw=black,inner sep=0pt,minimum size=1.3mm},
	>=stealth,
	}

\makeatletter
\def\DeclareSymbol#1#2#3{\xsavebox{#1}{\tikz[baseline=#2,scale=0.15]{#3}}}
\def\<#1>{\xusebox{#1}}
\makeatother

\newcommand{\pe}{\mathbin{\scaleobj{0.7}{\tikz \draw (0,0) node[shape=circle,draw,inner sep=0pt,minimum size=8.5pt] {\scriptsize  $=$};}}}

\newcommand{\pl}{\mathbin{\scaleobj{0.7}{\tikz \draw (0,0) node[shape=circle,draw,inner sep=0pt,minimum size=8.5pt] {\scriptsize $<$};}}}
\newcommand{\pg}{\mathbin{\scaleobj{0.7}{\tikz \draw (0,0) node[shape=circle,draw,inner sep=0pt,minimum size=8.5pt] {\scriptsize $>$};}}}

\newcommand{\pez}{\mathbin{\scaleobj{0.7}{\tikz \draw (0,0) node[shape=circle,draw,
fill=white, 
inner sep=0pt,minimum size=8.5pt]{} ;}}}

\usetikzlibrary{decorations.pathmorphing, patterns,shapes}

\DeclareSymbol{dot}{-2.7}{\draw (0,0) node[dot]{};}

\DeclareSymbol{1}{0}{\draw[white] (-.4,0) -- (.4,0); \draw (0,0)  -- (0,1.2) node[dot] {};}
\DeclareSymbol{2}{0}{\draw (-0.5,1.2) node[dot] {} -- (0,0) -- (0.5,1.2) node[dot] {};}
\DeclareSymbol{3}{0}{\draw (0,0) -- (0,1.2) node[dot] {}; \draw (-.7,1) node[dot] {} -- (0,0) -- (.7,1) node[dot] {};}
\DeclareSymbol{30}{-3}{\draw (0,0) -- (0,-1); \draw (0,0) -- (0,1.2) node[dot] {}; \draw (-.7,1) node[dot] {} -- (0,0) -- (.7,1) node[dot] {};}
\DeclareSymbol{31}{-3}{\draw (0,0) -- (0,-1) -- (1,0) node[dot] {}; \draw (0,0) -- (0,1.2) node[dot] {}; \draw (-.7,1) node[dot] {} -- (0,0) -- (.7,1) node[dot] {};}
\DeclareSymbol{32}{-3}{\draw (0,0) -- (0,-1) -- (1,0) node[dot] {}; \draw (0,0) -- (0,-1) -- (-1,0) node[dot] {}; \draw (0,0) -- (0,1.2) node[dot] {}; \draw (-.7,1) node[dot] {} -- (0,0) -- (.7,1) node[dot] {};}

\DeclareSymbol{320}{-3}{
\draw (0,0.4) -- (0,-1.2);
\draw (0,0.6) -- (0,-0.4) -- (1,0.6) node[dot] {}; 
\draw (0,0.6) -- (0,-0.4) -- (-1,0.6) node[dot] {}; 
\draw (0,0.6) -- (0,1.8) node[dot] {};
 \draw (-.7, 1.6) node[dot] {} -- (0,0.6) -- (.7,1.6) node[dot] {};}

\DeclareSymbol{20}{-3}{\draw (0,0) -- (0,-1);\draw (-.7,1) node[dot] {} -- (0,0) -- (.7,1) node[dot] {};}
\DeclareSymbol{22}{-3}{\draw (0,0.3) -- (0,-1) -- (1,0) node[dot] {}; \draw (0,0.3) -- (0,-1) -- (-1,0) node[dot] {};\draw (-.7,1) node[dot] {} -- (0,0.3) -- (.7,1) node[dot] {};}
\DeclareSymbol{31p}{-3}{\draw (0,0) -- (0,-1) -- (1,0) node[dot] {}; \draw (0,0) -- (0,1.2) node[dot] {}; \draw (-.7,1) node[dot] {} -- (0,0) -- (.7,1) node[dot] {}; 
\draw (0,-1) node{\scaleobj{0.5}{\pez}}; 
\draw (0,-1) node{\scaleobj{0.5}{\pe}}; }
\DeclareSymbol{32p}{-3}{\draw (0,0) -- (0,-1) -- (1,0) node[dot] {}; \draw (0,0) -- (0,-1) -- (-1,0) node[dot] {}; \draw (0,0) -- (0,1.2) node[dot] {}; \draw (-.7,1) node[dot] {} -- (0,0) -- (.7,1) node[dot] {}; 
\draw (0,-1) node{\scaleobj{0.5}{\pez}};
\draw (0,-1) node{\scaleobj{0.5}{\pe}};}
\DeclareSymbol{22p}{-3}{\draw (0,0.3) -- (0,-1) -- (1,0) node[dot] {}; \draw (0,0.3) -- (0,-1) -- (-1,0) node[dot] {};\draw (-.7,1) node[dot] {} -- (0,0.3) -- (.7,1) node[dot] {}; 
\draw (0,-1) node{\scaleobj{0.5}{\pez}};
\draw (0,-1) node{\scaleobj{0.5}{\pe}};}
\DeclareSymbol{21p}{-3}{\draw (0,0.3) -- (0,-1) -- (1,0) node[dot] {}; 
\draw (-.7,1) node[dot] {} -- (0,0.3) -- (.7,1) node[dot] {}; 
\draw (0,-1) node{\scaleobj{0.5}{\pez}};
\draw (0,-1) node{\scaleobj{0.5}{\pe}};}

\DeclareSymbol{70}{-3}{
\draw (0,-0.4) -- (0,-1.2);
\draw (0,0.8) node[dot] {} -- (0,-0.4) -- (0.8,0.6); 
\draw (0.8, .6) -- (1.4, 1.35)  node[dot] {}; 
\draw (0.8, .6) -- (1.8, .8)  node[dot] {}; 
\draw (0.8, .6) -- (0.85, 1.6)  node[dot] {}; 
\draw (0,.6) -- (0,-0.4) -- (-0.8,0.6) ;
\draw (-0.8, 0.6) -- (-1.4, 1.35)  node[dot] {}; 
\draw (-0.8, .6) -- (-1.8, .8)  node[dot] {}; 
\draw (-0.8, .6) -- (-0.85, 1.6)  node[dot] {}; 
}

\DeclareSymbol{7}{-3}{
\draw (0,0.2) node[dot] {} -- (0,-1) -- (0.8,0); 
\draw (0.8, 0) -- (1.4, 0.75)  node[dot] {}; 
\draw (0.8, 0) -- (1.8, 0.2)  node[dot] {}; 
\draw (0.8, 0) -- (0.85, 1)  node[dot] {}; 
\draw (0,0) -- (0,-1) -- (-0.8,0) ;
\draw (-0.8, 0) -- (-1.4, 0.75)  node[dot] {}; 
\draw (-0.8, 0) -- (-1.8, 0.2)  node[dot] {}; 
\draw (-0.8, 0) -- (-0.85, 1)  node[dot] {}; 
}

\DeclareSymbol{11p}{-3}{\draw (0,0) node[dot2] {} -- (0,-1) -- (1,0) node[dot] {}; 
\draw (0,0) -- (0,1.2) node[dot] {};  
\draw (0,-1) node{\scaleobj{0.5}{\pez}}; 
\draw (0,-1) node{\scaleobj{0.5}{\pe}}; }
\DeclareSymbol{100}{-3}
{\draw[white] (-.4,0) -- (.4,0); 
\draw (0,0) node[dot2] {} -- (0,-1)  ; 
\draw (0,0) -- (0,1.2) node[dot] {};}  
\usetikzlibrary{mindmap,backgrounds,trees,arrows,decorations.pathmorphing,decorations.pathreplacing,positioning,shapes.geometric,shapes.misc,decorations.markings,decorations.fractals,calc,patterns}

\tikzset{>=stealth',
         cvertex/.style={circle,draw=black,inner sep=1pt,outer sep=3pt},
         vertex/.style={circle,fill=black,inner sep=1pt,outer sep=3pt},
         star/.style={circle,fill=yellow,inner sep=0.75pt,outer sep=0.75pt},
         tvertex/.style={inner sep=1pt,font=\scriptsize},
         gap/.style={inner sep=0.5pt,fill=white}}
\tikzstyle{mybox} = [draw=black, fill=blue!10, very thick,
    rectangle, rounded corners, inner sep=10pt, inner ysep=20pt]
\tikzstyle{boxtitle} =[fill=blue!50, text=white,rectangle,rounded corners]

\tikzstyle{decision} = [diamond, draw, fill=blue!20,
    text width=4.5em, text badly centered, node distance=2.5cm, inner sep=0pt]
\tikzstyle{block} = [rectangle, draw, fill=blue!20,
    text width=2.7cm, text centered, rounded corners, minimum height=3em]

\tikzstyle{line} = [draw, very thick, color=black!50, -latex']

\tikzstyle{cloud} = [draw, ellipse,fill=red!40, 
node distance=2.5cm,
    minimum height=1em]

\tikzstyle{cloud2} = [draw, ellipse,fill=red!30, text=white,text width=10em, node distance=2.5cm, text centered, minimum height=4em]

\tikzstyle{cloud3} = [draw, ellipse, fill=cyan!30, 
node distance=2.5cm,
    minimum height=3em, 
    text width=1.7cm]

\tikzstyle{cloud4} = [draw, ellipse,fill=orange!70, node distance=2.5cm,
   text width=2.1cm,  
           minimum height=3em]

\tikzstyle{cloud5} = [draw, ellipse,fill=red!20, node distance=2.5cm,
    text centered, 
    text width=3.0cm, 
    minimum height=3em]

\tikzstyle{cloud6} = [draw, ellipse,fill=red!20, node distance=2.5cm,
    text width=1.5cm, 
    text centered, 
    minimum height=3em]

\tikzset{
    position/.style args={#1:#2 from #3}{
        at=(#3.#1), anchor=#1+180, shift=(#1:#2)
    }
}

\newtheorem{theorem}{Theorem} [section]

\newtheorem{lemma}[theorem]{Lemma}
\newtheorem{proposition}[theorem]{Proposition}
\newtheorem{remark}[theorem]{Remark}

\newtheorem{definition}[theorem]{Definition}

\DeclareMathOperator*{\supp}{supp}
\DeclareMathOperator{\med}{med}

\newcommand{\noi}{\noindent}
\newcommand{\Z}{\mathbb{Z}}
\newcommand{\R}{\mathbb{R}}

\newcommand{\T}{\mathbb{T}}

\let\P= \undefined
\newcommand{\P}{\mathbf{P}}

\newcommand{\E}{\mathbb{E}}

\renewcommand{\L}{\mathcal{L}}

\newcommand{\F}{\mathcal{F}}

\newcommand{\al}{\alpha}
\newcommand{\be}{\beta}
\newcommand{\dl}{\delta}

\newcommand{\nb}{\nabla}

\newcommand{\Dl}{\Delta}
\newcommand{\eps}{\varepsilon}
\newcommand{\kk}{\kappa}
\newcommand{\g}{\gamma}
\newcommand{\G}{\Gamma}
\newcommand{\ld}{\lambda}

\newcommand{\s}{\sigma}

\newcommand{\ft}{\widehat}

\newcommand{\wt}{\widetilde}
\newcommand{\cj}{\overline}

\newcommand{\dt}{\partial_t}
\newcommand{\dd}{\partial}

\newcommand{\ta}{\theta}

\renewcommand{\l}{\ell}
\renewcommand{\o}{\omega}
\renewcommand{\O}{\Omega}

\newcommand{\les}{\lesssim}
\newcommand{\ges}{\gtrsim}

\newcommand{\jb}[1]
{\langle #1 \rangle}

\newcommand{\ind}{\mathbf 1}

\newcommand{\too}{\longrightarrow}

\def\e{\eps}

\newcommand{\N}{\mathbb{N}}

\renewcommand{\H}{\mathcal{H}}

\newcommand{\Sep}{\<70>_{\hspace{-1.7mm} N}}
\newcommand{\Sepp}{\<7>_{N}}
\newcommand{\fSepp}{\ft{\<7>}_{N}}
\newtheorem*{ackno}{Acknowledgements}

\newcommand{\I}{\mathcal{I}}

\newcommand{\If}{\mathfrak{I}}

\newcommand{\RR}{\mathcal{R}}

\numberwithin{equation}{section}
\numberwithin{theorem}{section}

\DeclareMathOperator{\Sym}{\mathtt{Sym}}
\newcommand{\EE}{\mathcal{E}}
\newcommand{\hf}{\mathfrak{h}}
\newcommand{\Hf}{\mathfrak{H}}
\newcommand{\TT}{\mathcal{T}}

\makeatletter
\@namedef{subjclassname@2020}{%
  \textup{2020} Mathematics Subject Classification}
\makeatother

\begin{document}
\baselineskip = 14pt

\title[3-$d$ cubic SNLW 
with almost space-time white noise]
{Three-dimensional  stochastic cubic nonlinear wave equation
with almost space-time white noise}

\author[T.~Oh, Y.~Wang, and Y.~Zine]
{Tadahiro Oh, Yuzhao Wang, and Younes Zine}

\address{
Tadahiro Oh, School of Mathematics\\
The University of Edinburgh\\
and The Maxwell Institute for the Mathematical Sciences\\
James Clerk Maxwell Building\\
The King's Buildings\\
Peter Guthrie Tait Road\\
Edinburgh\\ 
EH9 3FD\\
 United Kingdom}

\email{hiro.oh@ed.ac.uk}

\address{
Yuzhao Wang\\
School of Mathematics\\
Watson Building\\
University of Birmingham\\
Edgbaston\\
Birmingham\\
B15 2TT\\ United Kingdom}

\email{y.wang.14@bham.ac.uk}

\address{
Younes Zine,  School of Mathematics\\
The University of Edinburgh\\
and The Maxwell Institute for the Mathematical Sciences\\
James Clerk Maxwell Building\\
The King's Buildings\\
Peter Guthrie Tait Road\\
Edinburgh\\ 
EH9 3FD\\
 United Kingdom}

\email{y.p.zine@sms.ed.ac.uk}


\dedicatory{Dedicated to 
Professor Istv\'an Gy\"ongy  on the
occasion of his seventieth birthday}


\begin{abstract}
We study the stochastic cubic nonlinear wave  equation (SNLW) 
 with an additive noise on the three-dimensional torus $\T^3$. 
 In particular, we prove local well-posedness of the (renormalized) SNLW when the noise 
 is almost a space-time white noise. 
In recent years, 
the paracontrolled calculus has played a crucial  role  
 in the  well-posedness study of  singular SNLW on $\T^3$
 by Gubinelli, Koch, and the first author (2018), 
 Okamoto, Tolomeo, and the first author (2020), 
 and Bringmann (2020).
 Our approach, however, 
 does not rely on the paracontrolled calculus.
 We instead proceed with the second order expansion 
 and study the resulting equation for the residual term, 
 using multilinear dispersive smoothing.

\end{abstract}

%
\maketitle

%
\tableofcontents

\newpage

\section{Introduction}
\label{SEC:1}

\subsection{Singular stochastic nonlinear wave equation}

In this paper, 
we study the following Cauchy problem
for the stochastic nonlinear wave equation (SNLW)
with a cubic nonlinearity
on the three dimensional torus
$\T^3=(\R/(2\pi \Z))^3$, driven by an additive noise:
\begin{align}
\begin{cases}
\dt^2 u  + (1 -  \Dl)  u +  u^3  = \phi \xi\\
(u, \dt u) |_{t = 0} = (u_0, u_1), 
\end{cases}
\quad (x, t) \in \T^3\times \R,
\label{SNLW1}
\end{align}

\noi
where
$\xi(x, t)$ denotes a (Gaussian) space-time white noise on $\T^3\times \R$
with the space-time covariance given by
\[ \E\big[ \xi(x_1, t_1) \xi(x_2, t_2) \big]
= \dl(x_1 - x_2) \dl (t_1 - t_2) \]

\noi
and $\phi$ is a bounded operator on $L^2(\T^3)$.
Our main goal is to present a concise proof of  local well-posedness of~\eqref{SNLW1}, 
when $\phi$ is the Bessel potential of order $\al$:
\begin{align}
\phi = \jb{\nb}^{-\al} = (1- \Dl)^{-\frac \al2}
\label{P1}
\end{align}

\noi
for any $\al > 0$.
Namely, we consider \eqref{SNLW1} with an ``almost'' space-time white noise.

Given $\al \in \R$, let $\phi = \phi_\al$ be as in \eqref{P1}.
Then, a standard computation shows that  the stochastic convolution: 
\begin{align}
 \<1>  
= \I  (\jb{\nb}^{-\al}\xi)
\label{P2}
\end{align}
\noi
belongs almost surely to $C(\R; W^{s, \infty}(\T^3))$
for any $s < \al - \frac 12$. See Lemma \ref{LEM:sto1} below.
Here,  we adopted Hairer's convention to denote  stochastic terms by trees;
the vertex ``\,$\<dot>$\,'' in $\<1>$ corresponds to 
the random noise $\phi \xi = \jb{\nb}^{-\al} \xi$,
while the edge denotes the Duhamel integral operator:
 \begin{align}
 \I = (\dt^2 + (1 - \Dl))^{-1}, 
\label{P3}
 \end{align}
 
 \noi
corresponding to 
the forward fundamental solution to the linear wave equation.
Note that 
when $\al > \frac 12$, 
 the stochastic convolution $\<1>$
 is a function of   positive (spatial) regularity $\al - \frac 12-\eps$.\footnote{In this discussion, 
 we only discuss spatial regularities.
 Moreover, we do not worry about the regularity of the initial data $(u_0, u_1)$.}
Then, by  proceeding with 
the first order expansion:
\begin{align*}
u = \<1> + v
\end{align*}

\noi
and studying the equation for the residual term $v = u - \<1>$, 
we can  show that \eqref{SNLW1} is locally well-posed,
when $\al > \frac 12$.
See \cite{BT2, Poc} in the case 
of the deterministic cubic nonlinear wave equation (NLW):
\begin{align}
\dt^2 u  + (1 -  \Dl)  u +  u^3  = 0
\label{NLW0}
\end{align}

\noi
with random initial data.
Furthermore, by controlling the growth of the $\H^1$-norm 
of the residual term $v$ via a Gronwall-type argument, we can prove global well-posedness
of \eqref{SNLW1},  when $\al > \frac 12$.\footnote{This globalization argument is the only place, 
where the defocusing nature of the nonlinearity plays a role.
See also Remarks \ref{REM:triv} and \ref{REM:gwp}.
In particular, all the local-in-time results, 
including Theorem \ref{THM:1},   also hold in the focusing case.}
See \cite{BT2}.

When $\al \le \frac 12$, 
solutions  to \eqref{SNLW1} are expected to be  merely  distributions
of  negative regularity $\al -\frac 12  -\eps$, inheriting the regularity 
of the stochastic convolution, 
and thus we need to consider the renormalized version of \eqref{SNLW1},
which formally reads
\begin{align}
\begin{cases}
\dt^2 u  + (1 -  \Dl)  u +  u^3 - \infty\cdot u  = \jb{\nb}^{-\al} \xi\\
(u, \dt u) |_{t = 0} = (u_0, u_1), 
\end{cases}
\label{SNLW2}
\end{align}

\noi
where the formal expression $u^3- \infty\cdot u $ 
denotes the renormalization of the cubic power~$u^3$. 
In the range $\frac 14 < \al \leq \frac 12$, 
a straightforward computation 
 with the second order expansion:
\begin{align*}
u = \<1> - \<30> + v
\end{align*}

\noi
yields local well-posedness of the renormalized SNLW \eqref{SNLW2}
(in the sense of Theorem \ref{THM:1} below).
Here, the second order process $\<30>$ is defined by 
\begin{align*}
 \<30> = \I (\<3>), 
\end{align*}

\noi
where $\<3>$ denotes the renormalized version of $\<1>^3$.
See
 \cite{OPTz}
for this argument in the context of the deterministic renormalized cubic NLW \eqref{NLW0}
with random initial data.

We state our main result.

\begin{theorem}\label{THM:1}
Let $0 < \al \le \frac 12$.
Given $ s >  \frac 12$, 
let $(u_0, u_1) \in \H^{s}(\T^3) = H^s(\T^3)\times H^{s-1}(\T^3)$.
Then, there exists a unique local-in-time solution to the renormalized cubic SNLW \eqref{SNLW2}
with $(u, \dt u)|_{t = 0} = (u_0, u_1)$.

More precisely, given $N \in \N$, let $ \xi_N =  \pi_N \xi$, 
where $\pi_N$ is the frequency projector onto the spatial frequencies $\{|n|\leq N\}$
defined in~\eqref{pi} below.
Then,  there exists a sequence of time-dependent constants $\{\s_N(t)\}_{N\in \N}$
 tending to $\infty$ \textup{(}see \eqref{sigma1} below\textup{)} such that, given small $\eps = \eps(s) > 0$, 
the solution $u_N$ to the following truncated renormalized SNLW\textup{:}
\begin{align}
\begin{cases}
\dt^2 u_N + (1-  \Dl)  u_N  + u_N^3  - 3\s_N u_N  = \jb{\nb}^{-\al}  \xi_N\\
(u_N, \dt u_N)|_{t = 0} = (u_0, u_1)
\end{cases}
\label{SNLW3}
\end{align}
converges to a non-trivial\footnote{Here, non-triviality means that the limiting process $u$ is not zero 
 or a linear solution. 
 As we see below, the limiting process $u$ admits a decomposition 
$ u = \<1> - \<30> + v$, 
where the residual term $v$ satisfies the nonlinear equation~\eqref{SNLW11}.
See Remark \ref{REM:triv}\,(ii)
on a triviality result for the unrenormalized equation.
See also  \cite{HRW,OOR, OPTz, ORSW} for related triviality results.} stochastic process $u \in C([-T, T]; H^{\al -\frac 12 -\eps} (\T^3))$ almost surely,
where $T = T(\o)$ is an almost surely positive stopping time.
\end{theorem}

Stochastic nonlinear wave equations 
 have been studied extensively
in various settings; 
see \cite[Chapter 13]{DPZ14} for the references therein.
In particular, over the last few years, we have  witnessed a rapid progress
in the theoretical understanding of
 nonlinear wave equations with singular stochastic forcing
 and/or rough random initial data; see 
\cite{OTh2, GKO, GKO2, GKOT,  OPTz,ORTz, OOR, OOTz, Tolo, 
Deya1, 
 ORSW, OOComp, ORW, Deya2, ORSW2, OOTol, Bring, 
 OOTol2}.
  In \cite{GKO2}, 
Gubinelli, Koch, and the first author studied 
the quadratic SNLW on $\T^3$:
\begin{align}
\dt^2 u  + (1 -  \Dl)  u +  u^2   = \xi.
\label{SNLW4}
\end{align}

\noi
By adapting the paracontrolled calculus \cite{GIP}, 
originally introduced 
by Gubinelli, Imkeller, and Perkowski
in the study of stochastic parabolic PDEs,  to the dispersive setting, 
the authors of \cite{GKO2} reduced~\eqref{SNLW4} into a system of two unknowns.
This system  was then shown to be locally well-posed
by exploiting the following two ingredients:
(i)~multilinear dispersive smoothing coming from a multilinear interaction of 
 random waves (see also \cite{OOComp, Bring})
 and 
(ii) novel 
random operators (the so-called 
paracontrolled operators) which incorporate
the paracontrolled structure in their definition.
These random operators
are used to replace
 commutators which are standard in the parabolic paracontrolled approach~\cite{CC, MW1}.

More recently, 
Okamoto, Tolomeo, and the first author \cite{OOTol}
and Bringmann \cite{Bring}
independently studied the following SNLW
with a cubic Hartree-type nonlinearity:\footnote{In \cite{Bring}, Bringmann studied
the corresponding deterministic Hartree NLW with random initial data.}
\begin{align}
\dt^2 u  + (1 -  \Dl)  u +  (V*u^2) u   = \xi, 
\label{SNLW5}
\end{align}

\noi
where $V$ is the kernel of  the Bessel potential 
$\jb{\nb}^{-\be}$ 
of order $\be > 0$.\footnote{We point out that
the scope of the papers \cite{OOTol, Bring} goes much further than what is described here.
The main goal of \cite{OOTol} is to study the focusing
problem, in particular the (non-)construction of the focusing Gibbs measure
associated to the focusing Hartree SNLW.  They  identified
the critical value $\be = 2$ and 
proved sharp global well-posedness of the focusing problem
(with a small coefficient in front of the nonlinearity when $\be = 2$).
%
On the other hand, the main goal in \cite{Bring}
is the construction of global-in-time dynamics in the defocusing case, 
where there was a significant difficulty in  adapting Bourgain's invariant measure argument
\cite{BO94, BO96}. This is due to (i) the singularity of  the associated Gibbs measure
 with respect to the base Gaussian free field 
for $0 < \be \leq \frac 12$
 \cite{OOTol, Bring1} and (ii) the paracontrolled structure imposed in the local theory, which
must be propagated in the construction of global-in-time solutions.
See the introductions of \cite{OOTol, Bring} for further discussion.
}
In \cite{OOTol}, the authors proved local well-posedness
for $\be > 1$ by viewing the nonlinearity as the nested bilinear interactions
and utilizing the paracontrolled operators introduced in \cite{GKO2}.
In \cite{Bring}, Bringmann went much further
and proved local well-posedness of \eqref{SNLW5}
for any $\be > 0$.
The main strategy in \cite{Bring}
is to extend the paracontrolled approach in \cite{GKO2}
to the cubic setting.
The main task is then  
to study regularity  properties of various random operators
and random distributions.
This was done by an intricate  combination of 
deterministic analysis, stochastic analysis, counting arguments, 
the random matrix/tensor approach by Bourgain \cite{BO96, BO97} 
and Deng, Nahmod, and Yue~\cite{DNY2}, 
and 
the physical space approach via 
the (bilinear) Strichartz estimates due to 
Klainerman and Tataru~\cite{KT}, 
analogous to the random data Cauchy theory 
for the nonlinear Schr\"odinger equations on $\R^d$ as in \cite{BOP1, BOP2, BOP3}.

From the scaling point of view, 
the cubic SNLW \eqref{SNLW2} with a slightly smoothed space-time white noise
(i.e.~small $\al >0$) is essentially the same as the Hartree SNLW
\eqref{SNLW5} with small $\be > 0$.
Hence, Theorem \ref{THM:1} is expected to hold in view of Bringmann's recent result \cite{Bring}.
The main point of this paper is that we present a concise proof of 
Theorem \ref{THM:1}
{\it without} using the paracontrolled calculus.
In the next subsection, we outline our strategy.

Due to the time reversibility of the equation, 
we only consider positive times in the remaining part of the paper.

\begin{remark}\rm

The equations \eqref{SNLW1}
and \eqref{SNLW2}
indeed correspond
to the stochastic nonlinear Klein-Gordon equations.
The same results with  inessential modifications
also hold for
the stochastic nonlinear wave equation,
where we replace the linear part  in \eqref{SNLW1}
and \eqref{SNLW2}
by $\dt^2 u- \Dl u$.
In the following, we simply refer to \eqref{SNLW1} 
and \eqref{SNLW2} as 
the stochastic nonlinear wave equations.

\end{remark}

\begin{remark}\rm
Our argument also applies to the deterministic (renormalized)
cubic NLW on~$\T^3$ with random initial data of the form:
\begin{equation*}
(u_0^\o, u_1^\o)  = \bigg(\sum_{n \in \Z^3} \frac{g_n(\o)}{\jb{n}^{1+\al}}e^{in\cdot x}, 
\sum_{n \in \Z^3} \frac{h_n(\o)}{\jb{n}^{\al}}e^{in\cdot x}\bigg), 
\end{equation*}

\noi
where the series   $\{ g_n \}_{n \in \Z^3}$ and  $\{ h_n \}_{n \in \Z^3}$ are two families of 
independent standard   complex-valued  Gaussian random variables  
 conditioned that  $g_n=\overline{g_{-n}}$,  $h_n=\overline{h_{-n}}$, 
$n \in \Z^3$.
In particular, Theorem \ref{THM:1} provides an improvement
of the main result (almost sure local well-posedness) in~\cite{OPTz}
from $\al > \frac 14$ to $\al > 0$.

\end{remark}

\begin{remark}\label{REM:triv}\rm
(i) The first part of the statement in Theorem \ref{THM:1}
is merely a formal statement in view
of the divergent  behavior $\s_N (t) \to \infty$ for $t\ne 0$.
In the next subsection, we provide a precise meaning
to what it means to be a solution to \eqref{SNLW2}
and also make the uniqueness statement more precise.
See Remark \ref{REM:uniq}.

\smallskip
\noi
(ii) In the case of the defocusing cubic SNLW with damping:
\begin{align*}
\dt^2 u  + \dt u + (1 -  \Dl)  u +  u^3   = \jb{\nb}^{-\al} \xi, 
\end{align*}

\noi
 a combination of our argument with that in \cite{OOR} yield the
 following triviality result.
Consider 
the following truncated (unrenormalized)  SNLW with damping:
\begin{align*}
\begin{cases}
\dt^2 u_N + \dt u_N + (1-  \Dl)  u_N  + u_N^3   = \jb{\nb}^{-\al}  \xi_N\\
(u_N, \dt u_N)|_{t = 0} = (u_0, u_1), 
\end{cases}
\end{align*}

\noi
where $\xi_N = \pi_N \xi$.
As we remove the regularization (i.e.~take $N\to \infty$), 
the solution $u_N$ converges in probability 
to the trivial function $u_\infty \equiv 0$
for any (smooth) initial data $(u_0, u_1)$. 
See \cite{OOR} for details.

\end{remark}

\begin{remark}\rm

(i) In our proof, we use the Fourier restriction norm method
(i.e.~the $X^{s, b}$-spaces defined in \eqref{Xsb}), following \cite{OTh2, Bring}.
While it may be  possible to give a proof of Theorem \ref{THM:1}
based only on the physical-side spaces (such as the Strichartz spaces)
as in \cite{GKO, GKO2, GKOT}, 
we do not pursue this direction
since our main goal is to present a concise proof of Theorem \ref{THM:1}
by adapting various estimates in  \cite{Bring}
to our current setting.
Note that the use of the physical-side spaces would allow us to take 
the initial data $(u_0, u_1)$ in the critical space $\H^\frac{1}{2}(\T^3)$
(for the cubic NLW on $\T^3$).
See for example  \cite{GKO}.
 One may equally use 
 the Fourier restriction norm method adapted to the space of functions of bounded $p$-variation and its pre-dual, 
  introduced and developed by Tataru, Koch, and their collaborators 
  \cite{KochT, HHK, HTT}, 
which would also allow us to 
 take 
the initial data $(u_0, u_1)$ in the critical space $\H^\frac{1}{2}(\T^3)$.
See for example \cite{BOP2, OOP}
in the context of the nonlinear Schr\"odinger equations with random initial data.
Since our main focus is to handle rough noises
(and not about rough deterministic initial data), 
we do not pursue this direction.

\smallskip

\noi
(ii) On $\T^3$, the Bessel potential  $\phi_\al = \jb{\nb}^{-\al}$
is Hilbert-Schmidt from $L^2(\T^3)$ to $H^s(\T^3)$
for $s < \al - \frac 32$.
It would be of interest to extend Theorem \ref{THM:1}
to a general Hilbert-Schmidt operator~$\phi$, say  from 
$L^2(\T^3)$ to $H^{\al - \frac 32}(\T^3)$ as in \cite{DD, OPW, OOk}.\footnote{Or a general $\g$-radonifying
operator $\phi$ as in \cite{FOW}, 
where the authors proved local well-posedness
of the one-dimensional stochastic cubic nonlinear Schr\"odinger equation
with an almost space-time white noise.}
Note that our argument uses the independence of the Fourier coefficients
of the stochastic convolution $\<1>$
but that such independence will be lost for 
a general Hilbert-Schmidt operator $\phi$.
\end{remark}

\begin{remark}\label{REM:gwp}\rm
(i) When $\al  = 0$, 
SNLW \eqref{SNLW2} with damping
\begin{align}
\dt^2 u  + \dt u + (1 -  \Dl)  u +  u^3 - \infty\cdot u  =  \xi
\label{SNLW8}
\end{align}

\noi
corresponds to 
 the so-called canonical stochastic quantization equation\footnote{Namely, the Langevin equation
 with the momentum $v = \dt u$.}
 for the Gibbs measure given by the $\Phi^4_3$-measure on $u$
 and the white noise measure on $\dt u$. See~\cite{RSS}.
In this case (i.e.~when $\al = 0$), 
our approach and the more sophisticated approach of Bringmann \cite{Bring}
for~\eqref{SNLW5} with  $\be > 0$
completely break down.
This is a very challenging problem, 
for which one would certainly need to use the paracontrolled approach in \cite{GKO2, OOTol, Bring}
and combine with the techniques in~\cite{DNY2}.

\smallskip

\noi
(ii) 
As mentioned above, when $\al > \frac 1 2$, 
the globalization argument by Burq and Tzvetkov \cite{BT2} yields
global well-posedness of 
SNLW \eqref{SNLW1} with $\phi$ as in \eqref{P1}.
When $\al = 0$, 
we expect that (a suitable adaptation of) Bourgain's invariant measure argument
would yield almost sure global well-posedness
once we could prove local well-posedness of \eqref{SNLW8}
(but this is a very challenging problem).
It would be of interest to investigate the issue of global well-posedness
of \eqref{SNLW2} for $0 < \al \le \frac 12$.
See \cite{GKOT, Tolo}
for the global well-posedness results 
on SNLW with an additive space-time white noise
in the two-dimensional case.


\end{remark}

\subsection{Outline of the proof}
Let us now describe the strategy  to prove Theorem \ref{THM:1}.
 Let  $W$ denote a cylindrical Wiener process on $L^2(\T^3)$:\footnote{By convention, 
 we endow $\T^3$ with the normalized Lebesgue measure $(2\pi)^{-3} dx$.}
\begin{align*}
W(t)
 = \sum_{n \in \Z^3} B_n (t) e_n, 
\end{align*}

\noi
where 
$e_n(x) = e^{ i n \cdot x}$ and 
$\{ B_n \}_{n \in \Z^3}$ 
is defined by 
$B_n(t) = \jb{\xi, \ind_{[0, t]} \cdot e_n}_{ x, t}$.
Here, $\jb{\cdot, \cdot}_{x, t}$ denotes 
the duality pairing on $\T^3\times \R$.
As a result, 
we see that $\{ B_n \}_{n \in \Z^3}$ is a family of mutually independent complex-valued
Brownian motions conditioned so that $B_{-n} = \cj{B_n}$, $n \in \Z^3$. 
In particular, $B_0$ is  a standard real-valued Brownian motion.
Note that we have, for any $n \in \Z^2$,  
 \[\text{Var}(B_n(t)) = \E\big[
 \jb{\xi, \ind_{[0, t]} \cdot e_n}_{x, t}\cj{\jb{\xi, \ind_{[0, t]} \cdot e_n}_{x, t}}
 \big] = \|\ind_{[0, t]} \cdot e_n\|_{L^2_{x, t}}^2 = t.\]

With this notation, we can formally write 
the stochastic convolution $\<1> = \I(\jb{\nb}^{-\al} \xi) $ in~\eqref{P2} 
as
\begin{align}
\<1>  
 = \int_{0}^t \frac{\sin ((t-t')\jb{\nb})}{\jb{\nb}^{1+\al}}dW(t')
 = \sum_{n \in \Z^3} e_n 
  \int_0^t \frac{\sin ((t - t') \jb{ n })}{\jb{ n }^{1+\al}} d B_n (t'), 
\label{sto1}
\end{align}

\noi
where $\jb{\nb} = \sqrt{1-\Dl}$ and $\jb{n} = \sqrt{1 + |n|^2}$.
We indeed construct the stochastic convolution $\<1>$ in~\eqref{sto1}
as the  limit of
 the truncated stochastic convolution  $\<1>_N$ defined by 
\begin{align}
\<1>_N 
= \I(\pi_N \jb{\nb}^{-\al} \xi) 
= \sum_{\substack{n \in \mathbb{Z}^3 \\ |n| \le N}} e_n 
  \int_0^t \frac{\sin ((t - t') \jb{ n })}{\jb{ n }^{1+\al}} d B_n (t')
\label{sto2}
\end{align}

\noi
for $N \in \N$, 
where $\pi_N$  denotes the (spatial) frequency projector 
defined  by 
\begin{align}
\pi_N f = 
\sum_{ |n| \leq N}  \ft f (n)  \, e_n.
\label{pi}
\end{align}

\noi
A standard computation shows that 
the sequence $\{ \<1>_N\}_{N \in \N}$
is almost surely Cauchy in\footnote{Hereafter, we use $a-$ 
(and $a+$) to denote $a- \eps$ (and $a+ \eps$, respectively)
for arbitrarily small $\eps > 0$.
If this notation appears in an estimate, 
then  an implicit constant 
is allowed to depend on $\eps> 0$ (and it usually diverges as $\eps \to 0$).
} 
$C([0,T];W^{\al - \frac 12 - ,\infty}(\T^3))$ 
and thus converges almost surely to some limit, which we denote by~$\<1>$, 
in the same space.
See Lemma \ref{LEM:sto1} below.

We then define the Wick powers  $\<2>_N$ 
and $\<3>_N$  by 
\begin{align}
\begin{split}
\<2>_N(x, t)   & = (\<1>_N(x, t))^2 - \s_N(t), \\
\<3>_N (x, t)  & = (\<1>_N(x, t))^3 - 3 \s_N(t) \cdot \<1>_N(x, t), 
\end{split}
\label{sto3}
\end{align}

\noi
and the second order process $\<30>_N$ by 
\begin{align}
\<30>_N    = \I(\<3>_N), 
\label{sto4}
\end{align}

\noi
where $\I$ denotes the Duhamel integral operator in \eqref{P3}.
Here, $\s_N(t)$ is defined by\footnote{In our spatially homogeneous setting,
the variance $\s_N(t)$ is independent of $x \in \T^3$.} 
\begin{align}
\begin{split}
\s_N(t) & = \E\big[ (\<1>_N(x, t))^2\big]
=  \sum_{|n|\leq N}
\int_0^t \bigg[\frac{\sin((t - t')\jb{n})}{\jb{n}^{1+\al}} \bigg]^2 dt'\\
& = \sum_{|n|\leq N} \bigg\{\frac{t}{2\jb{n}^{2+2\al}} - \frac{\sin(2t \jb{n})}{4\jb{n}^{3+2\al}}\bigg\}
\sim 
\begin{cases}
t \log N , & \text{for } \al = \frac 12, \\
t N^{1-2\al} , & \text{for } 0 < \al< \frac 12.
\end{cases}
\end{split}
\label{sigma1}
\end{align}

\noi
We point out  that 
a standard argument shows that 
$\<2>_N$ and 
$\<3>_N$
converge almost surely to $\<2>$ 
in $C([0,T];W^{2\al -1 - ,\infty}(\T^3))$ 
and to 
$\<3>$
in $C([0,T];W^{3\al -\frac 32 - ,\infty}(\T^3))$, respectively,
but that we do not need these regularity properties 
of the Wick powers $\<2>$ and $\<3>$ in this paper.

As for the second order process $\<30>_N$ in \eqref{sto4}, 
if we 
 proceed with a ``parabolic thinking'',\footnote{Namely, 
if we only take into account  the (uniformly bounded in $N$) regularity 
$3\al - \frac 32 -$ of  $\<3>_N$
and  one degree of smoothing
from the Duhamel integral operator $\I$ {\it without} taking into account the product structure
and the oscillatory nature of the linear wave propagator.}
then we expect that  $\<30>_N$ has
regularity\footnote{By ``regularity'', we mean 
the spatial regularity $s$ of $\<30>_N$
as an element in $C([0,T];W^{s,\infty}(\T^3))$,  
uniformly bounded in $N \in \N$.} 
$3\al - \frac 12 - = (3\al - \frac 32 -) + 1$, 
which is negative for $\al \le \frac 16$.
In the dispersive setting, however, 
we can exhibit 
 multilinear smoothing
by exploiting multilinear dispersion
coming from an interaction of 
 (random) waves.
 In fact, by adapting the argument in \cite{Bring}
to our current problem, 
we can show an extra $\sim \frac 12$-smoothing
for 
$\<30>_N$, uniformly in $N \in \N$, and
for the limit $\<30> = \I(\<3>) = \lim_{N \to \infty} \<30>_N$
and thus they 
have positive regularity. 
See Lemma \ref{LEM:sto1}.
As in~\cite{GKO2, Bring}, 
such multilinear smoothing plays a fundamental role
in our analysis.

Let us now start with the 
truncated renormalized SNLW \eqref{SNLW3}
and obtain the limiting formulation of our problem.
By proceeding with the second order expansion:
\begin{align}
u_N = \<1>_N - \<30>_N + v_N, 
\label{exp3}
\end{align}

\noi
we  rewrite
\eqref{SNLW3}
as
\begin{align}
\begin{split}
(\dt^2  +1 -  \Dl)  v_N 
&  =  - (v_N + \<1> _N- \<30>_N)^3 +3\s_N (\<1>_N - \<30>_N + v_N) + \<3>_N\\
& = -  v_N^3  + 3 (\<30>_N - \<1>_N)v_N^2 
- 3( \<30>_N^2 -2 \<30>_N\<1>_N) v_N
-3  \<2>_N v_N\\
& \quad + \<30>_N^3
  -3\<30>_N^2 \<1>_N 
  + 3\<30>_N\<2>_N, 
\end{split}
\label{SNLW9}
\end{align}

\noi
where we used \eqref{sto3}.
The main problem in studying
singular stochastic PDEs lies in making sense of various products.
In this formal discussion, 
let us apply the following  ``rules'':

\begin{itemize}

\item A product of functions of regularities $s_1$ and $s_2$
is defined if $s_1 + s_2 > 0$.
When $s_1 > 0$ and $s_1 \geq s_2$, the resulting product has regularity $s_2$.

\smallskip

\item A product of stochastic objects (not depending on the unknown)
is always well defined, possibly with a renormalization.
The product of stochastic objects of regularities $s_1$ and $s_2$
has regularity $\min( s_1, s_2, s_1 + s_2)$.

\end{itemize}

\smallskip

We postulate that the unknown $v$
has regularity $\frac 12+$,\footnote{As for the unknown $v$, 
we measure its regularity in (the local-in-time version of) the $X^{s, \frac 12+}$-norm.}
which is subcritical with respect to the standard scaling heuristics
for the three-dimensional cubic NLW.
In order to close the Picard iteration argument, 
we need all the terms on the right-hand side of \eqref{SNLW9}
to have regularity $-\frac 12+$.
With the aforementioned regularities
of the stochastic terms $\<1>_N$, $\<2>_N$, and $\<30>_N$
and applying the rules above, 
we can handle the products on the right-hand side of \eqref{SNLW9}, 
giving regularity $-\frac 12+$, 
{\it except} for the following terms (for small $\al >0$):
\begin{align}
 \<30>_N\<1>_N v_N, 
\qquad 
\<2>_Nv_N, 
\qquad \text{and} \qquad
 \<30>_N\<2>_N.
 \label{sto5}
\end{align}

\noi
As for the first term  $\<30>_N\<1>_N v_N$,  
we first use stochastic analysis to make sense of 
$\<30>_N\<1>_N$ 
with regularity $\al -\frac 12 -$, 
uniformly in $N \in \N$,  (see Lemma \ref{LEM:sto2})
and then 
interpret the product as
\[  \<30>_N\<1>_N v_N = 
( \<30>_N\<1>_N) v_N.
\]

\noi
Note that the right-hand side is well defined  since the sum of the regularities
is positive: $(\al -\frac 12 -) + (\frac 12 +) > 0$.
The last product 
 $\<30>_N\<2>_N$ in \eqref{sto5} makes sense but the resulting regularity is 
 $2\al - 1-$, smaller than  the required regularity $ -\frac 12+$,
 when $\al$ is close to $0$.
As for the second term in \eqref{sto5},
it  depends on the unknown $v_N$ and 
thus the product does not make sense (at this point)
since the sum of regularities is negative
(when $\al > 0$ is small).

As we see below, 
by studying the last two terms in \eqref{sto5} under the Duhamel integral operator $\I$,
we can indeed give a meaning to them
and exhibit extra $(\frac 12+)$-smoothing
with the resulting regularity $\frac 12+$ (under $\I$), 
which allows us to close the argument.  
By writing \eqref{SNLW9}
with initial data $(u_0, u_1)$
in the Duhamel formulation, we have
\begin{align}
\begin{split}
  v_N 
& = S(t) (u_0, u_1) + \I\big(-  v_N^3  + 3 (\<30>_N - \<1>_N)v_N^2 
- 3 \<30>_N^2 v_N\big) \\
& \quad
+6\I\big((\<30>_N\<1>_N) v_N\big)
-3 \If^{  \<2>_N }(v_N)\\
& \quad +
\I\big( \<30>_N^3
  -3\<30>_N^2 \<1>_N \big)
  + 3 \,  \<320>_N, 
\end{split}
\label{SNLW10}
\end{align}

\noi
where  $S(t) (u_0, u_1) = \cos (t\jb{\nb}) u_0 + 
\frac{\sin (t\jb{\nb})}{\jb{\nb}}u_1$ denotes the (deterministic) linear solution.
Here, 
$\If^{  \<2>_N }$ denotes the random operator defined by 
\begin{align}
 \If^{  \<2>_N }(v)
 = \I( \<2>_Nv)
\label{ran1}
\end{align}

\noi
and (as the notation suggests), the last term in \eqref{SNLW10} 
is defined by 
\begin{align}
 \<320>_N = \I( \<30>_N\<2>_N)
\label{sto6}
\end{align}

\noi
(without a renormalization).
By exploiting random multilinear dispersion, 
we show that 
\begin{itemize}
\item
 the random operator 
$\If^{  \<2>_N }$ maps functions of regularity $\frac 12+$
to those of regularity $\frac 12+$
(measured in the $X^{s, b}$-spaces)
with the operator norm uniformly bounded in $N \in \N$
and 
$\If^{  \<2>_N }$
converges to some limit, denoted by 
$\If^{  \<2>}$, as $N \to \infty$.
We study the random operator $\If^{\<2>_N}$ via the random matrix approach
\cite{BO96, BO97,Rich, DNY2, Bring}.\footnote{We also mention
a recent preprint \cite{Seong}, where the 
 random matrix approach is also used to 
 prove probabilistic local well-posedness
 of the Zakharov-Yukawa system on the two-dimensional torus $\T^2$.}
See Lemma \ref{LEM:ran1}.

\smallskip

\item
the third order process $ \<320>_N $
has regularity $\frac 12 +$
(measured in the $X^{s, b}$-spaces)
with the  norm uniformly bounded in $N \in \N$
and 
$ \<320>_N $
converges to some limit, denoted by~$ \<320>$, as $N \to \infty$.
See Lemma \ref{LEM:sto3}.

\end{itemize}

\noi
We deduce these claims
as corollaries
to Bringmann's
work \cite{Bring}.
In \cite{Bring}, 
the 
smoothing  coming from the  potential $V = \jb{\nb}^{-\be}$ in the 
Hartree nonlinearity $(V*u^2)u$ played an important role.
In our problem, this is replaced by the smoothing $\jb{\nb}^{-\al}$
on the noise
and we reduce our problem to that in \cite{Bring}, 
essentially 
by the following simple observation:
\begin{align}
\prod_{j = 1}^k \jb{n_j}^{-\g}
\les \jb{n_1+ \cdots + n_k}^{-\g}
\label{PP1}
\end{align}

\noi
for any $\g \ge 0$.

\begin{remark}\rm

In the following, we
also set 
\begin{align}
\Sep = \I(\<30>_N^2 \<1>_N).
\label{sto6a}
\end{align}

\noi
By carrying out analysis analogous to 
(but more involved than) 
that for $\<30>_N \<1>_N$
studied in Lemma \ref{LEM:sto2}
below, 
we can
show that 
$\{\<30>_N^2 \<1>_N\}_{N \in \N}$
forms a Cauchy sequence in 
$C([0,T];W^{\al - \frac 12 - ,\infty}(\T^3))$
 almost surely, 
 thus converging to some limit 
$\<30>^2\<1>$.
In this paper, however, 
we proceed with space-time analysis as in \cite{Bring}.
Namely, we study $\Sep$ in the $X^{s, b}$-spaces
and show that it converges to some limit denoted by $\<70>$.
See Lemma \ref{LEM:sto3}.

\end{remark}

Putting everything together, 
we can take $N \to \infty$ in \eqref{SNLW10}
and obtain the following limiting equation
for $v = u - \<1> + \<30>$:
\begin{align}
\begin{split}
  v
& = 
 S(t) (u_0, u_1) + 
\I\big(-  v^3  + 3 (\<30> - \<1>)v^2 
- 3 \<30>^2 v\big)\\ 
& \quad +6\I\big((\<30>\<1>) v\big)
-3 \If^{  \<2> }(v)\\
& \quad 
+ \I\big( \<30>^3\big)
  -3\, \<70>
  + 3 \,  \<320>.
\end{split}
\label{SNLW11}
\end{align}

\noi
By the Fourier restriction norm method with the Strichartz estimates, 
we can then prove
local well-posedness of \eqref{SNLW11}
in the deterministic manner.
Namely, given the following enhanced data set
\begin{align}
\Xi = \big (u_0, u_1, 
\<1>, \<30>, 
 \<30>\<1>, 
\<320>, \<70>, 
\If^{  \<2> }\big)
\label{data1}
\end{align}

\noi
of appropriate regularities
(depicted by stochastic analysis), 
there exists a unique local-in-time solution
$v$ to \eqref{SNLW11}, 
continuously depending on the enhanced data set $\Xi$.
See Proposition \ref{PROP:LWP}
for a precise statement.

This local well-posedness result
together with the convergence of 
$\<1>_N$ and  $\<30>_N$
 then yields
the convergence of 
$u_N = 
\<1>_N - \<30>_N + v_N$ in \eqref{exp3}
to the limiting process
\begin{align*}
u = \<1> - \<30> + v, 
\end{align*}

\noi
where $v$ is the solution to \eqref{SNLW11}.

\begin{remark}\label{REM:trilin}\rm
In terms of regularity counting, 
the sum of the regularities in $ \<1>\cdot v^2 $
is positive.
In the parabolic setting, one may then proceed with a product estimate.
In the current dispersive setting, however, 
integrability of functions plays an important role
and thus we need to proceed with care.
See Lemmas~\ref{LEM:trilin1} and~\ref{LEM:trilin2}.

\end{remark}

\begin{remark}\label{REM:uniq}\rm
(i) 
By the use of stochastic analysis, 
the stochastic terms
$\<1>$, $\<30>$, 
 $\<30>\<1>$, 
$\<320>$,
$\<70>$, 
and  
$\If^{  \<2> }$
in the enhanced data set 
are defined as the unique limits
of their truncated versions.
Furthermore, 
by deterministic analysis, we prove that 
a solution $v$ to \eqref{SNLW11}
is pathwise unique in an appropriate class. 
Therefore, under the decomposition  $u = \<1> - \<30> + v$, 
the uniqueness of~$u$ 
refers to 
(a) the uniqueness of $\<1>$ and $\<30>$ 
as the limits of $\<1>_N$ and $\<30>_N$
and (b) the uniqueness of $v$ as a solution to \eqref{SNLW11}.

\smallskip

\noi
(ii) In this paper, we work with the 
 frequency projector $\pi_N$
 with a sharp cutoff function on the frequency side.
It is also possible to work
with  smooth mollifiers $\eta_{\dl}(x) = \dl^{-3}\eta(\dl^{-1}x)$, 
where 
  $\eta \in C^\infty(\R^3 ; [0, 1])$ is a smooth, non-negative, even function with 
 $\int \eta dx = 1$ and $\supp \eta \subset (-\pi, \pi]^3\simeq \T^3$.
In this case, working with
\begin{align}
\begin{cases}
\dt^2 u_\dl + (1-  \Dl)  u_\dl  + u_\dl^3  - 3\s_\dl u_\dl  = \jb{\nb}^{-\al} \eta_\dl *  \xi\\
(u_\dl, \dt u_\dl)|_{t = 0} = (u_0, u_1), 
\end{cases}
\label{SNLW12}
\end{align}

\noi
we can show that a solution $u_\dl$ to \eqref{SNLW12} converges 
in probability to some limit  $u$ in $ C([-T_\o, T_\o]; H^{\al -\frac 12 -\eps} (\T^3))$ as $\dl \to 0$.
Furthermore, 
 the limit $u_\dl$ is independent of the choice of a mollification kernel~$\eta$
 and agrees with the limiting process $u$ constructed in Theorem \ref{THM:1}.
This is the second meaning of the uniqueness of the limiting process $u$.

\end{remark}

\begin{remark}\rm
(i) 
From the ``scaling'' point of view,  our problem for $0 < \al \ll 1$ is more difficult than the quadratic SNLW
\eqref{SNLW4} considered in \cite{GKO2}, 
where the paracontrolled calculus played an essential role.
On the other hand, 
for the proof of Theorem \ref{THM:1}, 
we do not need to use 
 the paracontrolled ansatz for the remainder terms $v  = u - \<1> + \<30>$
thanks to the smoothing on the noise
and the use of space-time estimates, 
which allows us to place $v$ in the subcritical regularity $\frac 12 + $.

Our approach to \eqref{SNLW2} and Bringmann's approach 
in \cite{Bring}
crucially exploit various multilinear smoothing, gaining  $\sim \frac 12$-derivative.
When $\al = 0$ (or $\be = 0$ in the Hartree SNLW~\eqref{SNLW5}), 
such multilinear smoothing seems to give (at best) $\frac 12$-smoothing
and thus the arguments in this paper and in \cite{Bring}
break down in the $\al = 0$ case.

\smallskip

\noi
(ii)  In \cite{GKO2}, 
Gubinelli, Koch, and the first author studied
the quadratic SNLW on $\T^3$ with an additive space-time white noise
(i.e.~$\al = 0$):
\begin{align}
\dt^2 u  + (1 -  \Dl)  u +  u^2  = \xi.
\label{q1}
\end{align}

\noi
With the Wick renormalization and the second order expansion
$u = \<1> - \<20> + v$, where $\<20> = \I(\<2>)$, 
the remainder term $v = u - \<1> + \<20>$
satisfies  
\begin{align}
(\dt^2  +1 -  \Dl)  v 
  =  - (v  - \<20>)^2  - 2 \<1> v + 2\<1> \<20>.
\label{q2}
\end{align}

\noi
As observed in \cite{GKO2}, 
the main issue in studying \eqref{q2} comes
from the regularity $\frac 12 - $
of $v$, which is inherited from the regularity $-\frac 12 - $ of  $\<1>\<20>$.
As a result, the product 
$\<1> v $ in \eqref{q2}  is not well defined since the sum of the regularities
of $\<1>$ and $v$ is negative.
As in 
\eqref{ran1}, 
it is tempting to directly 
define 
the random operator 
$\If^{\<1>} (v) = \I(\<1> v)$,  
using the random matrix estimates.
However, there is an issue 
in handling 
 the ``high $\times$ high $\to$ low'' interaction
 and thus the random matrix approach alone is not sufficient to close the argument.
In \cite{GKO2}, 
this issue was overcome by a paracontrolled ansatz
and an iteration of the Duhamel formulation.
We point out that 
the use of  the paracontrolled ansatz in \cite{GKO2}
led to the following paracontrolled operator
$\If_{\pl}(v) = \I(v\pl \<1>)$, 
which avoids the undesirable 
 high $\times$ high $\to$ low interaction.
Instead of the paracontrolled calculus, 
one may  use the random averaging operator
from~\cite{DNY1}
together with an iteration of the Duhamel formulation.
We, however, point out that 
due to the problematic
 high $\times$ high interaction, 
 the random averaging operator  as introduced in \cite{DNY1} alone (without iterating the Duhamel formulation)
 does not seem to be sufficient to study the quadratic SNLW \eqref{q1}.

\end{remark}

\medskip

\noi
$\bullet$ {\bf Organization of the paper.}
In Section \ref{SEC:2}, we go over the basic definitions and lemmas 
from deterministic and stochastic analysis.
In Section \ref{SEC:LWP}, we first 
state the almost sure regularity and convergence
properties
of (the truncated versions of) the  stochastic objects in the enhanced data set
$\Xi$ in \eqref{data1}.
Then, we present the proof of our main result (Theorem~\ref{THM:1}).
In Section~\ref{SEC:sto2}, 
we establish
 the almost sure regularity and convergence
properties
of the stochastic objects in the enhanced data set.
In Section \ref{SEC:A}, 
we recall the counting lemmas from~\cite{Bring}
which play a crucial role in Section \ref{SEC:sto2}.
In Sections \ref{SEC:B} and \ref{SEC:C}, 
we provide the basic definitions and lemmas
on multiple stochastic integrals 
and (random) tensors, respectively.

\section{Notations and basic lemmas}
\label{SEC:2}

We write $ A \les B $ to denote an estimate of the form $ A \leq CB $. 
Similarly, we write  $ A \sim B $ to denote $ A \les B $ and $ B \les A $ and use $ A \ll B $ 
when we have $A \leq c B$ for small $c > 0$.
We also use  $ a+ $ (and $ a- $) to mean  $ a + \eps $ (and $ a-\eps $, respectively)
 for arbitrarily small $ \eps >0 $.

When we work with space-time function spaces, we use short-hand notations such as
 $C_T H^s_x  = C([0, T]; H^s(\T^3))$.

When there is no confusion,
we simply use $\ft{u}$ or $\F(u)$
to denote
the spatial, temporal, or space-time Fourier transform
of $u$, depending on the context.
We also use $\F_x$, $\F_t$, and $\F_{x, t}$
to denote
the spatial, temporal, and space-time Fourier transforms, respectively.

We use the following short-hand notation:
$n_{ij} = n_i + n_j$, etc.
For example, $n_{123} = n_1 + n_2 + n_3$.

\subsection{Sobolev spaces 
 and Besov spaces}

Let $s \in \R$ and $1 \leq p \leq \infty$.
We define the $L^2$-based Sobolev space $H^s(\T^3)$
by the norm:
\begin{align*}
\| f \|_{H^s} = \| \jb{n}^s \ft f (n) \|_{\l^2_n}
\end{align*}

\noi
and set $\H^s(\T^3)$ to be 
\[\H^s(\T^3) = H^s(\T^3)\times H^{s-1}(\T^3).\]

\noi
We also define the $L^p$-based Sobolev space $W^{s, p}(\T^3)$
by the norm:
\begin{align*}
\| f \|_{W^{s, p}} = \big\| \F^{-1} (\jb{n}^s \ft f(n))\big\|_{L^p}.
\end{align*}

\noi
When $p = 2$, we have $H^s(\T^3) = W^{s, 2}(\T^3)$.

Let $\phi:\R \to [0, 1]$ be a smooth  bump function supported on $\big[-\frac{8}{5}, \frac{8}{5}\big]$ 
and $\phi\equiv 1$ on $\big[-\frac 54, \frac 54\big]$.
For $\xi \in \R^3$, we set $\phi_0(\xi) = \phi(|\xi|)$
and 
\[\phi_{j}(\xi) = \phi\big(\tfrac{|\xi|}{2^j}\big)-\phi\big(\tfrac{|\xi|}{2^{j-1}}\big)\]

\noi
for $j \in \N$.
Note that we have 
\begin{align}
\sum_{j \in \N_0} \phi_j(\xi) = 1
\label{phi1}
 \end{align}

\noi
for any $\xi \in \R^3$.
Then, for $j \in \N_0 := \N \cup\{0\}$, 
we define  the Littlewood-Paley projector  $\P_j$ 
as the Fourier multiplier operator with a symbol $\phi_j$.
Thanks to  \eqref{phi1}, 
we have 
\begin{align}
 f = \sum_{j = 0}^\infty \P_j f.
 \label{para1a}
\end{align}

Next, we recall the following paraproduct decomposition due to 
 Bony~\cite{Bony}.
See \cite{BCD, GIP} for further details.
Let $f$ and $g$ be  functions  on $\T^3$
of regularities $s_1$ and $s_2$, respectively.
Using~\eqref{para1a}, 
we write the product $fg$ as
\begin{align}
\begin{split}
fg 
& 
= f\pl g + f \pe g + f \pg g \\
 : \! & = \sum_{j < k-2} \P_{j} f \, \P_k g
+ \sum_{|j - k|  \leq 2} \P_{j} f\,  \P_k g
+ \sum_{k < j-2} \P_{j} f\,  \P_k g.
\end{split}
\label{para1}
\end{align}

\noi
The first term 
$f\pl g$ (and the third term $f\pg g$) is called the paraproduct of $g$ by $f$
(the paraproduct of $f$ by $g$, respectively)
and it is always well defined as a distribution
of regularity $\min(s_2, s_1+ s_2)$.
On the other hand, 
the resonant product $f \pe g$ is well defined in general 
only if $s_1 + s_2 > 0$.

We briefly recall the basic properties of the Besov spaces $B^s_{p, q}(\T^3)$
defined by the norm:
\begin{equation*}
\| u \|_{B^s_{p,q}} = \Big\| 2^{s j} \| \P_{j} u \|_{L^p_x} \Big\|_{\l^q_j(\N_0)}.
\end{equation*}

\noi
Note that  $H^s(\T^3) = B^s_{2,2}(\T^3)$.

\begin{lemma}\label{LEM:para}
\textup{(i) (paraproduct and resonant product estimates)}
Let $s_1, s_2 \in \R$ and $1 \leq p, p_1, p_2, q \leq \infty$ such that 
$\frac{1}{p} = \frac 1{p_1} + \frac 1{p_2}$.
Then, we have 
\begin{align}
\| f\pl g \|_{B^{s_2}_{p, q}} \les 
\|f \|_{L^{p_1}} 
\|  g \|_{B^{s_2}_{p_2, q}}.  
\label{para2a}
\end{align}

\noi
When $s_1 < 0$, we have
\begin{align}
\| f\pl g \|_{B^{s_1 + s_2}_{p, q}} \les 
\|f \|_{B^{s_1 }_{p_1, q}} 
\|  g \|_{B^{s_2}_{p_2, q}}.  
\label{para2}
\end{align}

\noi
When $s_1 + s_2 > 0$, we have
\begin{align}
\| f\pe g \|_{B^{s_1 + s_2}_{p, q}} \les 
\|f \|_{B^{s_1 }_{p_1, q}} 
\|  g \|_{B^{s_2}_{p_2, q}}  .
\label{para3}
\end{align}

\noi
\textup{(ii)}
Let $s_1 <  s_2 $ and $1\leq p, q \leq \infty$.
Then, we have 
\begin{align} 
\| u \|_{B^{s_1}_{p,q}} 
&\les \| u \|_{W^{s_2, p}}.
\label{embed}
\end{align}

\end{lemma}

The product estimates \eqref{para2a},  \eqref{para2},  and \eqref{para3}
follow easily from the definition \eqref{para1} of the paraproduct 
and the resonant product.
See \cite{BCD, MW2} for details of the proofs in the non-periodic case
(which can be easily extended to the current periodic setting).
The embedding \eqref{embed}
 follows from the $\l^{q}$-summability 
of $\big\{2^{(s_1 - s_2)j}\big\}_{j \in \N_0}$ for $s_1 < s_2$
and the uniform boundedness of the Littlewood-Paley projector $\P_j$.

We also recall the following product estimate from \cite{GKO}.

\begin{lemma}\label{LEM:gko}
Let $0\leq s\leq 1$.
%
%
%
%
%
Let   $1<p,q,r<\infty$ such that $s \geq   3\big(\frac{1}{p}+\frac{1}{q}-\frac{1}{r}\big)$.
Then, we have
$$
\|\jb{\nb}^{-s}(fg)\|_{L^r(\T^3)}
\lesssim\| \jb{\nb}^{-s} f\|_{L^{p}(\T^3)} \| \jb{\nb}^{s} g\|_{L^{q}(\T^3)} .
$$
\end{lemma}

Note that
while  Lemma \ref{LEM:gko} 
was shown only for 
$s =   3\big(\frac{1}{p}+\frac{1}{q}-\frac{1}{r}\big)$
in \cite{GKO}, 
the general case
$s \geq    3\big(\frac{1}{p}+\frac{1}{q}-\frac{1}{r}\big)$
follows the embedding $L^{r_1}(\T^3) \subset 
L^{r_2}(\T^3)$, $r_1 \geq r_2$.

\subsection{Fourier restriction norm method
and 
Strichartz estimates}

We first  recall  the so-called  $X^{s, b}$-spaces, 
also known as the hyperbolic Sobolev spaces,
due to Klainerman-Machedon~\cite{KM} and Bourgain~\cite{BO93}, 
defined by the norm:
\begin{align}
\|u\|_{X^{s, b} ( \T^3 \times\R)} = \|\jb{n}^s \jb{|\tau|- \jb{n}}^b \ft u(n, \tau)\|_{\l^2_n L^2_\tau( \Z^3\times \R)}.
\label{Xsb}
\end{align}

\noi
For $ b > \frac{1}{2}$, we have $X^{s, b} \subset C(\R; H^s(\T^3))$.
Given an  interval $I \subset \R$,
we define the local-in-time version $X^{s, b}(I)$
as a restriction norm:
\begin{align}
 \|u \|_{X^{s, b}(I)} = \inf\big\{ \|v\|_{X^{s, b}(\T^3 \times \R)}: \, v|_I = u\big\}.
 \label{Xsb2}
\end{align}

\noi
When $I = [0, T]$, we set $X^{s, b}_T = X^{s, b}(I)$.
%
%

Next, we recall  the Strichartz estimates
for the  linear wave/Klein-Gordon  equation. 
Given  $0 \leq s \leq 1$, 
we say that a pair $(q, r)$ is $s$-admissible
if $2 < q \leq \infty$, 
 $2 \leq r < \infty$, 
\begin{align*}
\frac{1}{q} + \frac 3r  =  \frac 32 - s
\qquad \text{and} \qquad 
\frac 1q + \frac{1}{r} \leq \frac 1 2.
\end{align*}

\noi
Then, we have the following Strichartz estimates.

\begin{lemma}\label{LEM:Str1}
 Given $0 \leq s \leq 1$,
let $(q, r)$ be $s$-admissible.
Then, we have 
\begin{align}
\|S(t)(\phi_0, \phi_1)  \|_{L^q_TL^r_x(\T^3)}
\les
\|(\phi_0 ,\phi_1) \|_{\H^{s}(\T^3)}
\label{Str1}
\end{align}

\noi
for any 
 $0 < T \leq 1$.
\end{lemma}

See Ginibre-Velo \cite{GV}, Lindblad-Sogge \cite{LS},  and Keel-Tao \cite{KeelTao}
for the Strichartz estimates  on $\R^d$.
See also \cite{KSV}. 
The Strichartz estimates \eqref{Str1} on $\T^3$ in Lemma \ref{LEM:Str1} follows from 
those on $\R^3$ and the finite speed of propagation.

When $b > \frac 12$, the $X^{s, b}$-spaces enjoy the transference principle.
In particular, 
as a corollary to Lemma \ref{LEM:Str1}, 
we obtain the following space-time estimate.
See \cite{KS, TAO} for the proof.

\begin{lemma}\label{LEM:Str2}
Let $0 < T \leq 1$. 
Given $0 \leq s \leq 1$,
let $(q, r)$ be $s$-admissible.
Then, for $b > \frac{1}{2}$, we have 
\begin{align*}
\| u \|_{L^q_T L^r_x} \les
\| u\|_{X_T^{s, b}}.
\end{align*}

\end{lemma}

We also state the nonhomogeneous linear  estimate.
See  \cite{GTV}.

\begin{lemma}\label{LEM:lin1}
Let $ - \frac 12 < b' \leq 0 \leq b \leq b'+1$.
Then, for $0 < T \leq 1$, we have 
\begin{align*}
\| \I(F)\|_{X^{s, b}_T}= 
\bigg\|  \int_0^t \frac{\sin((t-t')\jb{\nb})}{\jb{\nb}} F(t') dt'\bigg\|_{X^{s, b}_T}
\les T^{1-b+b'} \|F\|_{X^{s-1, b'}_T}.
\end{align*}

\end{lemma}

In the following, we briefly go over the main trilinear estimate
for the basic 
 local well-posedness of the cubic NLW \eqref{NLW0} in $\H^{\frac 12+\eps}(\T^3)$.

\begin{lemma}\label{LEM:tri}
Fix small $\dl_1, \dl_2 > 0$ with $4\dl_2 \leq \dl_1$.
Then, we have
\begin{align}
\| \I(u_1u_2u_3)\|_{X^{\frac 12 + \dl_1 , \frac 12 + \dl_2}_T}
\les
T^{\dl_2}\prod_{j = 1}^3 \| u_j\|_{X_T^{\frac 12 +\dl_1, \frac 12+\dl_2}}
\label{L0}
\end{align}

\noi
for any $0 < T\leq 1$.
\end{lemma}

\begin{proof}

Recall that $(q, r) = (4, 4)$ is $\frac 12$-admissible.
Then, in view of Lemma \ref{LEM:Str2}, 
interpolating 
\begin{align}
\| u \|_{L^4_{T, x}} \les
\| u\|_{X_T^{\frac 12 , \frac 12+\dl_0}}
\qquad \text{and}
\qquad 
\| u \|_{L^2_{T, x}} = 
\| u\|_{X_T^{0, 0}}
\label{L0a}
\end{align}

\noi
with small $\dl_0 > 0$, 
we obtain 
\begin{align}
\| u \|_{L^\frac{4}{1+2\dl_1}_{T, x}} \les
\| u\|_{X_T^{\frac 12 -\dl_1, \frac 12-\frac 12\dl_1}}.
\label{L1}
\end{align}

\noi
Moreover, noting that $\big(\frac{12}{3-2\dl_1},\frac{12}{3-2\dl_1}\big)$
is $\big(\frac 12 + \frac 23 \dl_1\big)$-admissible,  
we obtain from  Lemma \ref{LEM:Str2} that 
\begin{align}
\| u \|_{L^\frac{12}{3-2\dl_1}_{T, x}} \le C_{\dl_1, \dl_2}
\| u\|_{X_T^{\frac 12 +\frac 23 \dl_1, \frac 12+\dl_2}}
\label{L2}
\end{align}

\noi
for any $\dl_2 > 0$.

Hence, 
from Lemma \ref{LEM:lin1}, duality, H\"older's inequality, \eqref{L1},
 and \eqref{L2}, we obtain
\begin{align*}
\| \I(u_1u_2u_3)\|_{X^{\frac 12 + \dl_1 , \frac 12 + \dl_2}_T}
& \les T^{\dl_2} 
\| u_1u_2u_3\|_{X^{-\frac 12 + \dl_1, - \frac 12 + 2\dl_2}}\\
& = T^{\dl_2} \sup_{\|w\|_{X^{\frac 12 - \dl_1,  \frac 12 - 2\dl_2}}= 1}
\bigg|\int_0^T \int_{\T^3} u_1u_2u_3 w dx dt\bigg|\\
& \leq T^{\dl_2} \sup_{\|w\|_{X^{\frac 12 - \dl_1,  \frac 12 - 2\dl_2}}= 1}
 \bigg( \prod_{j = 1}^3 \| u_j \|_{L^\frac{12}{3-2\dl_1}_{T, x}}\bigg)
\| w \|_{L^\frac{4}{1+2\dl_1}_{T, x}} \\
& \les
T^{\dl_2} \prod_{j = 1}^3 \| u_j\|_{X_T^{\frac 12 +\frac 23 \dl_1, \frac 12+\dl_2}}, 
\end{align*}

\noi
provided that $0 < 4\dl_2 \leq \dl_1 \ll 1$.
This proves \eqref{L0}.
\end{proof}

We conclude this part by establishing the following trilinear estimate, 
which will be used to control the term $\<1> \, v^2$ in \eqref{SNLW11}.
See  Proposition 8.6 in \cite{Bring}
for an analogous trilinear estimate.

\begin{lemma}\label{LEM:trilin1}
Let $\dl_1, \dl_2 > 0$ be sufficiently small such that $8 \dl_2 \leq \dl_1$.  Then, we have
\begin{align}
\|u_1  u_2  u_3 \|_{X^{-\frac12 +\dl_1, -\frac{1}{2}+2\dl_2}_T}
\les
\|u_1\|_{L^\infty_T W^{  - \frac 12 + 2\dl_1, \infty}_x}
\| u_2\|_{X^{\frac 12 + \dl_1, \frac 12 + \dl_2 }_T}
\|u_3 \|_{X^{\frac 12 + \dl_1 , \frac 12 + \dl_2}_T}
\label{L4}
\end{align}

\noi
for any $0 < T \le 1$.

\end{lemma}

\begin{proof}
By applying the Littlewood-Paley decompositions, 
we have
\begin{align*}
& \text{LHS of }
 \eqref{L4} \\
& \quad \leq 
\sum_{j_1, j_{23}, j_{123} = 0}^\infty 
\big\|\P_{j_{123}}\big(\P_{j_1} u_1  \P_{j_{23}}(u_2  u_3)\big) \big\|_{X^{-\frac12 +\dl_1, -\frac{1}{2}+2\dl_2}_T}.
\end{align*}

\noi
For simplicity of notation, we set $N_1 = 2^{j_1}$, 
$N_{23} = 2^{j_{23}}$, and $N_{123} = 2^{j_{123}}$,
denoting the dyadic frequency sizes 
of $n_1$ (for $u_1$), $n_{23}$ (for $u_2u_3$), 
and $n_{123}$ (for $u_1 u_2 u_3$), respectively.
We set  $v_k = \P_{j_k} u_k$.
In view of $n_{123} = n_1 + n_{23}$, we separately estimate
the contributions from 
(i) ~$N_{123} \sim \max(N_1, N_{23})$
and 
(ii) $N_{123} \ll \max(N_1, N_{23})$.

\smallskip

\noi
$\bullet$ {\bf Case 1:}  $N_{123} \sim \max(N_1, N_{23})$.
\\
\indent
By 
H\"older's inequality
and the $L^4$-Strichartz estimate~\eqref{L0a}, we have 
\begin{align*}
\big\|\P_{j_{123}} \big(v_1 \P_{j_{23}} & (u_2  u_3)\big) \big\|_{X^{-\frac12 +\dl_1, -\frac{1}{2}+2\dl_2}_T}
 \les
N_{123}^{-\frac 12 + \dl_1}
\|v_1  \P_{j_{23}}(u_2  u_3) \|_{L^2_{T, x}}\\
& \les
N_{123}^{- \dl_1} 
\|u_1 \|_{L^\infty_T W^{-\frac 12 + 2\dl_1, \infty}_x}
\prod_{j= 2}^3\| u_j  \|_{L^4_{T, x}}\\
& \les 
N_{123}^{- \dl_1} 
\|u_1\|_{L^\infty_T W^{  - \frac 12 +2 \dl_1, \infty}_x}
\prod_{j = 2}^3\| u_j\|_{X^{\frac 12 , \frac 12 + \dl_2}_T}.
\end{align*}

\noi
This is  summable in 
dyadic $N_1,  N_{23}, N_{123}\ge 1$, 
yielding \eqref{L4} in this case.

\smallskip

\noi
$\bullet$ {\bf Case 2:}  $N_{123} \ll \max(N_1, N_{23})$.
\\
\indent
In this case, we further apply the 
Littlewood-Paley decompositions for $u_2$ and $u_3$
and write 
\[ u_2 u_3 = \sum_{j_2, j_3= 0}^\infty (\P_{j_2} u_2)(\P_{j_3} u_3).\]
Without loss of generality, assume $N_3  \ge N_2$, where $N_k  = 2^{j_k}$, $k = 2, 3$.
Then, we have 
\begin{align}
N_{123} \les N_1 \sim N_{23} \les N_3.
\label{L6a}
\end{align}

By duality and  \eqref{L1} (with $\dl_1 = 4\dl_2$), we have 
\begin{align}
\begin{split}
\| \P_j u \|_{X^{0, - \frac 12 + 2\dl_2}_T}
& = \sup_{\|  v \|_{X^{0,  \frac 12 - 2\dl_2}}= 1}\bigg|
\int_0^T\int_{\T^3} (\P_j u) \big((\P_{j-1} + \P_j + \P_{j+1})v\big) dx dt 
\bigg|\\
& \les 2^{(\frac 12 - 4\dl_2)j} \|\P_j u \|_{L^\frac{4}{3-8 \dl_2}_{T, x}}.
\end{split}
\label{L7}
\end{align}

\noi
Then,  
from   \eqref{L7}, 
\eqref{L2}, and 
\eqref{L6a} with $8 \dl_2 \leq \dl_1$, we have
\begin{align*}
\big\|\P_{j_{123}} &  \big(v_1 \P_{j_{23}}  (v_2  v_3)\big) \big\|_{X^{-\frac12 +\dl_1, -\frac{1}{2}+2\dl_2}_T}
 \les
N_{123}^{-\frac 12 + \dl_1}
 N_{123}^{\frac 12 - 4\dl_2}
\|v_1 \P_{j_{23}}  (v_2  v_3) \|_{L^\frac{4}{3-8 \dl_2}_{T, x}}\\
& \les
N_{123}^{\dl_1 - 4\dl_2} N_1^{\frac 12 - 2\dl_1}
 \|v_1 \|_{L^\infty_T W^{-\frac 12 + 2 \dl_1, \infty}_x}
\| v_2 \|_{L^\frac{4}{1 - 8 \dl_2}_{T, x}}
\|v_3 \|_{L^2_{T, x}}\\
& \les
N_{123}^{ \dl_1 - 4\dl_2} N_1^{\frac 12 - 2\dl_1} N_3^{-\frac 12 - \dl_1}\\
& \quad 
\times  \|u_1 \|_{L^\infty_T W^{-\frac 12 + 2 \dl_1, \infty}_x}
\|u_2 \|_{X^{\frac 12 + 8 \dl_2 , \frac 12 + \dl_2}_T}
\|u_3 \|_{X^{\frac 12 + \dl_1 , 0}_T}\\
& \les
N_{123}^{ - 4\dl_2}
N_1^{- \dl_1}
 N_3^{ - \dl_1}
\|u_1 \|_{L^\infty_T W^{-\frac 12 + 2 \dl_1, \infty}_x}
\prod_{j = 2}^3 \|u_j \|_{X^{\frac 12 + \dl_1 , \frac 12 + \dl_2}_T}.
\end{align*}

\noi
This is summable 
in 
dyadic $N_1,  N_2, N_3, N_{23}, N_{123}\ge 1$, 
yielding \eqref{L4} in this case.
\end{proof}

\subsection{On discrete convolutions}

Next, we recall the following basic lemma on a discrete convolution.

\begin{lemma}\label{LEM:SUM}
\textup{(i)}
Let $d \geq 1$ and $\al, \be \in \R$ satisfy
\[ \al+ \be > d  \qquad \text{and}\qquad \al, \be < d.\]
\noi
Then, we have
\[
 \sum_{n = n_1 + n_2} \frac{1}{\jb{n_1}^\al \jb{n_2}^\be}
\les \jb{n}^{d - \al - \be}\]

\noi
for any $n \in \Z^d$.

\smallskip

\noi
\textup{(ii)}
Let $d \geq 1$ and $\al, \be \in \R$ satisfy $\al+ \be > d$.
\noi
Then, we have
\[
 \sum_{\substack{n = n_1 + n_2\\|n_1|\sim|n_2|}} \frac{1}{\jb{n_1}^\al \jb{n_2}^\be}
\les \jb{n}^{d - \al - \be}\]

\noi
for any $n \in \Z^d$.

\end{lemma}

Namely, in the resonant case (ii), we do not have the restriction $\al, \be < d$.
Lemma \ref{LEM:SUM} follows
from elementary  computations.
See, for example,  Lemmas 4.1 and 4.2 in \cite{MWX} for the proof.

\subsection{Tools from stochastic analysis}

We conclude this section by recalling useful lemmas
from stochastic analysis.
See \cite{Bog, Shige, Nua} for basic definitions.
See also Appendix \ref{SEC:B} for basic definitions
and properties
for multiple stochastic integrals.

Let $(H, B, \mu)$ be an abstract Wiener space.
Namely, $\mu$ is a Gaussian measure on a separable Banach space $B$
with $H \subset B$ as its Cameron-Martin space.
Given  a complete orthonormal system $\{e_j \}_{ j \in \N} \subset B^*$ of $H^* = H$, 
we  define a polynomial chaos of order
$k$ to be an element of the form $\prod_{j = 1}^\infty H_{k_j}(\jb{x, e_j})$, 
where $x \in B$, $k_j \ne 0$ for only finitely many $j$'s, $k= \sum_{j = 1}^\infty k_j$, 
$H_{k_j}$ is the Hermite polynomial of degree $k_j$, 
and $\jb{\cdot, \cdot} = \vphantom{|}_B \jb{\cdot, \cdot}_{B^*}$ denotes the $B$--$B^*$ duality pairing.
We then 
denote the closure  of 
polynomial chaoses of order $k$ 
under $L^2(B, \mu)$ by $\mathcal{H}_k$.
The elements in $\H_k$ 
are called homogeneous Wiener chaoses of order $k$.
We also set
\[ \H_{\leq k} = \bigoplus_{j = 0}^k \H_j\]

\noi
 for $k \in \N$.

Let $L = \Dl -x \cdot \nabla$ be 
 the Ornstein-Uhlenbeck operator.\footnote{For simplicity, 
 we write the definition of the Ornstein-Uhlenbeck operator $L$
 when $B = \R^d$.}
Then, 
it is known that 
any element in $\mathcal H_k$ 
is an eigenfunction of $L$ with eigenvalue $-k$.
Then, as a consequence
of the  hypercontractivity of the Ornstein-Uhlenbeck
semigroup $U(t) = e^{tL}$ due to Nelson \cite{Nelson2}, 
we have the following Wiener chaos estimate
\cite[Theorem~I.22]{Simon}.
See also \cite[Proposition~2.4]{TTz}.

\begin{lemma}\label{LEM:hyp}
Let $k \in \N$.
Then, we have
\begin{equation*}
\|X \|_{L^p(\O)} \leq (p-1)^\frac{k}{2} \|X\|_{L^2(\O)}
 \end{equation*}
 
 \noi
 for any $p \geq 2$
 and any $X \in \H_{\leq k}$.

\end{lemma}

The following lemma will be used in studying regularities of stochastic objects.
We say that a stochastic process $X:\R_+ \to \mathcal{D}'(\T^d)$
is spatially homogeneous  if  $\{X(\cdot, t)\}_{t\in \R_+}$
and $\{X(x_0 +\cdot\,, t)\}_{t\in \R_+}$ have the same law for any $x_0 \in \T^d$.
Given $h \in \R$, we define the difference operator $\dl_h$ by setting
\begin{align*}
\dl_h X(t) = X(t+h) - X(t).
\end{align*}

\begin{lemma}\label{LEM:reg}
Let $\{ X_N \}_{N \in \N}$ and $X$ be spatially homogeneous stochastic processes
$:\R_+ \to \mathcal{D}'(\T^d)$.
Suppose that there exists $k \in \N$ such that 
  $X_N(t)$ and $X(t)$ belong to $\H_{\leq k}$ for each $t \in \R_+$.

\smallskip
\noi\textup{(i)}
Let $t \in \R_+$.
If there exists $s_0 \in \R$ such that 
\begin{align}
\E\big[ |\ft X(n, t)|^2\big]\les \jb{n}^{ - d - 2s_0}
\label{reg1}
\end{align}

\noi
for any $n \in \Z^d$, then  
we have
$X(t) \in W^{s, \infty}(\T^d)$, $s < s_0$, 
almost surely.

\smallskip
\noi\textup{(ii)}
Suppose that $X_N$, $N \in \N$, satisfies \eqref{reg1}.
Furthermore, if there exists $\g > 0$ such that 
\begin{align*}
\E\big[ |\ft X_N(n, t) - \ft X_M(n, t)|^2\big]\les N^{-\g} \jb{n}^{ - d - 2s_0}
\end{align*}

\noi
for any $n \in \Z^d$ and $M \geq N \geq 1$, 
then 
$X_N(t)$ is a Cauchy sequence in $W^{s, \infty}(\T^d)$, $s < s_0$, 
almost surely, 
thus converging to some limit 
 in  $W^{s, \infty}(\T^d)$.

\smallskip
\noi\textup{(iii)}
Let $T > 0$ and suppose that \textup{(i)} holds on $[0, T]$.
If there exists $\s \in (0, 1)$ such that 
\begin{align*}
 \E\big[ |\dl_h \ft X(n, t)|^2\big]
 \les \jb{n}^{ - d - 2s_0+ \s}
|h|^\s
\end{align*}

\noi
for any  $n \in \Z^d$, $t \in [0, T]$, and $h \in [-1, 1]$,\footnote{We impose $h \geq - t$ such that $t + h \geq 0$.}
then we have 
$X \in C([0, T]; W^{s, \infty}(\T^d))$, 
$s < s_0 - \frac \s2$,  almost surely.

\smallskip
\noi\textup{(iv)}
Let $T > 0$ and suppose that \textup{(ii)} holds on $[0, T]$.
Furthermore, 
if there exists $\g > 0$ such that 
\begin{align*}
 \E\big[ |\dl_h \ft X_N(n, t) - \dl_h \ft X_M(n, t)|^2\big]
 \les N^{-\g}\jb{n}^{ - d - 2s_0+ \s}
|h|^\s
\end{align*}

\noi
for any  $n \in \Z^d$, $t \in [0, T]$,  $h \in [-1, 1]$, and $M\ge N \geq 1$, 
then 
$X_N$ is a Cauchy sequence in  $C([0, T]; W^{s, \infty}(\T^d))$, $s < s_0 - \frac{\s}{2}$,
almost surely, 
thus converging to some process in 
$C([0, T]; W^{s, \infty}(\T^d))$.

\end{lemma}

Lemma \ref{LEM:reg} follows
from a straightforward application of the Wiener chaos estimate
(Lemma~\ref{LEM:hyp}).
For the proof, see Proposition 3.6 in \cite{MWX}
and  Appendix in \cite{OOTz}.
As compared to  Proposition~3.6 in \cite{MWX}, 
we made small adjustments.
In studying the time regularity, we 
made the following modifications:
$\jb{n}^{ - d - 2s_0+ 2\s}\mapsto\jb{n}^{ - d - 2s_0+ \s}$
and $s < s_0 - \s \mapsto s < s_0 - \frac \s2$  
so that it is suitable
for studying  the wave equation.
Moreover, while the result in \cite{MWX} is stated in terms of the
Besov-H\"older space $\mathcal{C}^s(\T^d) = B^s_{\infty, \infty}(\T^d)$, 
Lemma \ref{LEM:reg} handles the $L^\infty$-based Sobolev space $W^{s, \infty}(\T^3)$.
Note that 
the required modification of the proof is straightforward
since $W^{s, \infty}(\T^d)$ and $B^s_{\infty, \infty}(\T^d)$
differ only logarithmically:
\begin{align}
\| f \|_{W^{s, \infty}} \leq 
\sum_{j = 0}^\infty
 \|\P_j f \|_{W^{s, \infty}}
 \les \| f\|_{B^{s+\eps}_{\infty, \infty}}
\label{Bes1}
\end{align}

\noi
for any $\eps > 0$.
For the proof of the almost sure convergence claims, 
see \cite{OOTz}.

%
%
%
%
%
%
%
%
%
%

\section{Local well-posedness of SNLW, $\al > 0$}
\label{SEC:LWP}

In this section, we present the proof of local well-posedness 
of \eqref{SNLW11} (Theorem \ref{THM:1}).
In Subsection \ref{SUBSEC:sto1}, 
we first state the regularity and convergence properties
of the stochastic objects 
in the enhanced data set $\Xi$ in \eqref{data1}.
In Subsection \ref{SUBSEC:LWP2}, we then present a deterministic local well-posedness
result by viewing elements in the enhanced data
set as given (deterministic) distributions and a
given (deterministic)  operator
with prescribed regularity properties.

\subsection{On the stochastic terms}
\label{SUBSEC:sto1}

In this subsection,  
we state the regularity and convergence properties
of the stochastic objects 
in \eqref{data1} whose proofs are presented in Section \ref{SEC:sto2}.

\begin{lemma}\label{LEM:sto1}

Let  $\al  >0$ and $T > 0$.

\smallskip

\noi
\textup{(i)}
For any $s< \al - \frac12 $,
$\{ \<1>_N \}_{N \in \N}$ defined in \eqref{sto2} is
 a Cauchy sequence
in 
$C([0,T];W^{s,\infty}(\T^3))$,  almost surely.
In particular,
denoting the limit by $\<1>$ \textup{(}formally given by \eqref{sto1}\textup{)},
we have
  \[\<1> \in C([0,T];W^{\al - \frac12 - \eps,\infty}(\T^3))
  \]
  
  \noi 
for any $\eps >0$, almost surely.

\smallskip

\noi
\textup{(ii)} 
Let $0 < \al \leq \frac 12$.
Then, for any $s< \al $,
$\{ \<30>_N \}_{N \in \N}$ defined in \eqref{sto4} is
 a Cauchy sequence
in 
$C([0,T];W^{s,\infty}(\T^3))$ almost surely.
In particular,
denoting the limit by $\<30>$,
we have
  \[ \<30> \in C([0,T];W^{\al - \eps,\infty}(\T^3))
  \]
  
  \noi 
for any $\eps>0$, almost surely.

\end{lemma}

\begin{remark}\rm
(i) As mentioned in Section \ref{SEC:1}, 
a parabolic thinking gives regularity $3\al - \frac 12-$
for~$\<30>$.   
Lemma \ref{LEM:sto1}\,(ii) states that, when $\al > 0$ is small,  
we indeed gain about $\frac 12$-regularity by 
exploiting multilinear dispersion
as in the quadratic case studied in \cite{GKO2}.
We point out that our proof is based on an adaptation 
of Bringmann's analysis on the corresponding term
in the Hartree case \cite{Bring}
and thus the regularities we obtain 
in Lemma \ref{LEM:sto1}\,(ii)
as well as 
Lemmas~\ref{LEM:sto2}, \ref{LEM:sto3}, and \ref{LEM:ran1}
may not be sharp (especially for large $\al>0$; see, for example, a crude bound~\eqref{X7a}).
They are, however, sufficient for our purpose.

\smallskip

\noi
(ii) In this section, we only state almost sure convergence
but the same argument also yields 
 convergence in $L^p(\O)$
 with an exponential tail estimate (as in \cite{GKOT, OOTol, Bring}).
Our goal is, however, to prove local well-posedness
and thus the almost sure convergence suffices for our purpose.

\smallskip

\noi

\end{remark}

\begin{lemma}\label{LEM:sto2}

Let $0 < \al \leq \frac 12$ and $T >0$. 
Let  $\{ \<1>_N \}_{N \in \N}$ and $\{ \<30>_N \}_{N \in \N}$ 
be as in \eqref{sto2} and~\eqref{sto4}.
Then, for any $s< \al - \frac12 $,
$\big\{  \<30>_N\<1>_N \big\}_{N \in \N}$ is
 a Cauchy sequence
in 
$C([0,T];W^{s,\infty}(\T^3))$ almost surely.
In particular,
denoting the limit by $\<30>\<1>$,
we have
  \[ \<30>\<1> \in C([0,T];W^{\al - \frac12 - \eps,\infty}(\T^3))
  \]
  
  \noi 
for any $\eps >0$, almost surely.

%
%

\end{lemma}

\begin{lemma}\label{LEM:sto3}

Let $\al > 0$, $T > 0$, 
and $b > \frac 12$ be sufficiently close to $\frac 12$.

\smallskip

\noi
\textup{(i)}   
For any $s < \al + \frac 12$,  
$\big\{ \<320>_N \big\}_{N \in \N}$ defined in \eqref{sto6} is
 a Cauchy sequence
in 
$X^{s, b}([0, T])$.
In particular,
denoting the limit by $\<320>$,
we have
  \[ \<320> \in X^{\al + \frac 12 -\eps, b}([0, T]), 
  \]
  
  \noi 
  for any $\eps > 0$, 
almost surely.

\smallskip

\noi
\textup{(ii)}   
For any $s < \al + \frac 12$, 
$\big\{ \<70>_{\hspace{-1.7mm} N} \big\}_{N \in \N}$ defined in \eqref{sto6a} is
 a Cauchy sequence
in 
$X^{s, b}([0, T])$.
In particular,
denoting the limit by $\<70>$,
we have
  \[ \<70> \in X^{\al + \frac 12 -\eps, b}([0, T]), 
  \]
  
  \noi 
  for any $\eps > 0$, almost surely.

\end{lemma}

Given Banach spaces $B_1$ and $B_2$, 
we use $\L(B_1; B_2)$ to denote the space
of bounded linear operators from $B_1$ to $B_2$.
We also set
\begin{align}
 \L^{s_1, s_2, b}_{T_0}
=  \bigcap_{0 < T < T_0}
\L\big(X^{s_1, b}([0, T]); X^{s_2, b}([0, T])\big)
\label{Op1}
\end{align}

\noi
endowed 
with the norm given by 
\begin{align}
\| S\|_{\L^{s_1, s_2, b}_{T_0}}
= \sup_{0 < T < T_0} T^{-\ta} \|S\|_{\L(X^{s_1, b}_T; X^{s_2, b}_T)}
\label{Op2}
\end{align}

\noi
for some small $\ta >0$.

\begin{lemma}\label{LEM:ran1}

 Let $\al > 0$ and $T_0>0$. 
 Then, given sufficiently  small  $\dl _1, \dl_2 > 0$, 
 the sequence of 
 the random operators $ \{ \If^{\<2>_N} \}_{N \in \N}$ defined in \eqref{ran1} is a Cauchy sequence 
 in the class $ \L^{\frac 12  + \dl_1, \frac 12 +\dl_1,  \frac 12 +\dl_2}_{T_0}$, 
 almost surely.
 In particular, denoting the limit by $\If^{\<2>}$, we have 
 \[\If^{\<2>} \in 
 \L^{\frac 12  + \dl_1, \frac 12 +\dl_1,  \frac 12 +\dl_2}_{T_0},\] 
 
 \noi
 almost surely.

\end{lemma}

The following trilinear estimate is an immediate consequence of Lemma \ref{LEM:trilin1}.

\begin{lemma}\label{LEM:trilin2}

Let $\al > 0$. 
Let $\dl_1, \dl_2, \eps  > 0$ be sufficiently small
such that $2\dl_1 + \eps \le \al $.  Then, we have
\begin{align*}
\|\<1>  \, v_1  v_2  \|_{X^{-\frac12 +\dl_1, -\frac{1}{2}+2\dl_2}_T}
\les
\|\<1>\|_{L^\infty_T W^{ \al  - \frac 12 - \eps, \infty}_x}
\| v_1\|_{X^{\frac 12 + \dl_1, \frac 12 + \dl_2 }_T}
\|v_2 \|_{X^{\frac 12 + \dl_1 , \frac 12 + \dl_2}_T}
\end{align*}

\noi
for any $0 < T \le 1$.

\end{lemma}

\subsection{Proof of Theorem \ref{THM:1}}
\label{SUBSEC:LWP2}

In this section, we prove the following proposition.
Theorem~\ref{THM:1} then follows
from this proposition and Lemmas \ref{LEM:sto1} - \ref{LEM:ran1}.

\begin{proposition}\label{PROP:LWP}
Let $\al > 0$, $s > \frac 12$,  and $T_0 >0$.
Then, there exists small $\eps = \eps(\al, s)$, 
 $\dl_1  = \dl_1(\al, s) $, 
  $\dl_2 = \dl_2(\al, s) >0$
 such that 
if 
\begin{itemize}
\item   $\<1> $ is a distribution-valued function belonging to $C([0, T_0]; W^{\al -\frac 12 - \eps, \infty}(\T^3))$, 

\smallskip
\item   
$\<30> $ is a distribution-valued function belonging to $C([0, T_0]; W^{\al - \eps, \infty}(\T^3))$, 

\smallskip
\item   
$\<30> \<1>$ is a distribution-valued function belonging to $C([0, T_0]; W^{\al -\frac 12 - \eps, \infty}(\T^3))$, 


\smallskip
\item   
$\<320> $ is a  function belonging to 
$X^{\al + \frac 12 -\eps, \frac 12 +\dl_2}([0, T_0])$,

\smallskip
\item   
$\<70> $ is a function belonging to 
$X^{\al + \frac 12 -\eps, \frac 12 +\dl_2}([0, T_0])$,

\smallskip
\item 
 the operator
 $\If^{\<2>}$ 
belongs to  the class
$ \L^{\frac 12  + \dl_1, \frac 12 +\dl_1,  \frac 12 +\dl_2}_{T_0}$
defined in \eqref{Op1},

\smallskip

\end{itemize}

\noi
then the equation \eqref{SNLW11} is locally well-posed in 
$\H^{s}(\T^3)$.
More precisely, 
given any $(u_0, u_1)\in \H^{s}(\T^3)$, 
there exist $0 < T \le T_0$ 
and   a unique solution $v $ to 
the cubic SNLW~\eqref{SNLW11} on $[0, T]$
in the class 
\[X^{\frac 12 +\dl_1, \frac 12 +\dl_2}([0, T])
\subset C([0, T]; H^{\frac 12 +\dl_1}(\T^3)).\]

\noi
Furthermore, the solution $v$
depends  continuously 
on the enhanced data set
\begin{align}
\Xi = \big (u_0, u_1, 
\<1>, \<30>, 
 \<30>\<1>, 
\<320>, 
\<70>, 
\If^{  \<2> }\big)
\label{data2}
\end{align}

\noi
in the class
\begin{align*}
\mathcal{X}^{s, \al, \eps}_T
& = \H^{s}(\T^3) 
\times 
C([0,T]; W^{\al -\frac 12 - \eps, \infty}(\T^3))\\
& \hphantom{X}
\times 
C([0,T]; W^{\al  - \eps, \infty}(\T^3))
\times 
C([0,T]; W^{\al - \frac 12 - \eps, \infty}(\T^3))\\
& \hphantom{X} 
\times 
X^{\al + \frac 12 -\eps, \frac 12 +\dl_2}([0, T])
\times X^{\al + \frac 12 -\eps, \frac 12 +\dl_2}([0, T])\\
& \hphantom{X}
\times
 \L\big(X^{\frac 12  + \dl_1, \frac 12 +\dl_2}([0, T]); X^{\frac 12 +\dl_2, \frac 12 +\dl_2}([0, T])\big). 
\end{align*}

\end{proposition}

\begin{proof}
Given $\al > 0$ and $s > \frac 12$, 
fix small $\eps > 0$ such that $\eps < \min(\al, s - \frac 12 )$.
Given an enhanced data set $\Xi$ as in~\eqref{data2}, 
we set
\begin{align*}
\Xi(\xi)  = 
 \big (
\<1>, \<30>, 
 \<30>\<1>, 
\<30>^2 \<1>, 
\<320>, 
\If^{  \<2> }\big)
\end{align*}

\noi
and 
\begin{align*}
\| \Xi (\xi)  \|_{\mathcal{Y}^{\al,  \eps}_{T_0}}
& =
\|\<1>\|_{C_{T_0} W^{\al -\frac 12  - \eps, \infty}_x} 
+ \|\<30>\|_{C_{T_0} W^{\al  - \eps, \infty}_x} 
+ \| \<30>\<1>\|_{C_{T_0} W^{\al -\frac 12  - \eps, \infty}_x} \\
& \hphantom{X}
+ \big\|\<320>\big\|_{X^{\al + \frac 12 -\eps, \frac 12 +\dl_2}_{T_0}}
+ \big\|\<70>\big\|_{X^{\al + \frac 12 -\eps, \frac 12 +\dl_2}_{T_0}}\\
& \hphantom{X}
+ \|\If^{  \<2> }\|_{\L^{\frac 12  + \dl_1, \frac 12  + \dl_1, \frac 12  + \dl_2}_{T_0}}, 
\end{align*}

\noi
where $\L^{\frac 12  + \dl_1, \frac 12  + \dl_1, \frac 12  + \dl_2}_{T_0}$ is as in \eqref{Op2}.
In the following, we assume that 
\begin{align}
\| \Xi (\xi)  \|_{\mathcal{Y}^{\al,  \eps}_{T_0}}
\leq K
\label{data4}
\end{align}

\noi
for some $K \geq 1$.

Given the enhanced data set $\Xi$ in \eqref{data2}, define a map $\G_\Xi$ by 
\begin{align*}
\G_\Xi(  v)
& = 
 S(t) (u_0, u_1) + 
\I\big(-  v^3  + 3 (\<30> - \<1>)v^2 
- 3 \<30>^2 v\big)\\ 
& \quad +6\I\big((\<30>\<1>) v\big)
-3 \If^{  \<2> }(v)\\
& \quad 
+ \I\big( \<30>^3\big)
  -3\, \<70>
  + 3 \,  \<320>.
\end{align*}

\noi
Fix $0 < T \le T_0$.
From 
Lemmas \ref{LEM:lin1}
and \ref{LEM:Str2} with \eqref{data4}, 
we have
\begin{align}
\begin{split}
\|\I( \<30>v^2 )\|_{X^{\frac 12 +\dl_1, \frac 12 +\dl_2}_T}
& \les T^{\dl_2}  \| \<30>v^2 \|_{X^{-\frac 12 +\dl_1, -\frac 12 +2\dl_2}_T}
\le T^{\ta}  \| \<30>\|_{L^\infty_{T, x}}
\|v \|_{L^4_{T, x}}^2\\
& \les 
T^{\ta}  K 
\|v \|_{X^{\frac 12 +\dl_1, \frac 12 +\dl_2}_T}^2
\end{split}
\label{D2}
\end{align}

\noi
and
\begin{align}
\begin{split}
\|\I( \<30>^2v )\|_{X^{\frac 12 +\dl_1, \frac 12 +\dl_2}_T}
& \le 
T^{\ta}  \| \<30>\|_{L^\infty_{T, x}}^2
\|v \|_{L^2_{T, x}}
 \le 
T^{\ta}  K^2 
\|v \|_{X^{\frac 12 +\dl_1, \frac 12 +\dl_2}_T}
\end{split}
\label{D3}
\end{align}

\noi
for some $\ta > 0$.
Similarly, we have
\begin{align}
\|\I( \<30>^3 )\|_{X^{\frac 12 +\dl_1, \frac 12 +\dl_2}_T}
& \les T^{\dl_2}  \| \<30>^3 \|_{X^{-\frac 12 +\dl_1, -\frac 12 +2\dl_2}_T}
\le T^{\ta}  \| \<30>\|_{L^\infty_{T, x}}^3
\le T^{\ta}  K^3 .
\label{D3a}
\end{align}

\noi
From 
Lemma \ref{LEM:lin1}
and Lemma \ref{LEM:gko}
with \eqref{data4}, we have
\begin{align}
\begin{split}
\big\|\I\big((\<30>\<1>) v\big)\big\|_{X^{\frac 12 +\dl_1, \frac 12 +\dl_2}_T}
& \les T^{\dl_2}
\|(\<30>\<1>) v\|_{X^{- \frac 12 +\dl_1,-  \frac 12 +2\dl_2}_T}
\les T^{\dl_2}
\|(\<30>\<1>) v\|_{L^2_T H^{-\frac 12 + \dl_1}_x}\\
& \les T^{\ta}
\|\<30>\<1>\|_{L^\infty_T W^{-\frac 12 + \dl_1, \infty}_x}
\| v\|_{L^2_T H^{\frac 12 - \dl_1}_x}\\
& \les T^{\ta}
K 
\| v\|_{X^{\frac 12 +\dl_1, \frac 12 +\dl_2}_T}, 
\end{split}
\label{D4}
\end{align}

\noi
provided that $\dl_1 + \eps \leq \al$.
From \eqref{Op2} and \eqref{data4}, we have
\begin{align}
\|\If^{  \<2> }(v)\|_{X^{\frac 12 +\dl_1, \frac 12 +\dl_2}_T}
\leq T^\ta \|\If^{  \<2> }\|_{\L^{\frac 12 +\dl_1, \frac 12 +\dl_1, \frac 12 +\dl_2}_{T_0}}
\|v\|_{X^{\frac 12 +\dl_1, \frac 12 +\dl_2}_T}
\leq T^\ta K
\|v\|_{X^{\frac 12 +\dl_1, \frac 12 +\dl_2}_T}.
\label{D5}
\end{align}

\noi
Hence, by applying
Lemmas \ref{LEM:Str1} 
and \ref{LEM:lin1}, 
then
Lemma \ref{LEM:tri}, 
\eqref{D2}, 
Lemma 
\ref{LEM:trilin2}, 
\eqref{D3},  \eqref{D4}, \eqref{D5},  \eqref{D3a}, and Lemma \ref{LEM:sto3}
with  \eqref{data4}, we have 
\begin{align*}
\|\G_\Xi(  v)\|_{X^{\frac 12 +\dl_1, \frac 12 +\dl_2}_T}
& \les \|(u_0, u_1)\|_{\H^s}
+ T^\ta 
\Big( \|  v\|_{X^{\frac 12 +\dl_1, \frac 12 +\dl_2}_T}^3
+ K^3 \Big) + K.
\end{align*}

\noi
An analogous computation yields
a difference estimate
on $\G_\Xi(  v_1) - \G_\Xi(  v_2)$.
Therefore, Proposition \ref{PROP:LWP}
follows from a standard contraction argument.
\end{proof}

\section{Regularities of the stochastic terms}
\label{SEC:sto2}

In this section, we present the proof of Lemmas \ref{LEM:sto1} - \ref{LEM:ran1},
which are basic tools in applying Proposition \ref{PROP:LWP}
to finally prove Theorem \ref{THM:1}.
In view of the local well-posedness result in~\cite{OPTz}, 
we assume that 
$0 < \al \le \frac 14$ in the following.
Without loss of generality, we assume that $T \leq 1$.
The main tools in this section are the counting 
estimates from \cite[Section 4]{Bring} and the random matrix estimate
(see Lemma \ref{LEM:DNY} below) from \cite{DNY2},
which capture the multilinear dispersive effect of the wave equation.
For readers' convenience, 
we collect the relevant counting estimates in Appendix \ref{SEC:A}
and  the relevant definitions and estimates
for random matrices and tensors in Appendix \ref{SEC:C}.
We show in details how to reduce the relevant stochastic estimates
to some basic counting and (random) matrix/tensor estimates 
 studied in \cite[Section 4]{Bring} and~\cite{DNY2}.

In the remaining part of this section, we assume $0 < T < T_0 \leq 1$.

\subsection{Basic stochastic terms}
\label{SUBSEC:4.1}

We first present the proof of Lemma \ref{LEM:sto1}.

\begin{proof}[Proof of Lemma \ref{LEM:sto1}]

(i) 
Let $t \geq 0$.
From 
\eqref{sigma1}, we have
\begin{align}
\begin{split}
\E\big[|\ft{\<1>}_N(n, t)|^2\big] 
& \leq C(t) \jb{n}^{-2-2\al}
\end{split}
\label{sconv3}
\end{align}

\noi
for any $n \in \Z^3$ and $N \geq 1$.
Also, by the mean value theorem and an interpolation argument as in \cite{GKO2}, 
we have 
\[
\E\big[ |\ft{\<1>}_N(n, t_1) - \ft{\<1>}_N(n, t_2)|^2 \big]
\les_T \jb{n}^{-2(1+\al)+\theta} |t_1-t_2|^\theta
\]
for any $\theta \in [0, 1]$, $n \in \Z^3$,  and $0 \le t_2 \le t_1 \le T$ with $t_1-t_2 \le 1$,
uniformly in $N \in \N$.
Hence, from Lemma \ref{LEM:reg}, 
we conclude that 
 $\<1>_N \in C([0, T]; W^{\al - \frac 12- \eps, \infty}(\T^3))$
for any $\eps > 0$, 
 almost surely.
Moreover, a slight modification
of the argument, using Lemma \ref{LEM:reg},   yields
that $\{ \<1>_N \}_{N \in \N}$ is almost surely
 a Cauchy sequence in $C([0, T]; W^{\al -\frac 12 - \eps, \infty}(\T^3))$,
  thus converging to some limit $\<1>$.
Since the required modification is exactly the same as in  \cite{GKO2}, 
we omit the details here.

\begin{remark}	\rm
In the remaining part of this section,
we establish uniform (in $N$) regularity bounds
on the truncated stochastic  terms (such as $\<30>_N$)
but may omit the convergence part of the argument.
Furthermore, 
as for $\<30>_N\<1>_N $ studied in Lemma \ref{LEM:sto2}, 
we only establish a uniform (in $N$) regularity bound
on  $\<30>_N\<1>_N(t) $ for each fixed $0 < t \leq T \leq 1$.
A slight modification as above yields continuity in time but we omit details.

\end{remark}

\noi
(ii)
It is possible to prove 
this part
by proceeding as
in \cite{GKO2, OOComp}
(i.e.~without the use of the $X^{s, b}$-spaces).
In the following, however,  we follow Bringmann's approach
\cite{Bring}, adapted to the stochastic PDE setting.
More precisely, 
we show that 
given any $\dl_1> 0$ and sufficiently small $ \dl_2 > 0$, 
the sequence $\{\<3>_N\}_{N \in \N}$
is a Cauchy sequence 
in $X^{\al - 1 - \dl_1, - \frac 12 + \dl_2}([0, T])$, 
almost surely, 
and thus converges almost surely to  $\<3>$
in the same space, 
where
$\<3>$ is the almost sure limit of
$\{\<3>_N\}_{N \in \N}$ in  $C([0,T];W^{3\al -\frac 32 - ,\infty}(\T^3))$
discussed in Section \ref{SEC:1}.

Our first goal is to prove the following bound;
given any $\dl_1> 0$ and sufficiently small $ \dl_2 > 0$, 
there exists $\ta > 0$ such that 
\begin{align}
\Big\|\| \<3>_N\|_{X^{\al - 1 - \dl_1, - \frac 12 + \dl_2}_T}\big\|_{L^p(\O)}
\les p^\frac{3}{2} T^\ta 
\label{X1}
\end{align}

\noi
for any $p \ge 1$ and $0 < T \le 1$, uniformly in $N \in \N$.

Let us first compute the space-time Fourier transform of 
$\<3>_N$ (with a time cutoff function).
From  \eqref{sto3} with \eqref{sto2}, 
we can write the spatial Fourier transform 
 $\ft{\<3>}_N(n, t)$
  as the following multiple Wiener-Ito integral
 (as in \cite{MWX}):
\begin{align}
\ft{\<3>}_N(n, t) = 
\sum_{\substack{n = n_1 + n_2 + n_3\\|n_j|\leq N}}
\int_0^t \int_0^t \int_0^t
\prod_{j = 1}^3 \frac{\sin((t - t_j)\jb{n_j})}{\jb{n_j}^{1+\al}}
 dB_{n_3} (t_3)dB_{n_2}(t_2)dB_{n_1}(t_1).
\label{X2}
\end{align}

\noi
We emphasize that the renormalization in \eqref{sto3} is 
embedded in the definition of the multiple Wiener-Ito integral.

We now compute the space-time Fourier transform of $\ind_{[0, T]}\<3>$, 
where $\ind_{[0, T]}$ denotes the sharp cutoff function
on the time interval $[0, T]$.
From~\eqref{X2}
and 
the stochastic Fubini theorem (\cite[Theorem 4.33]{DPZ14}; see also Lemma \ref{LEM:B3}), we have
\begin{align}
\begin{split}
&  \ft{\ind_{[0, T]}\<3>}_N   (n, \tau) 
  = \frac{1}{\sqrt{2\pi}}
\sum_{\substack{n = n_1 + n_2 + n_3\\|n_j|\leq N}}
\int_\R \ind_{[0, T]} e^{- i t \tau} \int_0^t \int_0^{t} \int_0^{t}\\
& 
\hphantom{XXXXXXXXX}
\prod_{j = 1}^3 \frac{\sin((t - t_j)\jb{n_j})}{\jb{n_j}^{1+\al}}\,
 dB_{n_3} (t_3)dB_{n_2}(t_2)dB_{n_1}(t_1) dt\\
 & 
 \hphantom{XX}
 = \frac{1}{\sqrt{2\pi}}
\sum_{\substack{n = n_1 + n_2 + n_3\\|n_j|\leq N}}
\int_0^T   \int_0^{T} \int_0^{T}
F_{n_1, n_2, n_3} (t_1, t_2, t_3) 
 dB_{n_3} (t_3)dB_{n_2}(t_2)dB_{n_1}(t_1), 
\end{split}
\label{X5}
\end{align}

\noi
where $F_{n_1, n_2, n_3} (t_1, t_2, t_3, \tau) $ is defined by 
\begin{align}
F_{n_1, n_2, n_3} (t_1, t_2, t_3, \tau) 
= \int_{0}^T 
e^{- i t \tau}  
\prod_{j = 1}^3 \frac{\sin((t - t_j)\jb{n_j})}{\jb{n_j}^{1+\al}}
\ind_{[0,t]}(t_j)
dt .
\label{X5a}
\end{align}

\noi
Note that $F_{n_1, n_2, n_3} (t_1, t_2, t_3, \tau) $
is symmetric in $t_1, t_2, t_3$.

Given dyadic $N_j \ge 1$, $j = 1, 2, 3$, 
let us  denote by $A^N_{N_1, N_2, N_3}$
the contribution 
to $\ind_{[0, T]}\<3>_N$ from $|n_j|\sim N_j$, $j = 1, 2, 3$,  in \eqref{X5}.
We first  compute the $X^{s-1, b }$-norm of 
$A^N_{N_1, N_2, N_3}$ 
with $b = -\frac   12 - \dl$ for $\dl > 0$.
We then interpolate it with the trivial $X^{0, 0}$-bound.
Recall the trivial bound:
\begin{align}
\begin{split}
\| u\|_{X^{s, b}}
& = \| \jb{n}^{s} \jb{|\tau| - \jb{n}}^b \, 
\ft u(n, \tau)\|_{\l^2_n L^2_\tau}\\
& \leq \sum_{\eps_0 \in \{-1, 1\}}
\| \jb{n}^{s} \jb{\tau + \eps_0  \jb{n}}^b \, 
\ft u(n, \tau)\|_{\l^2_n L^2_\tau}\\
& = \sum_{\eps_0 \in \{-1, 1\}}
\| \jb{n}^{s} \jb{\tau }^b \, 
\ft u(n, \tau - \eps_0  \jb{n})\|_{\l^2_n L^2_\tau}
\end{split}
\label{X6}
\end{align}

\noi
for any $s, b \in \R$.
Then, 
defining
$\kk(\bar n) = \kk_{\eps_0, \eps_1, \eps_2, \eps_3}(n_1, n_2, n_3)$ by 
\begin{align}
\kk(\bar n) = \eps_0 \jb{n_{123}} + \eps_1\jb{n_1}
+ \eps_2\jb{n_2}+ \eps_3\jb{n_3}, 
\label{X6a}
\end{align}

\noi
with $\eps_j \in \{-1, 1\}$ for $j = 0,1,2,3$, 
it follows from 
\eqref{X6},  \eqref{X5}, 
Fubini's theorem, Ito's isometry, 
and  expanding the sine functions in 
\eqref{X5a} in terms of the complex exponentials
that 
\begin{align}
\Big\|\|  & A^N_{N_1, N_2, N_3}   \|_{X^{s - 1, -\frac 12 - \dl}_T}\Big\|_{L^2(\O)}^2
\notag \\
& \les 
\sum_{\eps_0\in \{-1, 1\}}
\sum_{n \in \Z^3}\int_\R \jb{n}^{2(s-1)} \jb{\tau}^{-1 - 2\dl} \notag\\
& 
\quad \times \Bigg\{\sum_{\substack{n = n_1 + n_2 + n_3\\|n_j|\leq N\\|n_j|\sim N_j }}
\int_{[0, T]^3} 
|F_{n_1, n_2, n_3} (t_1, t_2, t_3, \tau-\eps_0\jb{n}) |^2
dt_3 dt_2 dt_1\Bigg\} d\tau 
\notag \\
& \les \sum_{\eps_0, \eps_1, \eps_2, \eps_3 \in \{-1, 1\}}
\sum_{n \in \Z^3}\int_\R \jb{n}^{2(s-1)} \jb{\tau}^{-1 - 2\dl}
\Bigg\{\sum_{\substack{n = n_1 + n_2 + n_3\\|n_j|\leq N\\|n_j|\sim N_j }}
\prod_{j = 1}^3 
\frac{1}{\jb{n_j}^{2(1+\al)}}
\notag \\
& 
\quad 
\times 
\int_{[0, T]^3} \bigg|
\int_{\max(t_1, t_2, t_3)}^T
e^{- i t (\tau - \kk(\bar n))} 
dt\bigg|^2
dt_3 dt_2 dt_1\Bigg\} d\tau 
\notag \\
& 
\les \sum_{\eps_0, \eps_1, \eps_2, \eps_3 \in \{-1, 1\}}
\sum_{n \in \Z^3}\sum_{\substack{n = n_1 + n_2 + n_3\\|n_j|\sim N_j}}
 \frac{\jb{n}^{2(s-1)} }{\prod_{j =1}^3 \jb{n_j}^{2(1+\al)}}
\int_\R \frac{1}{\jb{\tau}^{1 + 2\dl} \jb{\tau - \kk(\bar n)}^2} d\tau 
\notag \\
& 
\les \sum_{\eps_0, \eps_1, \eps_2, \eps_3 \in \{-1, 1\}}
\sum_{n \in \Z^3}\sum_{\substack{n = n_1 + n_2 + n_3\\|n_j|\sim N_j}}
 \frac{\jb{n}^{2(s-1)} }{\prod_{j =1}^3 \jb{n_j}^{2(1+\al)}} \jb{\kk(\bar n)}^{-1 - 2\dl}
 \notag \\
& 
\les \sum_{\eps_0, \eps_1, \eps_2, \eps_3 \in \{-1, 1\}}
\sup_{m \in \Z}
\sum_{n \in \Z^3}\sum_{\substack{n = n_1 + n_2 + n_3\\|n_j|\sim N_j}}
 \frac{\jb{n}^{2(s-1)} }{\prod_{j =1}^3 \jb{n_j}^{2(1+\al)}} \cdot \ind_{\{|\kk(\bar n) - m|\leq 1\}}
\label{X7}
\end{align}

\noi
for any $\dl > 0$, uniformly in dyadic $N_j \geq 1$, $j = 1, 2, 3$.
By noting
\begin{align}
\prod_{j = 1}^3 \jb{n_j}^{- 2\al}\les \jb{n_{12}}^{-2\al},
\label{X7a} 
\end{align}

\noi
we can reduce the right-hand side of \eqref{X7}
to the setting of the Hartree nonlinearity studied in~\cite{Bring}.
In particular, 
from \eqref{X7} with \eqref{X7a}
and 
the cubic sum estimate (Lemma \ref{LEM:A1}), 
we obtain
\begin{align}
\Big\|\|  A^N_{N_1, N_2, N_3}   \|_{X^{s - 1, -\frac 12 - \dl}_T}\Big\|_{L^2(\O)}
\les N_{\max}^{s - \al}, 
\label{X8}
\end{align}

\noi
where 
$N_{\max} = \max(N_1, N_2, N_3)$.
This provides an estimate for $s < \al$ and $b = - \frac 12 - \dl < -\frac 12$.

On the other hand, using \eqref{X5}, we have 
\begin{align}
\begin{split}
\Big\|\| &  A^N_{N_1, N_2, N_3}   \|_{X^{0, 0}_T}\Big\|_{L^2(\O)}^2
= \Big\|\|   A^N_{N_1, N_2, N_3}   \|_{L^2_{T, x}}\Big\|_{L^2(\O)}^2
\\
& 
\les T^\ta 
\sum_{\substack{n_1 ,  n_2 ,  n_3\in \Z^3 \\|n_j|\sim N_j}}
\prod_{j =1}^3 \jb{n_j}^{-2(1+\al)}
 \\
& 
\les T^\ta   N_{\max}^{3 - 6\al}
\end{split}
\label{X9}
\end{align}

\noi
for some $\ta > 0$.
Hence, it follows from  interpolating   \eqref{X8} and \eqref{X9}
and then applying the Wiener chaos estimate
(Lemma~\ref{LEM:hyp})
that given $s < \al$, 
there exist small $\dl_2 > 0$ and  $\eps > 0$ such that 
\begin{align*}
\Big\|\|  A^N_{N_1, N_2, N_3}   \|_{X^{s - 1, -\frac 12 +  \dl_2}_T}\Big\|_{L^p(\O)}
\les p^\frac {3}{2} T^\ta N_{\max}^{-  \eps}
\end{align*}

\noi
for  any $p \geq 1$, 
uniformly in dyadic $N_j \geq 1$, $j = 1, 2, 3$.
By summing over dyadic blocks $N_j\geq 1$, $j = 1, 2, 3$, 
we obtain the bound \eqref{X1}
(with $b = -\frac 12 +  \dl_2 > -\frac 12$).

As for the convergence of $\<3>_N$
to $\<3>$ in 
$X^{\al - 1 - \dl_1, - \frac 12 + \dl_2}([0, T])$, 
we can simply repeat the computation 
above to estimate the difference
$\ind_{[0, T]}\<3>_M - \ind_{[0, T]}\<3>_N$
for $M \ge N \ge 1$.
Fix $s < \al$.
Then, in \eqref{X7}, 
we replace the restriction $|n_j| \leq N $
in the summation of $n_j$, $j = 1, 2, 3$, 
by 
 $N \le \max(|n_1|, |n_2|, |n_3|) \le M$, 
 which allows us to gain a small negative power of $N$.
As a result, in place of \eqref{X8}, we obtain
\begin{align*}
\Big\|\|  A^M_{N_1, N_2, N_3} -  A^N_{N_1, N_2, N_3}   \|_{X^{s - 1, -\frac 12 - \dl}_T}\Big\|_{L^2(\O)}
\les N^{-\eps} N_{\max}^{s - \al+\eps}
\end{align*}

\noi
for any  small $\eps > 0$
and $M \ge N \ge 1$.
Then, the interpolation argument with \eqref{X9} as above yields
that given $s < \al$, there exist small  $\dl_2>0$ and $\eps  > 0$ such that 
\begin{align}
\Big\|\|  \ind_{[0, T]}\<3>_M - \ind_{[0, T]}\<3>_N   \|_{X^{s - 1, -\frac 12 + \dl_2}_T}\Big\|_{L^p(\O)}
\les p^\frac{3}{2} T^\ta N^{-\eps} 
\label{ZZ0}
\end{align}

\noi
for any $p \ge 1$  
and $M \ge N \ge 1$.
Then, by applying Chebyshev's inequality 
and   the Borel-Cantelli lemma, 
we conclude the almost sure convergence of $\<3>_N$.
See \cite{OPTz}.

Finally, fix $s< \al$.
Given $N \in \N$, let 
$ H_N=  \I(\ind_{[0, T]}(\<3>_N - \<3>))$.
Then,  we have 
\begin{align}
\<30>_N (t) - \<30>(t)
= 
H_N(t)
\label{ZZ0a}
\end{align}

\noi
for $t \in [0, T]$, 
%
%
Note that  
from \eqref{X5}, we have 
$ \ft H_N(n, t) \in \H_3$ and,  furthermore, by the independence
of $\{B_n\}_{n \in \Z^3}$ (modulo $B_{-n} = \cj B_n$), we have 
\begin{align}
\E \big[\ft H_N(n, t_1)\ft H_N(m, t_2)\big] = 
\ind_{n+m = 0} \, 
\E \big[ \ft H_N(n, t_1) \cj{\ft H_N(n, t_2)}\big]
\label{ZZ1}
\end{align}

\noi 
for any $t_1, t_2 \in \R$.
Then, by 
 \eqref{ZZ0a}, 
Sobolev's inequality (with finite $r\gg 1$ such that $r\dl_0 > 3$
for some small $ \dl_0 > 0$), 
Minkowski's integral inequality, 
the Wiener chaos estimate (Lemma~\ref{LEM:hyp}) with \eqref{ZZ1}, 
Hausdorff-Young's inequality (in time),  we have, for any $p \geq \max(q, r) \gg1$,  
\begin{align*}
\big\|\| \<30>_N - \<30>\|_{L^\infty_T W_x^{s, \infty}}\big\|_{L^p(\O)}
& \les \big\|\| 
 H_N
\|_{ W_t^{\dl_0, r} W_x^{s+\dl_0, r}}\big\|_{L^p(\O)}\\
&  \leq \bigg\| 
\Big\|  \sum_{n \in \Z^3} \jb{\nb_t}^{\dl_0} \jb{n}^{s+\dl_0} \ft H_N(n, t) e_n(x)\Big\|_{L^p(\O)} \bigg\|_{   L^r_x L^r_t}\\
&  \les p^\frac{3}{2} \bigg\|
\Big\|  \sum_{n \in \Z^3} \jb{\nb_t}^{\dl_0} \jb{n}^{s+\dl_0}\ft H_N(n, t) e_n(x)\Big\|_{L^2(\O)} \bigg\|_{  L^r_x L^r_t }\\
&  \les p^\frac{3}{2} 
\| \jb{\tau}^{\dl_0} \jb{n}^{s+\dl_0} \ft H_N(n, \tau) \|_{L^2(\O; \l^2_nL^{r'}_\tau) }.
\end{align*}

\noi
Now, by the triangle inequality: $\jb{\tau}^{\dl_0}\les 
\jb{|\tau| - \jb{n}}^{\dl_0} \jb{n}^{\dl_0}$, 
H\"older's inequality (in $\tau$), 
followed by the  nonhomogeneous linear estimate (Lemma \ref{LEM:lin1})
and \eqref{ZZ0} (with $p = 2$, $M = \infty$, and $s$ replaced by $s + 2\dl_0 < \al$),
we obtain
\begin{align*}
\big\|\| \<30>_N - \<30>\|_{L^\infty_T W_x^{s, \infty}}\big\|_{L^p(\O)}
&  
   \les  p^\frac 32 \big\|  
\| H_N\|_{X^{s + 2\dl_0, \frac 12 + \dl_0}}\big\|_{L^2(\O)}\\
& \les p^\frac{3}{2} T^\ta N^{-\eps}
\end{align*}

\noi
by choosing $\dl_0 > 0$ sufficiently small.
Then, 
the regularity and convergence claim
for $\{\<30>_N\}_{N \in \N}$
follows from applying Chebyshev's inequality 
and   the Borel-Cantelli lemma as before.
\end{proof}

\begin{remark}\rm
Given a function $f \in L^2((\Z^3\times \R_+)^k)$, 
define the multiple stochastic integral $I_k[f]$ by 
\begin{align*}
I_k [f] = \sum_{n_1,\dots,n_k \in \Z^3} \int_{[0,\infty)^k} 
f( n_1, t_1, \dots,  n_k, t_k) 
d B_{n_1} (t_1) \cdots d B_{n_k} (t_k).
\end{align*}

\noi
See Appendix \ref{SEC:B} for the basic definitions and  properties
of multiple stochastic integrals.
In terms of   multiple stochastic integrals, 
we can express~\eqref{X2} as 
\begin{align*}
\ft{\<3>}_N(n, t) = 
I_3 \big[
f_{n, t}],
\end{align*}

\noi
where $f_{n, t}$ is defined  by 
\begin{align*}
f_{n, t}(n_1,t_1, n_2, t_2,n_3,t_3) =
\ind_{n = n_{123}}\cdot 
\bigg( \prod_{j = 1}^3 \frac{\sin((t - t_j)\jb{n_j})}{\jb{n_j}^{1+\al}} 
\cdot  \ind_{|n_j| \leq N}\cdot \ind_{[0,t]}(t_j)\bigg)
\end{align*}

\noi
for $(n_1,t_1, n_2, t_2,n_3,t_3) \in (\Z^3 \times \R) ^3$.
Then, by Fubini's theorem for multiple stochastic integrals (Lemma \ref{LEM:B3}), 
we have
\begin{align*}
 \ft{\ind_{[0, T]}\<3>}_N   (n, \tau) 
 = I_3 \big[\F_t(\ind_{[0, T]}f_{n, \cdot })(\tau)\big], 
\end{align*}

\noi
where $\F_t$ denotes the Fourier transform in time.
With this notation, 
it follows from Lemma~\ref{LEM:B1}
that 
we can write the second moment of the $X^{s, b}$-norm
of $A^N_{N_1, N_2, N_3}$, appearing in \eqref{X7} and~\eqref{X9}, 
in a concise manner:
\begin{align*}
\Big\|\|   A^N_{N_1, N_2, N_3}   \|_{X^{s , b}_T}\Big\|_{L^2(\O)}^2
=  3!
\sum_{n \in \Z^3}\int_\R \jb{n}^{2s} \jb{|\tau|- \jb{n}}^{2b}
\big\|\F_t(\ind_{[0, T]}f_{n, \cdot}^{\bar N})(\tau)\big\|_{\l^2_{n_1, n_2, n_3} L^2_{t_1, t_2, t_3} }^2d\tau, 
\end{align*}

\noi
where $f_{n, t}^{\bar N}$ is given by 
\[f_{n, t}^{\bar N}
= f_{n, t}\cdot \prod_{j = 1}^3 \ind_{|n_j|\sim N_j}.\]

In the following, for conciseness of the presentation, 
we express various stochastic objects
as  multiple stochastic integrals on $(\Z^3\times \R_+)^k$
and carry out analysis. 
For this purpose, we set
\begin{align}
z_j = (n_j, t_j) \in \Z^3 \times \R_+
\label{XX2}
\end{align}

\noi
and use the following short-hand notation:
\begin{align}
\| f(z_j) \|_{L^p_{z_j}} = \|f(n_j, t_j)\|_{\l^p_{n_j} L^p_{t_j}}.
\label{XX3}
\end{align}

\noi
Note, however, that one may also carry out equivalent analysis
at the level of
 multiple Wiener-Ito integrals
as in the proof of Lemma \ref{LEM:sto1} 
presented above.

\end{remark}

Next, we briefly discuss the proof of Lemma \ref{LEM:sto2}.

\begin{proof}[Proof of Lemma \ref{LEM:sto2}]

%

By the paraproduct decomposition \eqref{para1}, 
we have
\begin{align*}
\<30>_N\<1>_N =
\<30>_N\pl \<1>_N + \<30>_N\pe \<1>_N + \<30>_N\pg \<1>_N. 
\end{align*}

\noi
In view of Lemma \ref{LEM:para} with \eqref{Bes1}, 
the paraproducts
$\<30>_N\pl \<1>_N$ and  $\<30>_N\pg \<1>_N$
belong to 
$C([0,T];W^{\al - \frac12 - \eps,\infty}(\T^3))$
for any $\eps >0$, almost surely.
Hence, it remains to study 
the resonant product $\<31p>_N := \<30>_N\pe \<1>_N$.
We only study the regularity of the resonant product for
a fixed time since 
the continuity in time 
and the convergence follow from a systematic modification.
In the following, we show 
\begin{align}
\E\big[|\ft{\<31p>}_N(n, t)|^2\big]
& \les \jb{n}^{-2-4\al}
\label{XY0}
\end{align}

\noi
for any $n \in \Z^3$ and $N \geq 1$.
Note the bound \eqref{XY0} together with Lemma \ref{LEM:reg}
shows that the resonant product 
$\<31p>_N $ is smoother and has (spatial) regularity
$ 2\al - \frac 12 - = (\al - ) + \big(\al - \frac 12-\big)$.

As in \cite{MWX}, by decomposing
$\ft {\<31p>}_N(n, t)$ into components in 
the homogeneous Wiener chaoses $\H_k$, $k = 2, 4$,  
we have
\begin{align*}
\ft {\<31p>}_N(n, t)
= \ft{\<31p>}^{(4)}_N(n, t) + \ft{\<31p>}_N^{(2)}(n, t), 
\end{align*}

\noi
where 
$\ft{\<31p>}^{(4)}_N(n, t) \in \H_4$ and 
$\ft{\<31p>}^{(2)}_N(n, t) \in \H_2$.
See, for example, 
\cite[Proposition 1.1.2]{Nua} 
and Lemma \ref{LEM:prod} on
the product formula
for multiple Wiener-Ito integrals
(and it also follows from Ito's lemma  as explained in \cite{MWX}). 
From the orthogonality of $\H_4$ and $\H_2$, we have 
\begin{align*}
\E\Big[|\ft {\<31p>}_N(n, t)|^2\Big]
= \E\Big[|\ft{\<31p>}^{(4)}_N(n, t)|^2\Big] +  \E\Big[|\ft{\<31p>}_N^{(2)}(n, t)|^2\Big].
\end{align*}

\noi
Hence, it suffices to prove \eqref{XY0}
for $\<31p>_N^{(j)}$, $j = 2, 4$.

From a slight modification\footnote{Namely, with $s = 0$ and dropping the summation over $n$ in \eqref{X7}.} of \eqref{X7} with Lemma \ref{LEM:A1a}, we have 
\begin{align}
\E\big[|\ft{\<30>}_N(n, t)|^2\big] 
& \leq C(t) \jb{n}^{-3-2\al}
\label{XY2}
\end{align}

\noi
for any $n \in \Z^3$ and $N \geq 1$.
Then, from Jensen's inequality (see \eqref{Jen}),\footnote{See the discussion on $\<31p>$ in Section 4 of \cite{MWX}.
See also Section~10 in~\cite{Hairer}.} 
\eqref{sconv3}, \eqref{XY2}, and Lemma \ref{LEM:SUM}, we have
\begin{align}
\begin{split}
\E\big[|\ft{\<31p>}_N^{(4)}(n, t)|^2\big]
& \les \sum_{\substack{n \in \Z^3\\|n_1|\sim|n- n_1|}}
\E\big[|\ft{\<30>}_N(n_1, t)|^2\big] 
\E\big[|\ft{\<1>}_N(n-n_1, t)|^2\big] \\
& \le C(t)\sum_{\substack{n \in \Z^3\\|n_1|\sim|n- n_1|}}
 \frac{1}{\jb{n_1}^{3+2\al}\jb{n - n_1}^{2+2\al}}\\
& \le C(t) \jb{n}^{-2-4\al}
\end{split}
\label{XY3}
\end{align}

\noi
for any $n \in \Z^3$ and $N \geq 1$, 
where $|n_1|\sim|n- n_1|$ signifies the resonant product $\pe$. 
This yields~\eqref{XY0}
for $\ft{\<31p>}^{(4)}_N$.

From Ito's lemma (see also the product formula, 
Lemma \ref{LEM:prod}), 
\eqref{sto2},  and \eqref{X2} with \eqref{XX2}, we have
\begin{align*}
\ft{  \<31p>}^{(2)}_N   (n, t)
&  =  
3  
\int_0^t 
I_2 \big[  g_{n, t, t'}(z_2,z_3)  \big] dt', 
\end{align*}

\noi
where $g_{n, t, t'}$ is defined by 
\begin{align}\label{XY5}
\begin{split}
g_{n, t, t'}(z_2, z_3)  &= 
\sum_{\substack{|n_1|\leq N\\|n_1| \sim |n_{123}|}}
\ind_{n = n_{23}}\cdot \ind_{|n_2|\leq N}\cdot \ind_{|n_3|\leq N}
\int_{0}^{t'}    \frac{\sin((t - t')\jb{n_{123}})}{\jb{n_{123}}}    \\
& \qquad  \times  
\bigg(\prod_{j = 1}^3 \frac{\sin((t' - t_j)\jb{n_j})}{\jb{n_j}^{1+\al}} 
\cdot  \ind_{[0,t']}(t_j) \bigg)
\frac{\sin((t - t_1)\jb{n_1})}{\jb{n_1}^{1+\al}}  dt_1.
\end{split}
\end{align}

\noi
Note that $g_{n, t, t'}(z_2, z_3)$ is symmetric (in $z_2$ and $z_3$).
From Fubini's theorem (Lemma \ref{LEM:B3}), 
we have
\begin{align}
\begin{split}
\ft{  \<31p>}^{(2)}_N  (n, t)
&  =  
3  
I_2 \bigg[ \int_0^t  g_{n, t, t'}(z_2,z_3) dt' \bigg] .
\end{split}
\label{XY5a}
\end{align}

\noi
We now apply  Lemma \ref{LEM:B1}
to compute  the second moment of\eqref{XY5a}.
Then,
with $\kk(\bar n)$ as in~\eqref{X6a}, 
it follows from 
expanding the sine functions in 
\eqref{XY5} in terms of the complex exponentials
 and switching the order of integration in $t'$ and $t_1$ that 
\begin{align*}
\E\big[|\<31p>^{(2)}_N  (n, t)|^2\big]
& \sim 
\bigg\| \int_0^t  g_{n, t, t'}(n_2, t_2,n_3, t_3) dt' \bigg\|_{\l^2_{n_2, n_3} L^2_{t_2,t_3}}^2
\\
& \les 
\sum_{\eps_0, \eps_1, \eps_2, \eps_3 \in \{-1, 1\}}
\sum_{\substack{n = n_2 + n_3\\|n_j|\leq N}}
\frac{1}{\jb{n_2}^{2+2\al}\jb{n_3}^{2+2\al}}
 \\
& \hphantom{X}
\times 
\bigg( \sum_{\substack{|n_1|\leq N}}
\frac{1}{\jb{\kk(\bar n)}\jb{n_{123}} \jb{n_1}^{2 + 2\al}}
\bigg)^2
 \\
& \les
\sum_{\eps_0, \eps_1, \eps_2, \eps_3 \in \{-1, 1\}} 
\sum_{\substack{n = n_2 + n_3\\|n_j|\leq N}}
\frac{1}{\jb{n_2}^{2+2\al}\jb{n_3}^{2+2\al}}
\\
& \hphantom{X}
\times 
\bigg(  \sum_{m \in \Z} \sum_{ |n_1|\leq N}
\frac{1}{\jb{m} \jb{n_{123}} \jb{n_1}^{2+2\al}}
\cdot \ind_{\{|\kk(\bar n) - m|\leq 1\}}
\bigg)^2 .
\end{align*}

\noi
Under the condition $|n_1| \sim |n_{123}|$ and $n = n_2 + n_3$, 
we have $|n_1|\ges |n|$.
Then, by applying the basic resonant estimate (Lemma \ref{LEM:A2})
and Lemma \ref{LEM:SUM}, 
we obtain
\begin{align}
\label{XY61}
\begin{split}
\E\big[|\<31p>^{(2)}_N  (n, t)|^2\big]
& \les \frac 1{\jb{n}^{2}} \sum_{\substack{|n| \les N_1 \les N\\\text{dyadic}}}
\frac{\log^2 (2+ N_{1})}{N_1^{4\al}}
\sum_{\substack{n = n_2 + n_3\\|n_j|\le N_j}}
\frac{1}{\jb{n_2}^{2+2\al}\jb{n_3}^{2+2\al}} \\
& \hphantom{XX}
\times      \\
& \les \frac 1{\jb{n}^{2+4\al-}} 
\sum_{\substack{n = n_2 + n_3\\|n_j|\le N_j}}
\frac{1}{\jb{n_2}^{2+2\al}\jb{n_3}^{2+2\al}}  \les \jb{n}^{-3-8\al+}.
\end{split}
\end{align}

\noi
This computation with Lemma \ref{LEM:reg}
shows that 
$\<31p>^{(2)}_N $ is even smoother and has (spatial) regularity $4\al-$.

Therefore, 
putting \eqref{XY3} and \eqref{XY61} together, 
we obtain 
the desired bound 
\eqref{XY0}.
\end{proof}

\subsection{Quintic stochastic term}

In this subsection, we 
present the proof of Lemma \ref{LEM:sto3}\,(i)
on  the quintic stochastic process $\<320>_N$ defined in \eqref{sto6}.
In view of Lemma \ref{LEM:lin1}, we prove the following bound;
given any $\eps >0$ and sufficiently small $\delta_2 >0$,
there exists $\theta >0$ such that
\begin{align}
\label{T1}
\Big\| \big\|\<32>_N \big\|_{X_T^{ \al -\frac12 - \eps,  -\frac12 + \delta_2}}\Big\|_{L^p(\Omega)} \les p^{\frac52} T^\theta
\end{align}

\noi
for any $p \ge 1$ and $0 < T \le 1$, uniformly in $N \in \N$.

We start by computing the space-time Fourier transform of $\<32>_N$ with a  time cutoff.
As shown in \eqref{sto6}, 
the quintic stochastic objects $\ft{\<32>}_N$ is a convolution  of  $\ft{\<30>}_N$ in \eqref{sto4}
and  $\ft{\<2>}_N$ in~\eqref{sto3}:
\begin{align}
 \ft{\<32>}_N(n, t) = \sum_{n = n_{123} + n_{45}}
\ft{\<30>}_N (n_{123}, t)
\, \ft{\<2>}_N (n_{45}, t).
\label{Bes2}
\end{align}

\noi
Using  Lemma \ref{LEM:B3}, 
we can write 
$\ft{\<30>}_N$ 
and  $\ft{\<2>}_N$ 
as multiple stochastic integrals:
\begin{align}
\begin{split}
\ft{  \<30> }_N  (n, t)
&  =  
\int_0^t 
I_3 \big[  f_{n, t, t'}(z_1, z_2,z_3)  \big] dt'
  =  
I_3 \bigg[ \int_0^t  f_{n, t, t'}(z_1, z_2,z_3) dt' \bigg], \\
\ft{\<2>}_N(n, t) & = I_2 [g_{n, t}], 
\end{split}
\label{Wick1}
\end{align}

\noi
where $f_{n, t, t'}$ 
and $g_{n, t}$ are defined by 
\begin{align}
\begin{split}
f_{n, t, t'}(z_1, z_2, z_3)  
&= 
\ind_{n = n_{123}}
\cdot     \frac{\sin((t - t')\jb{n_{123}})}{\jb{n_{123}}}    \\
& \quad  \times  
\bigg(\prod_{j = 1}^3 \frac{\sin((t' - t_j)\jb{n_j})}{\jb{n_j}^{1+\al}} 
\cdot \ind_{|n_j|\leq N}
\cdot  \ind_{[0,t']}(t_j) \bigg), \\
g_{n, t}(z_1, z_2)
& =  \ind_{n = n_{12}} \cdot \bigg(\prod_{j = 1}^2  \frac{\sin((t - t_j)\jb{n_j})}{\jb{n_j}^{1+\al}} 
\cdot \ind_{|n_j|\le N}
\cdot \ind_{[0,t]}(t_j)\bigg).
\end{split}
\label{Wick1a}
\end{align}

\noi
By the product formula (Lemma \ref{LEM:prod}) to \eqref{Bes2}, 
we can decompose $\ft {\<32>}_N$ into the components in the homogeneous 
Wiener chaoses $\H_k$, $k =  1, 3, 5$:
\begin{align}
\label{product1}
\ft{\<32>}_N (n,t) = \ft{\<32>}_N^{(5)} (n,t)+  \ft{\<32>}_N^{(3)}(n,t) +  \ft{\<32>}_N^{(1)} (n,t),
\end{align}

\noi
where ${\ft{\<32>}_N}^{(5)} \in \H_5$,  ${\ft{\<32>}_N}^{(3)} \in \H_3$, and ${\ft{\<32>}_N}^{(1)} \in \H_1$.
By taking the Fourier transforms in time, 
the relation \eqref{product1} still holds.
Then, by using 
the orthogonality of $\H_5$, $\H_3$,  and $\H_1$, 
we have 
\begin{align*}
\E\Big[|\ft {\<32>}_N(n, t)|^2\Big]
= \sum_{j \in \{1, 3, 5\}}\E\Big[|\ft{\<32>}_N^{(j)}(n, t)|^2\Big] .
\end{align*}

\noi
Hence, it suffices to prove \eqref{T1}
for each $\<32>^{(j)}_N$, $j = 1, 3, 5$.

\medskip
\noi
$\bullet$ {\bf Case (i): Non-resonant term ${\ft{\<32>}_N}^{(5)}$.}
From \eqref{Wick1} and \eqref{Wick1a}, 
we have 
\begin{align*}
{\ft{\<32>}_N}^{(5)} (n,t)
& = I_5 \big[f^{(5)}_{n, t}\big], 
\end{align*}

\noi
where $f^{(5)}_{n, t}$ is defined by 
\begin{align}
\begin{split}
f^{(5)}_{n, t}(z_1, z_2, z_3, z_4, z_5)  
&= 
\ind_{n = n_{12345}}
\cdot \int_0^t    \frac{\sin((t - t')\jb{n_{123}})}{\jb{n_{123}}}    \\
& \quad  \times  
\bigg(\prod_{j = 1}^3 \frac{\sin((t' - t_j)\jb{n_j})}{\jb{n_j}^{1+\al}} 
\cdot \ind_{|n_j|\leq N}
\cdot  \ind_{[0,t']}(t_j) \bigg)
dt' \\
& \quad  \times  
 \bigg(\prod_{j = 4}^5  \frac{\sin((t - t_j)\jb{n_j})}{\jb{n_j}^{1+\al}} 
\cdot \ind_{|n_j|\le N}
\cdot \ind_{[0,t]}(t_j)\bigg).
\end{split}
\label{Q1a}
\end{align}

\noi
Let $\Sym(f^{(5)}_{n, t})$
be the symmetrization of $f^{(5)}_{n, t}$ defined in \eqref{sym}.
Then, from Lemma \ref{LEM:B1}\,(ii), we have 
\begin{align*}
{\ft{\<32>}_N}^{(5)} (n,t)
& = I_5 \big[\Sym(f^{(5)}_{n, t})\big].
\end{align*}

\noi
Then, by taking the temporal Fourier transform 
and applying Fubini's theorem (Lemma \ref{LEM:B3}), 
we have
\begin{align*}
\F_{t} \big( \ind_{[0,T]} {\ft{\<32>}_N}^{(5)} \big)  (n,\tau)
& = I_5 \big[\F_t(\ind_{[0,T]}\Sym(f^{(5)}_{n, \cdot})(\tau)\big]
 = I_5 \big[\Sym(\F_t(\ind_{[0,T]}f^{(5)}_{n, \cdot})(\tau))\big].
\end{align*}

\noi
Then,  by \eqref{X6}, Fubini's theorem, and Lemma \ref{LEM:B1}\,(iii)
with \eqref{XX2} and \eqref{XX3}, we have 
\begin{align}
\begin{split}
\Big\|\big\| 
& \ind_{[0,T]}  \<32>_N^{(5)}  \big\|_{X^{s , b}_T}\Big\|_{L^2(\O)}^2\\
& \les \sum_{\eps_0 \in \{-1, 1\}}
\sum_{n \in \Z^3}\int_\R \jb{n}^{2s} \jb{\tau}^{2b}
\big\|\Sym(\F_t(\ind_{[0,T]}f^{(5)}_{n, \cdot}(\bar z))(\tau -\eps_0\jb{n}))
\big\|_{L^2_{z_1, \dots,  z_5}  }^2d\tau, 
\end{split}
\label{Q4}
\end{align}

\noi
where  $\bar z = (z_1, \dots, z_5)$.

By expanding the sine functions in 
\eqref{Q1a} in terms of the complex exponentials, we have 
\begin{align}
\begin{split}
f^{(5)}_{n, t}  (z_1, z_2, z_3, z_4, z_5)  
&= c\cdot 
\ind_{n = n_{12345}}
\sum_{ \mathcal{E}}
\ft \eps\cdot 
   \frac{e^{i t\kk_1(\bar n )}}{\jb{n_{123}}}    
\int_{\max(t_1, t_2, t_3)}^t  
e^{- i t' \kk_2(\bar n )}
dt' \\
& \quad  \times  
 \bigg(\prod_{j = 1}^5  \frac{1}{\jb{n_j}^{1+\al}} 
\cdot \ind_{|n_j|\le N}\bigg)
\bigg(\prod_{j =4}^5  \ind_{[0,t]}(t_j)\bigg)
F_1(z_1, \dots, z_5), 
\end{split}
\label{Q6}
\end{align}

\noi
where 
$F_1(z_1, \dots, z_5)$ is independent of $t$ and $ t'$
with $|F_1|\leq 1$.
Here, $\EE$, $\ft \eps$, $\kk_1(\bar n)$, 
and $\kk_2(\bar n)$ are defined by 
\begin{align}
\begin{split}
\mathcal{E}  = \big\{\eps_1, &  \dots, \eps_5,  \eps_{123} \in \{-1, 1\}\big\},
\qquad 
\ft \eps  = \eps_{123} \prod_{j= 1}^5 \eps_j,  \\
\kk_1(\bar n) & = \eps_{123}  \jb{n_{123}}
+ \eps_{4}  \jb{n_{4}}+ \eps_{5}  \jb{n_{5}}, \\
\kk_2(\bar n) & = \eps_{123}  \jb{n_{123}} 
- \eps_{1}  \jb{n_{1}}- \eps_{2}  \jb{n_{2}} - \eps_{3}  \jb{n_{3}}.
\end{split}
\label{Q6a}
\end{align}

\noi
By integrating in $t'$, we have 
\begin{align}
\int_{\max(t_1, t_2, t_3)}^t  
e^{- i t' \kk_2(\bar n )}
dt' 
 =   \frac{e^{-it \kk_2(\bar n)} -e^{-it_{123}^* \kk_2(\bar n)}}{-i \kk_2(\bar n)}, 
\label{Q7}
\end{align}

\noi
where $t_{123}^* = \max(t_1, t_2, t_3)$.
Then, from \eqref{Q6} and \eqref{Q7}, we have
\begin{align}
\begin{split}
 \big|\F_t(\ind_{[0,T]}f^{(5)}_{n, \cdot}(\bar z) )(\tau - \eps_0 \jb{n})\big|
& \les \ind_{n = n_{12345}} \frac{1}{\jb{\kk_2(\bar n)}\jb{\min(|\tau - \kk_3(\bar n)|, |\tau - \kk_4(\bar n)|)}}\\
& \quad \times    \frac{1}{\jb{n_{123}}}    
 \bigg(\prod_{j = 1}^5  \frac{1}{\jb{n_j}^{1+\al}} \cdot \ind_{|n_j|\le N}
\cdot  \ind_{[0,T]}(t_j)\bigg), 
\end{split}
\label{Q8}
\end{align}

\noi
where   $\kk_3(\bar n)$
and $\kk_4(\bar n)$ are defined by 
\begin{align}
\begin{split}
\kk_3(\bar n) & = \eps_0 \jb{n_{12345}} + \eps_{123}  \jb{n_{123}}
+ \eps_{4}  \jb{n_{4}}+ \eps_{5}  \jb{n_{5}}, \\
\kk_4(\bar n) & = \eps_0 \jb{n_{12345}} 
+\sum_{j= 1}^5 \eps_{j}  \jb{n_{j}}.
\end{split}
\label{Q9}
\end{align}

Given dyadic $N_j \ge 1$, $j = 1,2,3,4,5$,
we denote by
$B^N_{N_1,\cdots,N_5}$
the contribution to $\ind_{[0,T]}  {\<32>}_N^{(5)} $
from $|n_j| \sim N_j$ in \eqref{Q8}.
Let $\EE_0 = \EE \cup \{\eps_0 \in \{-1, 1\}\}$
and  $N_{\max} = \max(N_1,\dots,N_5)$.
Then, from 
\eqref{Q4}, Jensen's inequality \eqref{Jen},  and \eqref{Q8}
with \eqref{PP1},  
we have 
\begin{align}
\begin{split}
\Big\|\big\| 
& \ind_{[0,T]} B^N_{N_1,\cdots,N_5}  \big\|_{X^{s -1, -\frac 12 - \dl}_T}\Big\|_{L^2(\O)}^2\\
& \les T^\ta
\sum_{ \mathcal{E}_0}
\sum_{n \in \Z^3}
\sum_{\substack{n = n_{12345}\\|n_j|\sim N_j}}
   \frac{\jb{n}^{2(s-1)}}{\jb{n_{123}}^2}    \frac{1}{\jb{\kk_2(\bar n)}^2}
 \bigg(\prod_{j = 1}^5  \frac{1}{\jb{n_j}^{2+2\al}}\bigg)\\
& \hphantom{XXXXXX}   
\times \int_\R  
\frac 1{\jb{\tau}^{1 + 2\dl}\jb{\min(|\tau - \kk_3(\bar n)|, |\tau - \kk_4(\bar n)|)}^2}d\tau
\\
 & \les T^\ta 
\sum_{ \mathcal{E}_0}
 \sup_{m,m' \in \Z}
\sum_{\substack{n_1, \dots, n_5 \in \Z^3\\|n_j|\sim N_j}}
\frac{\jb{n_{12345}}^{2(s - \al + \frac 12 \eps-1)} }{
\jb{n_{1234}}^{\frac{1}{2}\eps}\jb{n_{12}}^{\frac{1}{2}\eps}
\jb{n_{123}}^{2} \prod_{j=1}^5 \jb{n_j}^{2}} \\
& \hphantom{XXXXXX}   
\times 
\ind_{\{ |\kk_2(\bar n) -m|\leq 1\}}  
\Big(\ind_{\{ |\kk_3(\bar n) -m'|\leq 1\}}+\ind_{\{ |\kk_4(\bar n) -m'|\leq 1\}}\Big)
\end{split}
\label{Q10}
\end{align}

\noi
for some $\ta > 0$, provided that $\dl > 0$.
In the last step, we used the following bound:
\begin{align*}
 \int_\R  
& \frac 1{\jb{\tau}^{1 + 2\dl}\jb{\min(|\tau - \kk_3(\bar n)|, |\tau - \kk_4(\bar n)|)}^2}d\tau\\
& \le 
 \int_\R  
\frac 1{\jb{\tau}^{1 + 2\dl}\jb{\tau - \kk_3(\bar n)}^2}d\tau
+ 
 \int_\R  
\frac 1{\jb{\tau}^{1 + 2\dl}\jb{\tau - \kk_4(\bar n)}^2}d\tau\\
& \les \jb{\kk_3(\bar n )}^{-1-2\dl} + \jb{\kk_4(\bar n )}^{-1-2\dl}\\
& \les \sum_{m'\in \Z} 
\frac{1}
{\jb{m'}^{1+2\dl}} \Big(\ind_{\{ |\kk_3(\bar n) -m'|\leq 1\}}+\ind_{\{ |\kk_4(\bar n) -m'|\leq 1\}}\Big)
\end{align*}

\noi
for $\dl > 0$.
Then, by applying
Lemma \ref{LEM:A3} 
to \eqref{Q10}, we obtain 
\begin{align}
\Big\|\big\| 
& \ind_{[0,T]} B^N_{N_1,\cdots,N_5}  
\big\|_{X^{\al-\frac12- \eps, -\frac12 -\delta}_T}\Big\|_{L^2(\O)}^2
\les T^\ta N_{\max}^{- \dl_0}  
\label{Q11}
\end{align}

\noi
for some $\dl_0>0$, 
provided that $\eps , \dl> 0$.
Using \eqref{Q4} and \eqref{Q8}, a crude bound shows
\begin{align}
\bigg\| \Big\| \ind_{[0,T]} B^N_{N_1,\cdots,N_5} \Big\|_{X^{0,0} } 
\bigg\|_{L^2 (\Omega)}^2 \les T^\ta N_{\max}^{K} 
\label{Q12}
\end{align}

\noi
for some (possibly large) $K > 0$.
By interpolating \eqref{Q11} and \eqref{Q12}, applying 
 the Winner chaos estimate (Lemma \ref{LEM:hyp}), 
 and then summing over dyadic $N_j$, $j = 1, \dots,5$, 
we obtain 
\begin{align*}
\bigg\| \Big\| \ind_{[0,T]}  {\<32>}_N^{(5)}  \Big\|_{X^{\al-\frac12-\eps, -\frac12 +\dl_2} ([0,T])} 
\bigg\|_{L^p (\Omega)} \les p^{\frac 52} T^\ta
\end{align*}

\noi
for some $\theta >0$, 
uniformly in $N \in \N$.
Proceeding as in the end of the proof of Lemma \ref{LEM:sto1}\,(ii) on $\<3>_N$, 
a slight modification of the argument above yields
convergence of $\<32>_N^{(5)}$ to $\<32>^{(5)}$. 
Since the required modification is straightforward, 
we omit details.
A similar comment applies to 
$\<32>_N^{(3)}$ and $\<32>_N^{(1)}$ studied below.

\medskip
\noi
$\bullet$ {\bf Case (ii): Single-resonance term $\<32>_N^{(3)}$.}
In view of the product formula (Lemma~\ref{LEM:prod})\footnote{Note that both
$f_{n, t, t'}$ and $g_{n, t}$ 
in     \eqref{Wick1a} are symmetric in their arguments.}
and Definition \ref{DEF:B4}
together with 
 \eqref{Wick1} and  \eqref{Wick1a}, 
 we have
\begin{align*}
{\ft{\<32>}_N}^{(3)} (n,t)
& = I_3 \big[f^{(3)}_{n, t}\big], 
\end{align*}

\noi
where $f^{(3)}_{n, t}$ is defined by 
\begin{align*}
f^{(3)}_{n, t}(z_1, z_2,  z_4)  
&= \sum_{n_3 \in \Z^3}
\ind_{n = n_{124}}
\cdot 
 \bigg(\prod_{j = 1}^4  \ind_{|n_j|\le N}\bigg)
    \frac{\sin ((t - t_j)\jb{n_4})}{\jb{n_4}^{1+\al}} 
\cdot \ind_{[0,t]}(t_4)
 \\
& \quad  \times  
\int_0^t    \frac{\sin((t - t')\jb{n_{123}})}{\jb{n_{123}}} 
\bigg(\prod_{j = 1}^2 \frac{\sin((t' - t_j)\jb{n_j})}{\jb{n_j}^{1+\al}} 
\cdot  \ind_{[0,t']}(t_j) \bigg)
 \\
& \quad  \times  
\bigg(\int_0^{t'}  
\frac{\sin((t - t_3)\jb{n_3})
\sin((t' - t_3)\jb{n_3})}{\jb{n_3}^{2+2\al}} 
 dt_3\bigg) dt'.
\end{align*}

\noi
By the Wiener chaos estimate (Lemma \ref{LEM:hyp})
and  H\"older's inequality, we have
\begin{align}
\begin{split}
\Big\| \big\| {\<32>}_N^{(3)}   \big\|_{X^{\al - \frac12 -\eps, -\frac12 + \delta_2}_T} 
 \Big\|_{L^p(\O)}
& \les  \Big\| \big\| {\<32>}_N^{(3)}  \big\|_{L_T^{2} H_x^{\al - \frac12 -\eps} }
 \Big\|_{L^p(\O)}\\
& \les T^{\frac 12} p^{\frac32}
\sup_{t\in [0,T]} \Big\| \big\| {\<32>}_N^{(3)}(t)   \big\|_{H_x^{\al - \frac12 -\eps} } \Big\|_{L^2(\O)}
\end{split}
\label{Qa3}
\end{align} 

\noi
for small $\dl_2> 0$.
Hence, \eqref{T1} follows once we prove
\begin{align}
\sup_{t\in [0,T]} \Big\| \big\| {\<32>}_N^{(3)}  (t)\big\|_{H_x^{\al - \frac12 -\eps} } 
\Big\|_{L^2(\O)} < \infty
\label{Qa4}
\end{align}

\noi
for  $\eps >0$, 
uniformly in $N \in \N$.

With the symmetrization  $\Sym(f^{(3)}_{n, t})$
defined in  \eqref{sym}, 
it follows from   Lemma \ref{LEM:B1} 
and  Jensen's inequality \eqref{Jen} that 
\begin{align}
\begin{split}
\Big\| \big\|  {\<32>}_N^{(3)} &  (t)\big\|_{H_x^{\al - \frac12 -\eps}} \Big\|_{L^2(\O)}^2 
= 
\sum_{n \in \Z^3} 
 \jb{n}^{2\al -1 -2\eps}
\big\|I_3\big[\Sym(f^{(3)}_{n, t})\big]\big\|_{L^2(\O)}^2\\
& \les
\sum_{n \in \Z^3}
 \sum_{\substack{n = n_{124}\\ |n_j|\le N}} 
  \jb{n}^{2\al -1 -2\eps}
\bigg(\prod_{j \in \{1, 2, 4\} } \frac1{\jb{n_j}^{2+2\al}}\bigg)
\int_{[0, t]^3} |\If^{(3)} (z_1, z_2, t_4) | ^2 dt_1  dt_2  dt_4, 
\end{split} 
\label{Qa5}
\end{align}

\noi
where $\If^{(3)} (z_1, z_2, t_4) $ is defined by 
\begin{align}
\begin{split}
\If^{(3)} (z_1, z_2, t_4)& =   \sum_{|n_3|\leq N} 
 \frac1{\jb{n_{123}}
 \jb{n_3}^{2+2\al}}
 \int_{\max(t_1,t_2)}^t 
 \sin ((t-t') \jb{n_{123}})\\
& \hphantom{XXXXX} 
\times  \bigg(  \prod_{j=1}^2 \sin ((t'-t_j) \jb{n_{j}}) \bigg) \\ 
& \hphantom{XXXXX} 
\times \int_0^{t'}  
\sin ((t -t_3) \jb{n_{3}} )    \sin ((t'-t_3) \jb{n_{3}} ) dt_3 dt' .
\end{split}
\label{Qa6}
\end{align}

\noi
By switching the order of the integrals in~\eqref{Qa6}
(with $a = \max(t_1,t_2)$):
\[  \int_{a}^t 
\int_0^{t'} f dt_3 dt'
= \int_0^{a}
 \int_{a}^t 
 f dt' dt_3
 +  \int_{a}^t
 \int_{t_3}^t 
 f dt' dt_3 \]

\noi
and integrating in $t'$ first, we have 
\begin{align}
|\If^{(3)} (z_1, z_2, t_4)|
\les   
\sum_{\eps_1, \eps_2, \eps_3, \eps_{123} \in \{-1, 1\}}
\sum_{|n_3|\leq N} 
 \frac1{\jb{n_{123}}
 \jb{n_3}^{2+2\al}\jb{\kk_2(\bar n)}}, 
\label{Qa7}
\end{align}

\noi
where 
$\kk_2(\bar n)$ is as in \eqref{Q6a}.
Hence, from \eqref{Qa5},  \eqref{Qa7}, and Lemma \ref{LEM:A2}, 
we obtain 
\begin{align*}
\Big\| \big\|  {\<32>}_N^{(3)}   \big\|_{H_x^{\al - \frac12 -\eps}} \Big\|_{L^2(\O)}^2 
& \les\sum_{\eps_1, \eps_2, \eps_3, \eps_{123} \in \{-1, 1\}}
\sum_{n \in \Z^3}
 \sum_{\substack{n = n_{124}\\ |n_j|\le N}} 
  \jb{n}^{2\al -1 -2\eps}
\bigg(\prod_{j \in \{1, 2, 4\} } \frac1{\jb{n_j}^{2+2\al}}\bigg)\\
& \quad \times
 \bigg| \sum_{\substack{1\le N_3 \les N\\\text{dyadic}}}\sum_{m \in \Z}\sum_{|n_3|\sim N_3} 
 \frac{\ind_{\{|\kk_2(\bar n) - m |\le 1\}}}{\jb{n_{123}}
 \jb{n_3}^{2+2\al}\jb{m}}\bigg|^2\\
 & \les \sum_{n \in \Z^3}
 \sum_{\substack{n = n_{124}\\ |n_j|\le N}} 
\frac{1} { \jb{n}^{1 - 2\al + 2\eps}\jb{n_{12}}^2}
\bigg(\prod_{j \in \{1, 2, 4\} } \frac1{\jb{n_j}^{2+2\al}}\bigg)\\
& 
=  \sum_{|n_1|\leq N}
\frac{1}{\jb{n_1}^{2+2\al}}
\Bigg\{\sum_{|n_2|\leq N}
\frac{1} {\jb{n_{12}}^2\jb{n_2}^{2+2\al}}\\
& \quad \times
\bigg(\sum_{|n_4|\leq N}
\frac{1} { \jb{n_{124}}^{1 - 2\al + 2\eps}\jb{n_4}^{2+2\al}}\bigg)\Bigg\}.
\end{align*}

\noi
By applying  Lemma \ref{LEM:SUM} iteratively, we then obtain 
\begin{align*}
\Big\| \big\|  {\<32>}_N^{(3)}   \big\|_{H_x^{\al - \frac12 -\eps}} \Big\|_{L^2(\O)}^2 
& \les
1, 
\end{align*}

\noi
provided that $\dl_1 > 0$.
This yields  \eqref{Qa4}.

\medskip
\noi
$\bullet$ {\bf Case (iii): Double-resonance term ${\ft{\<32>}_N}^{(1)}$.}
As in Case (ii), from  the product formula (Lemma~\ref{LEM:prod})
and Definition \ref{DEF:B4}
together with 
 \eqref{Wick1} and   \eqref{Wick1a}, 
 we have
\begin{align*}
{\ft{\<32>}_N}^{(1)} (n,t)
& = I_1 \big[f^{(1)}_{n, t}\big], 
\end{align*}

\noi
where $f^{(1)}_{n, t}$ is defined by 
\begin{align*}
f^{(1)}_{n, t}(z_1)
&= \sum_{n_2, n_3 \in \Z^3}
\ind_{n = n_{1}}
\cdot 
 \bigg(\prod_{j = 1}^3  \ind_{|n_j|\le N}\bigg)
 \\
& \quad  \times  
\int_0^t    \frac{\sin((t - t')\jb{n_{123}})}{\jb{n_{123}}} 
 \frac{\sin((t' - t_1)\jb{n_1})}{\jb{n_1}^{1+\al}} 
\cdot  \ind_{[0,t']}(t_1) 
 \\
& \quad  \times  
\bigg(\int_0^{t'}  \int_0^{t'}  
\prod_{j = 2}^3 \frac{\sin((t - t_j)\jb{n_j})
\sin((t' - t_j)\jb{n_j})}{\jb{n_j}^{2+2\al}} 
dt_2 dt_3\bigg) dt'.
\end{align*}

\noi
Arguing as in \eqref{Qa3}, 
it suffices to show 
\begin{align}
\label{Qb4}
\sup_{t\in [0,T]} \Big\| \big\| {\<32>}_N^{(1)}  (t) \big\|_{H_x^{\al - \frac12 -\eps}  } 
\Big\|_{L^2(\O)} < \infty
\end{align}

\noi
for  $\eps >0$, 
uniformly in $N \in \N$.

With the symmetrization  $\Sym(f^{(1)}_{n, t})$
defined in  \eqref{sym}, 
it follows from   Lemma \ref{LEM:B1}  and Jensen's inequality \eqref{Jen} that 
\begin{align}
\begin{split}
\Big\| \big\|  {\<32>}_N^{(1)} &  (t)\big\|_{H_x^{\al - \frac12 -\eps}} \Big\|_{L^2(\O)}^2 
= 
\sum_{n \in \Z^3}
 \jb{n}^{2\al -1 -2\eps}
\big\| I_1\big[\Sym(f^{(1)}_{n, t})\big]\big\|_{L^2(\O)}^2 \\
& \les  
 \sum_{|n_1|\le N}
  \jb{n_1}^{-3 -2\eps}
\int_{[0, t]} |\If^{(1)} (z_1) | ^2 dt_1, 
\end{split} 
\label{Qb5}
\end{align}

\noi
where $\If^{(1)} (z_1) $ is defined by 
\begin{align}
\begin{split}
\If^{(1)} (z_1)& =   \sum_{|n_2|,  |n_3|\leq N} 
 \frac1{\jb{n_{123}}
 \jb{n_2}^{2+2\al}
 \jb{n_3}^{2+2\al}}\\
& \hphantom{XXXXX} 
\times  \int_{t_1}^t 
 \sin ((t-t') \jb{n_{123}}
 \sin ((t'-t_1) \jb{n_{1}}) \\
& \hphantom{XXXXX} 
\times \int_0^{t'} \int_0^{t'}  
\prod_{j = 2}^3 \sin ((t -t_j) \jb{n_{j}} )    \sin ((t'-t_j) \jb{n_{j}} ) dt_2 dt_3 dt' .
\end{split}
\label{Qb6}
\end{align}

\noi
By switching the order of the integrals in \eqref{Qb6}
and integrating in $t'$ first, we have 
\begin{align}
|\If^{(1)} (z_1)|
\les   
\sum_{|n_2|,  |n_3|\leq N} 
 \frac1{\jb{n_{123}}
 \jb{n_2}^{2+2\al}
 \jb{n_3}^{2+2\al}\jb{\kk_2(\bar n)}}, 
\label{Qb7}
\end{align}

\noi
where 
$\kk_2(\bar n)$ is as in \eqref{Q6a}.
Hence, from \eqref{Qb5} and   \eqref{Qb7},
we obtain 
\begin{align*}
\Big\| \big\|  {\<32>}_N^{(1)}   \big\|_{H_x^{\al - \frac12 -\eps}} \Big\|_{L^2(\O)}^2 
& \les\sum_{\eps_1, \eps_2, \eps_3, \eps_{123} \in \{-1, 1\}}
 \sum_{|n_1|\le N}
  \jb{n_1}^{-3 -2\eps}\\
& \quad   \times
\bigg|\sum_{m \in \Z}
\sum_{|n_2|,  |n_3|\leq N} 
 \frac{\ind_{\{|\kk_2(\bar n) - m |\le 1\}}}{\jb{n_{123}}
 \jb{n_2}^{2+2\al}
 \jb{n_3}^{2+2\al}\jb{m}}\bigg|^2.
\end{align*}

\noi
Now, apply the dyadic decompositions
$|n_j|\sim N_j$,  $j = 1, 2, 3$.
By noting that $\jb{n_{12}}^\al\les N_1^\al N_2^\al$
and 
that $|\kk_2(\bar n) - m |\le 1$ implies 
$|m| \les N_{\max} = \max(N_1, N_2, N_3)$, 
it follows from 
 Lemma \ref{LEM:A5} that 
\begin{align*}
\Big\| \big\|  {\<32>}_N^{(1)}   \big\|_{H_x^{\al - \frac12 -\eps}} \Big\|_{L^2(\O)}^2 
& \les\sum_{\eps_1, \eps_2, \eps_3, \eps_{123} \in \{-1, 1\}}
\sum_{\substack{1\le N_1, N_2, N_3\les N\\\text{dyadic}}}
N_{\max}^{\g} \frac{N_1^{2\al}}{N_2^{2\al} N_3^{4\al}}\\
& \quad   \times
 \sum_{|n_1|\sim N_1}
  \jb{n_1}^{-3 -2\eps}
\bigg|\sup_{m \in \Z}
\sum_{\substack{|n_2| \sim N_2  \\|n_3| \sim N_3}}
 \frac{\ind_{\{|\kk_2(\bar n) - m |\le 1\}}}{\jb{n_{123}}
\jb{n_{12}}^\al \jb{n_2}^{2}
 \jb{n_3}^{2}}\bigg|^2\\
& \les
\sum_{\substack{1\le N_1, N_2, N_3\les N\\\text{dyadic}}}
N_{\max}^{\g}
 \frac{N_1^{2\al-2\eps}}{N_2^{2\al} N_3^{4\al}}
 \max(N_1, N_2)^{-2\al + 2\g}\\
 & \les 1, 
\end{align*}

\noi
provided that $\eps > 0$, where 
 $\g =\g(\eps, \al) > 0$ is sufficiently small.
This yields  \eqref{Qb4}.

This concludes the proof of  Lemma \ref{LEM:sto3}\,(i).

\subsection{Septic stochastic term}\label{SUBSEC:4.3}

In this subsection, we present the proof of 
Lemma \ref{LEM:sto3}\,(ii)
on
the septic stochastic term $\Sep$ defined in \eqref{sto6a}.
Proceeding as in  \eqref{Qa3}, it suffices to show
\begin{align}
 \sup_{t\in [0,T]}\bigg\| \Big\| \Sepp(t) \Big\|_{H_x^{\al - \frac12 -\eps} } \bigg\|_{L^2(\O)} \les 1
\label{Sp1}
\end{align}

\noi
for $\eps > 0$, uniformly in $N \in \N$.
As in the previous subsections, 
we decompose $\fSepp (n,t)$ into the components in the homogeneous Wiener chaoses $\H_k$, 
$k = 1, 3, 5, 7$:
\begin{align}
\fSepp (n,t)  = \sum_{j=0}^3 \fSepp^{(2j+1)} (n,t) ,
\label{Sp2}
\end{align}

\noi
where $\fSepp^{(2j+1)}  \in \H_{2j+1}$.
From the orthogonality of $\H_k$,  we have 
\begin{align*}
\E\Big[|\fSepp(n, t)|^2\Big]
= \sum_{j=0}^3\E\Big[|\fSepp^{(2j+1)}(n, t)|^2\Big].
\end{align*}

\noi
Hence, it suffices to prove \eqref{Sp1}
for $\Sepp^{(2j+1)}$, $j = 0,1,2,3$.

\medskip

\noi
$\bullet$ {\bf Case (i): Non-resonant septic term.}
We first study  the non-resonant term $\fSepp^{(7)} \in \H_7$.
From \eqref{sto2} and \eqref{Wick1} with  \eqref{Wick1a}
and \eqref{XX2}, 
we have 
\begin{align}
\fSepp^{(7)} (n,t)
& = I_7 \big[f^{(7)}_{n, t}\big], 
\label{Sp4}
\end{align}

\noi
where $f^{(7)}_{n, t}$ is defined by 
\begin{align}
\begin{split}
f^{(7)}_{n, t}(z_1, \dots,  z_7)  
&= 
\ind_{n = n_{1234567}}\cdot \bigg(\prod_{j = 1}^7 \ind_{|n_j|\le N}\bigg)\\
& \hphantom{X} 
\times 
\int_0^t  \frac{\sin ((t-t') \jb{n_{123}})}{\jb{n_{123}}} 
\bigg(\prod_{j = 1}^3 \frac{\sin((t' - t_j)\jb{n_j} )}{\jb{n_j}^{1+\al}} 
\ind_{[t_j,t]} (t')\bigg)
dt' \\
& \hphantom{X} 
\times\int_0^t  \frac{\sin ((t-t'') \jb{n_{456}})}{\jb{n_{456}}}
\bigg(\prod_{j = 4}^6 \frac{\sin((t'' - t_j)\jb{n_j} )}{\jb{n_j}^{1+\al}} 
\ind_{[t_j,t]} (t'')\bigg) dt'' \\
& \hphantom{X} \times \frac{\sin ( (t-t_7) \jb{n_7})}{\jb{n_7}^{1+\al}}
\ind_{[0,t]} (t_7).
\end{split}
\label{Sp5}
\end{align}

\noi
By defining  the amplitude $\Phi$ by 
\begin{align}
\Phi(t, z_1, z_2, z_3) = \int_{\max(t_1,t_2,t_3)}^t  
\frac{\sin ((t-t') \jb{n_{123}})}{\jb{n_{123}}} 
\prod_{j= 1}^3 \frac{\sin((t' - t_j)\jb{n_j} )}{\jb{n_j}^{1+\al}} dt', 
\label{Sp6}
\end{align}

\noi
we have 
\begin{align*}
f^{(7)}_{n, t}(z_1, \dots,  z_7)  
& = \Phi(t, z_1, z_2, z_3) \Phi(t, z_4, z_5, z_6)
\frac{\sin ( (t-t_7) \jb{n_7})}{\jb{n_7}^{1+\al}}.
\end{align*}

\noi
Let $\kk_2(\bar n)$ be as in \eqref{Q6a}.
Then, from \eqref{Sp6}, 
we have 
\begin{align*}
\sup_{t \in [0,T]} | \Phi(t, z_1, z_2, z_3)|  
 \les K (n_1,n_2,n_3)\prod_{j = 1}^3 \jb{n_j}^{-\al} ,
\end{align*}

\noi
where $K (n_1,n_2,n_3)$ is defined by 
\begin{align}
K (n_1,n_2,n_3) = \frac{1}{\jb{n_{123}} \jb{\kk_2(\bar n)}} 
\prod_{j= 1}^3 \frac{1}{\jb{n_j}} .
\label{Sp9}
\end{align}

\noi
Note that from Lemma \ref{LEM:A1}, we have 
\begin{align}
\sum_{\substack{n_1, n_2, n_3\in \Z^3\\|n_j|\sim N_j}} 
K^2 (n_1,n_2,n_3) \les \max (N_1, N_2, N_3)^\g
\label{Sp10}
\end{align}

\noi
for any $\g > 0$.
In view of \eqref{Sp9} and \eqref{Q6a}, 
$K (n_1,n_2,n_3)$ depends on 
$\eps_{123}, \eps_1, \eps_2, \eps_3 \in \{-1, 1\}$.
In the following, however, we 
drop the dependence on 
$\eps_{123}, \eps_1, \eps_2, \eps_3 \in \{-1, 1\}$
since \eqref{Sp10} uniformly
in $\eps_{123}, \eps_1, \eps_2, \eps_3 \in \{-1, 1\}$.
The same comment applies to \eqref{Sp12} below.

With the symmetrization  $\Sym(f^{(7)}_{n, t})$
defined in  \eqref{sym}, 
it follows from   Lemma \ref{LEM:B1}, 
  Jensen's inequality \eqref{Jen}, and 
Lemma \ref{LEM:SUM} (to sum over $n_7$)
   that 
\begin{align*}
\Big\| \big\| &  \Sepp^{(7)}   (t)\big\|_{H_x^{\al - \frac12 -\eps}} \Big\|_{L^2(\O)}^2 \\
&  \sim 
\sum_{n \in \Z^3}
 \jb{n}^{2\al -1 -2\eps}
 \sum_{\substack{n = n_{1234567}\\|n_j|\le N}}
 \int_{[0, t]}
\big| \Sym(f^{(7)}_{n, t})\big|^2 dt_1 \cdots dt_7 \\
& \les  T^\ta 
 \sum_{\substack{n_1, \dots, n_7 \in \Z^3\\|n_j|\le N}}
\frac{ \jb{n_{1234567}}^{2\al -1 -2\eps}}{\jb{n_7}^{2+2\al}}
\bigg(\prod_{j = 1}^6 \frac 1{\jb{n_j}^{2\al}} \bigg)
K^2 (n_1,n_2,n_3)K^2 (n_4,n_5,n_6)\\
& \les  T^\ta 
 \sum_{\substack{ n_1, \dots, n_6 \in \Z^3\\|n_j|\le N}}
\frac{ 1}{\jb{n_{123456}}^{2\eps}}
\bigg(\prod_{j = 1}^6 \frac 1{\jb{n_j}^{2\al}} \bigg)
K^2 (n_1,n_2,n_3)K^2 (n_4,n_5,n_6)
\end{align*}

\noi
for some $\ta > 0$, provided that $\dl_1 > 0$. 
By applying the dyadic decomposition $|n_j| \sim N_j$, $j = 1, \dots, 7$, 
and then applying \eqref{Sp10}, we then obtain
\begin{align*}
\Big\| \big\|   \Sepp^{(7)}   (t)\big\|_{H_x^{\al - \frac12 -\eps}} \Big\|_{L^2(\O)}^2 
  \sim 
\sum_{\substack{1\le N_1, \dots, N_6\les N\\\text{dyadic}}}
\bigg(\prod_{j = 1}^6 \frac 1{N_j^{2\al-\g}} \bigg)
\les 1, 
\end{align*}

\noi
as long as $\g < 2\al$. 
This proves \eqref{Sp1}.

\medskip

\noi
$\bullet$ {\bf Case (ii): General septic terms.}
As we saw in the previous subsections, 
all other terms in \eqref{Sp2}
come from the contractions of the product of $\<30>_N \cdot \<30>_N\cdot\<1>_N$.
In order to fully describe these terms,  
we recall  the notion of a pairing from  
\cite[Definition 4.30]{Bring}
to describe the structure of the contractions.

\begin{definition}[pairing] \rm
\label{DEF:pair}
Let $J\ge1$. We call a relation $\mathcal P \subset \{1,\dots,J\}^2$
a pairing if
\begin{enumerate}
\item[(i)] $\mathcal P$ is reflexive, i.e. $(j,j) \notin \mathcal P$ for all $1 \le j \le J$,

\smallskip

\item[(ii)] $\mathcal P$ is symmetric, i.e. $(i,j) \in \mathcal P$ if and only if $(j,i) \in \mathcal P$,

\smallskip

\item[(iii)] $\mathcal P$ is univalent, i.e. for each $1 \le i \le J$, $(i,j) \in \mathcal P$ for at most one $1 \le j \le J$.
\end{enumerate}

\smallskip

\noi
If $(i,j) \in \mathcal P$, the tuple $(i,j)$ is called a pair.
If $1 \le j \le J$ is contained in a pair, we say that $j$ is paired.
With a slight abuse of notation, we also write $j \in \mathcal P$ if $j$ is paired.
If $j$ is not paired, we also say that $j$ is unpaired and write $j \notin \mathcal P$.
Furthermore, 
given a partition $\mathcal A = \{A_\l\}_{\l = 1}^L$ 
 of $\{1, \cdots, J\}$, 
we say that $\mathcal P$ respects $\mathcal A$ if $i,j\in  A_\l$ 
for some $1 \le \l \le L$ implies that $(i,j)\notin \mathcal P$.
Namely, $\mathcal P$ does not pair elements of the same set $A_\l \in \mathcal{A}$.
We say that  $(n_1, \dots, n_J) \in (\Z^3)^J$ is  admissible if $(i,j) \in \mathcal P$ implies that $n_i +  n_j = 0$.
\end{definition}

In order to represent $\fSepp^{(k)} (n,t) $, $k = 1, 3, 5$, 
as multiple stochastic integrals as in \eqref{Sp4}, 
we start with 
\eqref{Sp5} and perform a contraction over the variables $z_j = (n_j, t_j)$, 
namely,  we consider a (non-trivial)\footnote{Namely,  $\mathcal P = \varnothing$.} pairing
on $\{1, \dots, 7\}$.
Then, by integrating in $t'$ and $t''$ first
in \eqref{Sp5} after a contraction, 
a computation analogous to that in Case (i) 
yields
\begin{align}
\begin{split}
\bigg\| \Big\|  & \Sepp^{(k)}(t)   \Big\|_{H_x^{\alpha - \frac12 - \eps} } \bigg\|_{L^2(\O)}\\
&
 \les \sum_{\eps_{123}, \eps_1, \eps_2, \eps_3 \in \{-1, 1\}}\sum_{\mathcal P \in \Pi_k} 
\sum_{\{n_j\}_{j \notin \mathcal P}} 
 \jb{n_{\text{nr}}}^{2\al -1-2\eps} \\
 &
  \quad \times
 \bigg( \sum_{\{n_j\}_{j \in \mathcal P}}    
 \ind_{\substack{(n_1, \dots, n_7) 
 \\\text{admissible}}}
 \cdot 
\frac{K (n_1,n_2,n_3) K (n_4,n_5,n_6)}{\jb{n_7}^{1+\al}}
  \prod_{j=1}^6\frac1{\jb{n_j}^\al} \bigg)^2,
\end{split}
\label{Sp12}
\end{align}

\noi
where 
 $K$ is as in  \eqref{Sp9}  and 
the non-resonant frequency $n_{\text{nr}}$ is defined by 
\begin{align}
n_{\text{nr}} = \sum_{j \notin \mathcal P} n_j.
\label{nr1}
\end{align}

\noi
Here,  $\Pi_k$
denotes the collection of pairings $\mathcal P$ on $\{1, \dots, 7\}$
such that 
(i) $\mathcal P$ respects the partition $\mathcal A= \big\{\{1,2,3\}, \{4,5,6\}, \{7\}\big\}$
and (ii) $|\mathcal P| = 7-k$ (when we view $\mathcal P$ as a subset of $\{1, \dots, 7\}$).
Note
 that the estimate on $\Sepp^{(7)}$  
discussed in Case (i) 
is a special case of~\eqref{Sp12} with $\mathcal P = \varnothing$.
By applying Lemma \ref{LEM:A6} (with \eqref{PP1}), 
we then obtain 
\begin{align*}
\bigg\| \Big\|  \Sepp^{(k)}(t) \Big\|_{H_x^{\alpha - \frac12 - \eps} } \bigg\|_{L^2(\O)}\les 1, 
\end{align*}

\noi
provided that $\eps >0 $.
This concludes the proof of Lemma \ref{LEM:sto3}\,(ii).

\subsection{Random operator}
In this subsection, we present the proof of 
Lemma \ref{LEM:ran1} on the random operator $\If^{\<2>_N}$ defined in \eqref{ran1}.

In view of \eqref{Op1} and \eqref{Op2}
in the definition of $ \L^{s_1, s_2, b}_{T_0}$, 
\eqref{ran1}, 
and the nonhomogeneous linear estimate (Lemma \ref{LEM:lin1}), 
it suffices to show the following bound:
\begin{align}
\Big\| 
\sup_{T\in [0, 1]}\sup_{\, \| v \|_{X_T^{\frac12 + \dl_1, \frac12 + \dl_2}} \leq 1}
\| \<2>_N v  \|_{X_T^{-\frac12 + \dl_1, - \frac12 +2 \dl_2}} 
\Big\|_{L^p(\O)} \les  p
\label{R1}
\end{align}

\noi
for some small $\dl_1, \dl_2 >0$ and any $p \geq 1$, 
uniformly in $N \in \N$.
From \eqref{Xsb2}, 
we see that \eqref{R1} follows once we prove
\begin{align}
\Big\| 
\sup_{T\in [0, 1]}\sup_{\, \|  v \|_{X^{\frac12 + \dl_1, \frac12 + \dl_2}} \leq 1}
\|  \<2>_N v  \|_{X_T^{-\frac12 + \dl_1, - \frac12 +2 \dl_2}} 
\Big\|_{L^p(\O)} \les  p.
\label{R1a}
\end{align}

\noi
Furthermore, by  inserting a sharp  time-cutoff function on $[0, 1]$, 
we may drop the supremum in $T$ and reduce the bound
\eqref{R1a} to proving 
\begin{align}
\Big\| 
\sup_{\, \| v \|_{X^{\frac12 + \dl_1, \frac12 + \dl_2}} \leq 1}
\| \ind_{[0, 1]}(t) \cdot \<2>_N v  \|_{X^{-\frac12 + \dl_1, - \frac12 +2 \dl_2}} 
\Big\|_{L^p(\O)} \les  p.
\label{R1b}
\end{align}

\noi
As in the proof of Lemma \ref{LEM:sto1}\,(ii), 
we first prove 
\begin{align}
\Big\| 
\sup_{\, \| v \|_{X^{\frac12 + \dl_1, \frac12 + \dl_2}} \leq 1}
\| \ind_{[0, 1]}(t)\cdot  \<2>_N v  \|_{X^{-\frac12 + \dl_1, - \frac 12 - \dl}} 
\Big\|_{L^p(\O)} \les  p, 
\label{R2}
\end{align}

\noi
namely with $b = - \frac 12 - \dl < - \frac 12$ 
on the $X^{s, b}$-norm of  $\ind_{[0, 1]}(t)\cdot  \<2>_N v $
for  $\dl>0$.
In fact, we prove a frequency-localized version of \eqref{R2}
(see \eqref{H9} below)
and interpolate it with 
a trivial $X^{0, 0}$ estimate (see \eqref{H10} below), as in the proof of Lemma \ref{LEM:sto1}\,(ii) 
and Lemma \ref{LEM:sto3}\,(i), 
to establish~\eqref{R1b}
with $b = - \frac 12 + 2\dl_2 > - \frac 12$

We start by computing the space-time Fourier transform of 
$\ind_{[0, 1]}(t)\cdot \<2> _N v$.
From \eqref{Wick1} and~\eqref{Wick1a}, we have 
\begin{align*}
\F_x(\<2>_N \cdot v) (n,t) & = 
\sum_{n_3\in \Z^3}
 \ft{v}(n_3, t )  I_2[g_{n - n_3, t}], 
\end{align*}

\noi
where $g_{n - n_3, t}(z_1, z_2)$ is as in \eqref{Wick1a}.
Now, write $v =  v_1 + v_{-1}$, where
\[\ft v_{1}(n, \tau) = \ind_{[0, \infty)}(\tau)\cdot \ft  v(n, \tau)
\qquad \text{and}
\qquad 
\ft v_{-1}(n, \tau) = \ind_{(-\infty, 0)}(\tau)\cdot \ft  v(n, \tau).\]
 
\noi
Then, by noting 
$|\ft  v(n, \tau)|^2
= |\ft  v_{1}(n, \tau)|^2+ |\ft   v_{-1}(n, \tau)|^2$, 
we have
\begin{align}
\begin{split}
\|  v\|_{X^{s, b}}^2
& = \sum_{\eps_3 \in \{-1, 1\}}\|  v_{\eps_3}\|_{X^{s, b}}^2\\
& = \sum_{\eps_3 \in \{-1, 1\}}
\big\| \jb{n}^{s} \jb{\tau }^b \, 
\ft  v_{\eps_3}(n, \tau + \eps_3  \jb{n})\big\|_{\l^2_n L^2_\tau}^2.
\end{split}
\label{R3a}
\end{align}
 
\noi
With this in mind,  
we write
\begin{align}
\begin{split}
\F_{x,t}&( \ind_{[0, 1]}(t)\cdot \<2>_N  v_{\eps_3}) (n,\tau-\eps_0 \jb{n}) \\
& = \jb{n}^{\frac12 - \delta_1}
\sum_{ |n_3|\leq N } \jb{n_3}^{\frac12 + \delta_1} 
   \int_{\R} \ft  v_{\eps_3}(n_3, \tau_3 +\eps_3 \jb{n_3}) H(n,n_3,\tau,\tau_3)  d\tau_3 , 
\end{split}
\label{R4}
\end{align}

\noi
where $\eps_0, \eps_3 \in \{-1, 1\}$
and  the kernel $H = H^{\eps_0, \eps_3}$ is given by 
\begin{align*}
 H(n,n_3,\tau,\tau_3)  
 &=  \jb{n}^{-\frac12 + \delta_1} \jb{n_3}^{-\frac12 - \delta_1} 
\frac{1}{\sqrt{2\pi }} \int_0^1 
 e^{-i t (\tau  - \tau_3 - \eps_0\jb{n} - \eps_3\jb{n_3})}  
   I_2[g_{n - n_3, t}] dt.
\end{align*}

\noi
By Fubini's theorem (Lemma \ref{LEM:B3}), we can write $H$ as 
\begin{align}
\begin{split}
 H(n,n_3,\tau,\tau_3)  
 &=  \jb{n}^{-\frac12 + \delta_1} \jb{n_3}^{-\frac12 - \delta_1} 
   I_2[h_{n,  n_3, \tau, \tau_3}], 
\end{split}
\label{H2}
\end{align}

\noi
where $h_{n,  n_3, \tau, \tau_3}$ is given by 
\begin{align}
\begin{split}
h_{n,  n_3, \tau, \tau_3}(z_1, z_2)
 &=  \ind_{n- n_3 = n_{12}}
\cdot \frac{1}{\sqrt{2\pi }}  \int_0^1 
 e^{-i t (\tau  - \tau_3 - \eps_0\jb{n} - \eps_3\jb{n_3})}  \\
& \hphantom{XXXXX} 
\times \bigg(\prod_{j = 1}^2  \frac{\sin((t - t_j)\jb{n_j})}{\jb{n_j}^{1+\al}} \cdot \ind_{|n_j|\le N}
\cdot \ind_{[0,t]}(t_j)\bigg) dt.
\end{split}
\label{H3}
\end{align}

\noi
Then, by   \eqref{X6}, \eqref{R4},  Cauchy-Schwarz's inequality, 
and  \eqref{R3a}, 
 we have
\begin{align*}
 \Big\| & 
\sup_{\, \| v \|_{X^{\frac12 + \dl_1, \frac12 + \dl_2}} \leq 1}
\| \ind_{[0, 1]}(t)\cdot  \<2>_N v  \|_{X^{-\frac12 + \dl_1, - \frac 12 - \dl}} 
\Big\|_{L^p(\O)} \\
& \le
\sum_{\eps_3\in \{-1, 1\}}
\Big\| 
\sup_{\, \| v \|_{X^{\frac12 + \dl_1, \frac12 + \dl_2}} \leq 1}
\| \ind_{[0, 1]}(t)\cdot  \<2>_N v_{\eps_3}  \|_{X^{-\frac12 + \dl_1, - \frac 12 - \dl}} 
\Big\|_{L^p(\O)}\\
& \les \sup_{\eps_0, \eps_3 \in \{-1, 1\}} 
 \Bigg\|
 \sup_{\, \| v \|_{X^{\frac12 + \dl_1, \frac12 + \dl_2}} \leq 1}
\bigg\|  \jb{\tau}^{-\frac12 -\delta} \sum_{|n_3|\le N}
\jb{n_3}^{\frac12 + \delta_1}\\
& \hphantom{XXXXX}
 \times 
\bigg|\int_{\R} \ft{v}_{\eps_3}(n_3, \tau_3 +\eps_3 \jb{n_3}) H(n,n_3,\tau,\tau_3)  d\tau_3 
\bigg|
\bigg\|_{\l_n^2}
\Bigg\|_{L^p(\O; L^2_\tau)}\\
&  \les \sup_{\eps_0, \eps_3\in \{-1, 1\}} 
\bigg\| \jb{\tau}^{-\frac12 - \delta} \jb{\tau_3}^{-\frac12 - \delta_2} 
\| H(n,n_3,\tau,\tau_3)  \|_{\l^2_{n_3} \to \l^2_n} \bigg\|_{L^p(\O; L^2_{\tau, \tau_3})}\\
& \les   \sup_{\eps_0, \eps_3} \sup_{\tau,\tau_3 \in \R}  
\Big\| 
\| H(n,n_3,\tau,\tau_3)  \|_{\l^2_{n_3} \to \l^2_n} \Big\|_{L^p(\O)} , 
\end{align*}

\noi
as long as $\dl, \dl_2 > 0$, 
where, in the last step, 
we used  Minkowski's integral inequality followed
by 
 H\"older's inequality (in $\tau$ and $\tau_3$).
 Here, we viewed 
 $H(n,n_3,\tau,\tau_3)$ (for fixed $\tau, \tau_3 \in \R$)
 as an infinite dimensional matrix operator mapping
 from $\l^2_{n_3}$ into $\l^2_n$.
Hence, 
the estimate \eqref{R2} is reduced to proving 
\begin{align}
 \sup_{\eps_0, \eps_3} \sup_{\tau,\tau_3 \in \R} 
  \Big\| \| H(n,n_3,\tau,\tau_3)  \|_{\l^2_{n_3} \to \l^2_n} \Big\|_{L^p(\O)} \les  p.
\label{R5x}
\end{align}

As mentioned above, we instead establish
a frequency-localized version of \eqref{R5x}:
\begin{align}
 \sup_{\eps_0, \eps_3} \sup_{\tau,\tau_3 \in \R} 
  \Big\| \| H_{N_1, N_2, N_3} (n,n_3,\tau,\tau_3)  \|_{\l^2_{n_3} \to \l^2_n} \Big\|_{L^p(\O)} \les  p
  N_{\max}^{-\dl_0}, 
\label{R5}
\end{align}

\noi
for some small $\dl_0 > 0$, 
uniformly in dyadic $N_1, N_2, N_3\geq 1$, 
where $N_{\max} = \max(N_1, N_2, N_3)$
and $H_{N_1, N_2, N_3}$
is defined by \eqref{H2} and \eqref{H3}
with extra frequency localizations $\ind_{|n_j|\sim N_j}$, $j = 1, 2, 3$.
Namely, we have 
\begin{align}
\begin{split}
 H_{N_1, N_2, N_3}(n,n_3,\tau,\tau_3)  
 &=  \jb{n}^{-\frac12 + \delta_1} \jb{n_3}^{-\frac12 - \delta_1} 
   I_2\big[h^{N_1, N_2, N_3}_{n,  n_3, \tau, \tau_3}\big], 
\end{split}
\label{H4}
\end{align}

\noi
where $h^{N_1, N_2, N_3}_{n,  n_3, \tau, \tau_3}$ is given by 
\begin{align}
\begin{split}
h^{N_1, N_2, N_3}_{n,  n_3, \tau, \tau_3}(z_1, z_2)
 &=  \sum_{\eps_1, \eps_2\in \{-1, 1\}}c_{\eps_1, \eps_2}
 \ind_{n- n_3 = n_{12}}\cdot \ind_{|n_3|\sim N_3}\cdot
 \frac{1}{\sqrt{2\pi }} 
 \int_0^1 
 e^{-i t (\tau  - \tau_3 - \kk (\bar n) )}  \\
& \hphantom{XXXXX} 
\times \bigg(\prod_{j = 1}^2  \frac{e^{-it_j \eps_j \jb{n_j}}}{\jb{n_j}^{1+\al}} \cdot 
\ind_{\substack{|n_j|\sim N_j\\|n_j|\le N}}
\cdot \ind_{[0,t]}(t_j)\bigg) dt
\end{split}
\label{H5}
\end{align}

\noi
with  $\kk(\bar n)$  as in 
\eqref{X6a}.

For $m \in \Z$, define the tensor  $\hf^m$  by
\begin{align}
\begin{split}
\hf^m_{nn_1n_2n_3}  
& = c_{\eps_1, \eps_2} \ind_{n = n_{123} } 
\cdot
\ind_{|n_3|\sim N_3}
\bigg(\prod_{j=1}^2\ind_{\substack{|n_j|\sim N_j\\|n_j|\le N}}\bigg)\\
& \quad \times 
\ind_{\{|\kk(\bar n) - m|\le 1\}}
\frac{ \jb{n}^{-\frac12 + \dl_1}}{\jb{n_1}^{1+ \al} \jb{n_2}^{1+ \al}
 \jb{n_3}^{\frac12 + \dl_1}}.
\end{split}
\label{H6}
\end{align}

\noi
Then, from \eqref{H4}, \eqref{H5}, and \eqref{H6}, we have 
\begin{align}
\begin{split}
 H_{N_1, N_2, N_3}(n,n_3,\tau,\tau_3)  
 & = \sum_{\eps_1,  \eps_2 \in \{-1, 1\}}\sum_{m \in \Z} H^m(n,n_3, \tau,\tau_3)\\
 : \!& =\sum_{\eps_1 \eps_2 \in \{-1, 1\}} \sum_{m \in \Z} I_2\big[ \hf^m_{nn_1n_2n_3}
 \Hf^m_{n_3, \tau, \tau_3} \big], 
\end{split}
\label{H6a}
\end{align}

\noi
where $\Hf^m_{n_3, \tau, \tau_3}$ is given by 
\begin{align*}
\begin{split}
\Hf^m_{n_3, \tau, \tau_3}  (z_1, z_2)
 &=  \frac{1}{\sqrt{2\pi }} 
 \int_0^1 
 \ind_{\{|\kk(\bar n) - m| \leq 1\}}
 e^{-i t (\tau  - \tau_3 - \kk (\bar n) )}  
 \Big(\prod_{j = 1}^2  e^{-it_j \eps_j \jb{n_j}} 
\cdot \ind_{[0,t]}(t_j)\Big) dt.
\end{split}
\end{align*}

\noi
Performing $t$-integration, we have
\begin{align}
\| 
 \Hf^m_{n_3, \tau, \tau_3}  (z_1, z_2)
\|_{ \l^{\infty}_{n_1,n_2} L^2_{t_1, t_2}([0, 1]^2)} 
 \les  \jb{\tau - \tau_3 - m}^{-1}.
\label{H7}
\end{align}

\noi
Then from Lemma \ref{LEM:DNY},   \eqref{H7}, 
and Lemma \ref{LEM:Bring} (with \eqref{PP1}), 
there exists $\dl_3 > 0$ such that 
\begin{align}
\begin{split}
   \Big\| & \| H^m (n,n_3,\tau,\tau_3)  \|_{\l^2_{n_3} \to \l^2_n} \Big\|_{L^p(\O)}
\\ 
&  \les p N_{\max}^{\eps} \jb{\tau - \tau_3 - m}^{-1} \\
& \quad 
\times \max\Big( \| \hf^m \|_{n_1n_2n_3 \to n}, 
\| \hf^m \|_{n_3 \to nn_1n_2},
 \| \hf^m \|_{n_1n_3 \to nn_2}, 
\| \hf^m \|_{n_2n_3 \to nn_1} \Big) \\
&  \les  p N_{\max}^{\eps-\dl_3}  \jb{\tau - \tau_3 - m}^{-1} . 
\end{split}
\label{H8}
\end{align}

\noi
for any  $\eps >0$, provided that $\dl_1 < \al$, 
which is needed to apply Lemma \ref{LEM:Bring}.
Hence, by 
noting that the condition $|\kk(\bar n) - m|\le 1$
implies $|m| \les N_{\max}$ and 
summing over $m \in \Z$, 
the bound~\eqref{R5} follows 
from \eqref{H6a} and \eqref{H8}
(by taking $\eps>0$ sufficiently small), which in turn implies
\begin{align}
\Big\| 
\sup_{\, \| v \|_{X^{\frac12 + \dl_1, \frac12 + \dl_2}} \leq 1}
\| \ind_{[0, 1]}(t)\cdot  \<2>_N^{N_1, N_2} v_{N_3}  \|_{X^{-\frac12 + \dl_1, - \frac 12 - \dl}} 
\Big\|_{L^p(\O)} \les  p N_{\max}^{-\dl_0} 
\label{H9}
\end{align}

\noi
for some $\dl_0 > 0$, 
where $v_{N_3} = \F_x^{-1} (\ind_{|n|\sim N_3} \ft v(n))$
and
\begin{align*}
\ft{\<2>}_N^{N_1, N_2}(n, t) & = I_2 \bigg[  \ind_{n = n_{12}} \cdot \bigg(\prod_{j = 1}^2  \frac{\sin(t - t_j)\jb{n_j})}{\jb{n_j}^{1+\al}} 
\cdot \ind_{\substack{|n_j|\sim N_j\\|n_j|\le N}}
\cdot \ind_{[0,t]}(t_j)\bigg)\bigg].
\end{align*}

\noi
Namely, the frequencies $n_1$, $n_2$, and $n_3$
are localized to the dyadic blocks $\{|n_j|\sim N_j\}$, $j = 1, 2, 3$.

On the other hand,  a crude bound shows
\begin{align}
\Big\| 
\sup_{\, \| v \|_{X^{\frac12 + \dl_1, \frac12 + \dl_2}} \leq 1}
\| \ind_{[0, 1]}(t)\cdot  \<2>_N^{N_1, N_2} v_{N_3}  \|_{X^{0, 0}} 
\Big\|_{L^p(\O)} \les  p N_{\max}^K
\label{H10}
\end{align}

\noi
for some (possibly large) $K > 0$.
By interpolating \eqref{H9} and \eqref{H10}
 and then summing over dyadic $N_j$, $j = 1, \dots,3$, 
 we obtain \eqref{R1b}
 for some small $ \dl_2 > 0$.

Lastly, 
as for the convergence of $\If^{\<2>_N}$
to $\If^{\<2>}$,  
we can simply repeat the computation 
above to estimate the difference
$\ind_{[0,1]} \<2>_M  v - \ind_{[0,1]}  \<2>_N  v $
for $M \ge N \ge 1$.
 In considering the difference of 
the tensors $\hf^m$ in \eqref{H6}, 
we then 
obtain 
a new restriction  $ \max(|n_1|, |n_2|) \ges N$, 
 which allows us to gain a small negative power of $N$.
As a result, we obtain
\begin{align*}
\bigg\| 
\sup_{\, \| v \|_{X^{\frac12 + \dl_1, \frac12 + \dl_2}} \leq 1}
\Big\| \ind_{[0, 1]}(t)\cdot  \Big(\<2>_N^{N_1, N_2} - \<2>_M^{N_1, N_2}\Big) v_{N_3} 
\Big\|_{X^{-\frac12 + \dl_1, - \frac 12 - \dl}} 
\bigg\|_{L^p(\O)} \les  p N^{-\eps} N_{\max}^{-\dl_0'} 
\end{align*}

\noi
for some small $\eps, \dl_0' > 0$, 
Then, interpolating this with \eqref{H10}
and summing over dyadic blocks, 
we then obtain
\begin{align*}
\Big\| \| \If^{\<2>_M} -  \If^{\<2>_N} \|_{\L^{\frac 12  + \dl_1, \frac 12 +\dl_1,  \frac 12 +\dl_2}_{T_0}} \Big\|_{L^p(\O)} \les pN^{-\e} ,
\end{align*}
\noi
for any $p \ge 1$  
and $M \ge N \ge 1$.
Then, by applying Chebyshev's inequality, summing over $N \in \N$, 
and applying the Borel-Cantelli lemma, 
we conclude the almost sure convergence of~$\If^{\<2>_N}$.
This concludes the proof of Lemma \ref{LEM:ran1}.

\appendix

\section{Counting estimates}\label{SEC:A}

In this section, we state the counting estimates 
used in Section \ref{SEC:sto2}
to study the regularities of the stochastic terms.
These lemmas are  taken 
from Bringmann \cite{Bring}.
Note that some statements are given in 
a slightly simplified form.
The same comment applies
to Lemma \ref{LEM:Bring}.

\begin{lemma}[Proposition 4.20 in \cite{Bring}]
\label{LEM:A1}
Let $0 < s  \le \frac12$
and $0 \le  \be  \le \frac12$.
Given  $\eps_j \in \{-1, 1 \}$ for $j = 0,1,2,3$,
let 
$\kk(\bar n) = \kk_{\eps_0, \eps_1, \eps_2, \eps_3}(n_1, n_2, n_3)$ 
 be as in \eqref{X6a}.
Then, we have 
\begin{align*}
\sup_{m \in \Z} \sum_{\substack{n_1,n_2,n_3 \in \Z^3\\|n_j|\sim N_j}} 
&  \jb{n_{123}}^{2(s-1)} 
\frac{\ind_{\{|\kk(\bar n)- m| \le 1\}}}
{\jb{n_{12}}^{2\be} \prod_{j=1}^3 \jb{n_j}^{2}
}
   \les N_{\max}^{2(s- \be)},
\end{align*}

\noi
uniformly in dyadic $N_1,N_2,N_3 \ge 1$
and  $\eps_j \in \{-1, 1 \}$ for $j = 0,1,2,3$, 
where $N_{\max} = \max(N_1,N_2,N_3)$.

\end{lemma}

\begin{lemma}[Lemma 4.22\,(i) in \cite{Bring}]
\label{LEM:A1a}
Given  $\eps_j \in \{-1, 1\}$ for $j = 0,1,2,3$,
let 
$\kk(\bar n) = \kk_{\eps_0, \eps_1, \eps_2, \eps_3}(n_1, n_2, n_3)$ 
 be as in \eqref{X6a}.
Then, we have 
\begin{align*}
\sup_{m \in \Z} \sup_{n \in \Z^3} \, 
& \#\Big\{(n_1, n_2, n_3) \in \Z^3:
|n_j|\sim N_j, \, j = 1, 2, 3, \, n = n_{123}, \, 
\{|\kk(\bar n)- m| \le 1\}\Big\}  \\
& \les \med(N_1, N_2, N_3)^3 \min(N_1, N_2, N_3)^2, 
\end{align*}

\noi
uniformly in dyadic $N_1,N_2,N_3 \ge 1$
and  $\eps_j \in \{-1, 1 \}$ for $j = 0,1,2,3$.

\end{lemma}

\medskip

Next, we recall  the basic resonant estimate.

\begin{lemma}
[Lemma~4.25 in \cite{Bring}]
\label{LEM:A2}
Given  $\eps_j \in \{- 1, 1\}$ for $j = 0,1,2,3$,
let 
$\kk(\bar n) = \kk_{\eps_0, \eps_1, \eps_2, \eps_3}(n_1, n_2, n_3)$ 
 be as in~\eqref{X6a}.
Then, we have 
\begin{align*}
\sum_{m \in \Z} \sum_{\substack{n_1 \in \Z^3\\|n_1|\sim N_1}} 
\frac{\ind_{\{|\kk(\bar n)- m| \le 1\}}}{
\jb{m} \jb{n_{123}} \jb{n_1}^{2} }
  \les  \frac{\log(2+N_1)}{\jb{n_{23}}},
\end{align*}

\noi
\noi
uniformly in dyadic $N_1 \ge 1$
and  $\eps_j \in \{-1, 1 \}$ for $j = 0,1,2,3$.

\end{lemma}

The next two lemmas (and Lemma \ref{LEM:A2} above) are used
for estimating the quintic stochastic term.

\begin{lemma}
\label{LEM:A3}
Let $s \le \frac12 - \eta$ and $\be > 0$ for some $\eta > 0$.
Given  $\eps_{123}, \eps_j \in \{- 1, 1\}$ for $j = 0, \dots, 5$,
let 
$\kk_2(\bar n)$, $\kk_3(\bar n)$, and $\kk_4(\bar n)$
be as in  \eqref{Q6a} and  \eqref{Q9}.
Then, we have 
\begin{align}
\begin{split}
 & \sup_{m, m' \in \Z} \sum_{\substack{n_1, \dots, n_5 \in \Z^3\\|n_j|\sim N_j}}
\frac{ \jb{n_{12345}}^{2(s-1)}}{ 
\jb{n_{1234}}^{2\be}
\jb{n_{12}}^{2\be}
\jb{n_{123}}^{2} \prod_{j=1}^5 \jb{n_j}^{2}} \\
&\quad \times  \ind_{\{ |\kk_2(\bar n) -m|\leq 1\}}  
\Big(\ind_{\{ |\kk_3(\bar n) -m'|\leq 1\}}+\ind_{\{ |\kk_4(\bar n) -m'|\leq 1\}}\Big)
\\
 & \les \max(N_1, N_2, N_3, N_4)^{-2\be + \eps} N_5^{-\eta}
\end{split}
\label{XA1}
\end{align}

\noi
for any $\eps > 0$, 
uniformly in dyadic $N_1,\dots,N_5 \ge 1$
and   $\eps_{123}, \eps_j \in \{- 1, 1\}$ for $j = 0, \dots, 5$, 
where $N_{\max} = \max(N_1,\dots,N_5)$.
\end{lemma}

Lemma \ref{LEM:A3} is 
essentially 
Lemma 4.27 in \cite{Bring}, 
where 
the condition 
$ |\kk_4(\bar n) -m'|\leq 1$ in \eqref{XA1} is replaced
by $ |\kk_4(\bar n) + \eps_{123}\jb{n_{123}} -m'|\leq 1$.
We point out that this modification does not make any difference in the proof.
In our notation, the first step of the proof of Lemma~4.27 in \cite{Bring} is
to sum over $n_5$, using \cite[Lemma 4.17]{Bring}, 
for which the conditions  
$ |\kk_4(\bar n) -m'|\leq 1$ in \eqref{XA1} 
and $ |\kk_4(\bar n) + \eps_{123}\jb{n_{123}} -m'|\leq 1$ 
do not make any difference since the extra term 
$ \eps_{123}\jb{n_{123}}$ is fixed in summing over $n_5$.

\begin{lemma}[Lemma 4.29 in \cite{Bring}]
\label{LEM:A5}
Let $\be > 0$.
Given  $\eps_{123}, \eps_j \in \{- 1, 1\}$ for $j = 1, 2, 3$,
let 
$\kk_2(\bar n)$
be as in  \eqref{Q6a}.
Then, we have 
\begin{align*}
  \sup_{m \in \Z^3}  \sup_{|n_1| \sim N_1}
\sum_{\substack{|n_2|\sim N_2\\|n_3|\sim N_3}} 
 \frac{\ind_{\{|\kk_2(\bar n) - m |\le 1\}}}{\jb{n_{123}}
\jb{n_{12}}^\be \jb{n_2}^{2}
 \jb{n_3}^{2}}
  \les \max(N_1,N_2)^{-\be + \eps}
\end{align*}

\noi
\noi
for any $\eps>0$, 
uniformly in dyadic $N_1,N_2,N_3 \ge 1$
and   $\eps_{123}, \eps_j \in \{- 1, 1\}$ for $j = 1,2, 3$. 
\end{lemma}

Lastly, we state 
 the septic counting estimate.
 See Definition \ref{DEF:pair} in Subsection  \ref{SUBSEC:4.3}
 for the definition of a paring.

\begin{lemma}[Lemma 4.31 in  \cite{Bring}]
\label{LEM:A6}

Let $\frac12 < s < 1$ and $\be > 0$. 
Given  $\eps_{123}, \eps_j \in \{- 1, 1\}$ for $j = 1, 2, 3$,
let 
$\kk_2(\bar n)$
be as in  \eqref{Q6a}
and set 
\[
\mathcal{K} (n_1,n_2,n_3) = 
\sum_{m \in \Z} \frac{\ind_{\{|\kk_2(\bar n) -m|\le 1\}}}{\jb{m}\jb{n_{123}} \jb{n_{12}}^\be}
 \prod_{j=1}^3 \frac 1 {\jb{n_j}}.
\]

\noi
Let $\mathcal P$ be a pairing on $\{1, \cdots, 7\}$ which respects the partition 
$\big\{\{1,2,3\}, \{4, 5, 6\}, \{7\}\big\}$.
Then, we have 
\begin{align*}
  \sum_{\{n_j\}_{j \notin \mathcal P}}  & \jb{n_{\textup{nr}}}^{2(s-1)} 
 \bigg( \sum_{\{n_j\}_{j  \in \mathcal P}} 
\ind_{|n_{1234567}|\sim N_{1234567}}
\cdot \ind_{|n_{1237}|\sim N_{1237}}
 \cdot\ind_{|n_{456}|\sim N_{456}}
\cdot \ind_{|n_7|\sim N_7}
\\
& \quad 
\times 
 \ind_{\substack{(n_1, \dots, n_7) 
 \\\textup{admissible}}}
\cdot
\frac{\mathcal{K} (n_1,n_2,n_3) \mathcal{K} (n_4,n_5,n_6)}{ \jb{n_7}}\bigg)^2\\
  &    \les N_{\max}^{2 s- 1 + \eps}
\end{align*}

\noi
for any $\eps > 0$, 
uniformly in dyadic $N_{1234567}, N_{1237}, N_{456}, N_7 \ge 1$
and $\eps_{123}, \eps_j \in \{- 1, 1\}$ for $j = 1,2, 3$, 
where $N_{\max} = \max (N_1, \cdots, N_7)$
and 
$n_{\textup{nr}}$ is as in  \eqref{nr1}.

\end{lemma}

\section{Multiple stochastic integrals}
\label{SEC:B}

In this section, we go over the basic definitions
and properties of multiple stochastic integrals.
See \cite{Nua} and also \cite[Section 4]{Bring}
for further discussion.

Let $\ld$ be the measure on $Z: = \Z^3 \times \R_+$ defined by 
\[ d \ld=   d n d t, \]

\noi
where $d n$ is the counting measure on $\Z^3$. 
Given $ k \in \N$, 
we set $\ld_k = \bigotimes_{j =1}^k \ld$
and $L^2(Z^k) = L^2( (\Z^3 \times \R_+)^k,  \ld _k )$. 
Given a function $f \in L^2(Z^k)$, 
we can  adapt 
the discussion in \cite[Section 1.1]{Nua}
(in particular,   \cite[Example 1.1.2]{Nua})
to the complex-valued setting and 
define the multiple stochastic integral $I_k[f]$ by 
\begin{align*}
I_k [f] = \sum_{n_1,\dots,n_k \in \Z^3} \int_{[0,\infty)^k} 
f( n_1, t_1, \dots,  n_k, t_k) 
d B_{n_1} (t_1) \cdots d B_{n_k} (t_k).
\end{align*}

\noi
Given a function $f \in L^2(Z^k)$, we define its symmetrization $\Sym(f)$ by
\begin{align}
\Sym(f)(z_1, \dots, z_k) = \frac{1}{k!} \sum_{\s \in S_k} f( z_{\s(1)},  \dots,z_{\s(k)}  ),
\label{sym}
\end{align}		
\noi
where 
$z_j = (n_j, t_j)$ as in  \eqref{XX2} and $S_k$ denotes the symmetric group on $\{1, \dots, k\}$.
Note that by Jensen's equality, we have 
\begin{align}
|\Sym(f)(z_1, \dots, z_k)|^p \le \frac{1}{k!} \sum_{\s \in S_k} |f( z_{\s(1)},  \dots,z_{\s(k)}  )|^p
\label{Jen}
\end{align}

\noi
for any $p \geq 1$.
We say that $f$ is symmetric if $\Sym(f) = f$.
We now recall some basic properties of  multiple stochastic integrals.

\noi
\begin{lemma}\label{LEM:B1}
Let $k, \l \in \N$. 
The following statements hold for any $f \in L^2(Z^k)$ and $g \in L^2(Z^\l)$\textup{:}
\begin{itemize}
\item[\textup{(i)}] $I_k : L^2(Z^k) \to \H_k \subset L^2(\O)$ is a linear operator, 
where $\H_k$ denotes the $k$th Wiener chaos.

\smallskip

\item[\textup{(ii)}]
 $I_k[ \mathtt{Sym}(f)] = I_k[ f] $.

\smallskip
\item[\textup{(iii)}] Ito isometry\textup{:}
\[ \E \big[ I_k[f] \overline{ I_{\l}[g] } \, \big] = \ind_{k=\ell}\cdot  k! \int_{(\Z^3 \times \R)^k} \mathtt{Sym}(f) \overline{ \Sym(g) }  d\ld_k.\]

\smallskip

\item[\textup{(iv)}]
Furthermore, suppose that $f$ is symmetric.
Then, we have 
\begin{align*}
I_k[f] = k! \sum_{n_1, \cdots, n_k \in \Z^3} 
\int_{0}^\infty \int_{0}^{t_1} \int_{0}^{t_{k-1}} f(n_1,t_1, \dots, n_k,t_k)
 dB_{n_k}(t_k) \cdots dB_{n_1}(t_1),
\end{align*}

\noi
where the iterated integral on the right-hand side is understood as an iterated 
Ito integral.

\end{itemize}

\end{lemma}

We  state a version of Fubini's theorem for multiple stochastic integrals
that is convenient for our purpose.
See, for example, 
\cite[Theorem 4.33]{DPZ14}
for a version of the stochastic Fubini theorem.

\noi
\begin{lemma}\label{LEM:B3}
Let  $k \geq 1$. 
Given finite $T> 0$, let $f \in L^2 ( (\Z^3 \times [0, T])^k \times [0, T], 
d \ld _k \otimes  dt \big)$.  
\textup{(}In particular, we assume that the temporal support \textup{(}for the variables $t_1, \dots, t_k, t$\textup{)} of 
$f$ is contained in $[0, T]^{k+1}$
for any $(n_1, \dots, n_k)$.\textup{)}
Then,  we have
\begin{align}
\int_{0}^T I_k [f(\cdot,t)] d t = I_k \bigg[ \int_{0}^T f(\cdot,t) d t \bigg]
\label{fub0}
\end{align}

\noi
in $L^2(\O)$.

\end{lemma}

\noi
\begin{proof} From Lemma \ref{LEM:B1}\,(ii), we may assume that $f(z_1, \dots, z_k,  t)$ is symmetric
in $z_j = (n_j, t_j)$, $j = 1, \dots, k$. 
Let $\pmb{n} = (n_1, \dots, n_k)$
and $\pmb{t} = (t_1, \dots, t_k)$. 
From Minkowski's integral inequality,  Lemma \ref{LEM:B1}\,(iii),  
and Cauchy-Schwarz's inequality, we have 
\begin{align}
\begin{split}
\Big\| \int_0^T I_k [(f-\varphi)(\cdot,t)]  d t  \Big\|_{L^{2}(\O)} 
& \les  \int_0^T \| (f - \varphi)(\cdot,t) \|_{\l^2_{\pmb{n}} ((\Z^3)^k; L^2_{\pmb{t}}([0, T]^k))}dt  \\
& \le T^\frac 12    \| f - \varphi \|_{\l^2_{\pmb{n}}((\Z^3)^k; L_{t, \pmb{t}}^2([0, T]^{k+1}))}. 
\end{split}
\label{fub3}
\end{align}

\noi
On the other hand, 
by Lemma \ref{LEM:B1}\,(iii)
and Cauchy-Schwarz's inequality, we have  
\begin{align}
\begin{split}
\bigg\| I_k \Big[ \int_0^T  (f -\varphi)(\cdot,t) d t \Big] \bigg\|_{L^2(\O)} 
& \sim \bigg\|  \int_0^T  (f -\varphi)(\cdot,t) d t 
 \bigg\|_{\l^2_{\pmb{n}} ((\Z^3)^k; L^2_{\pmb{t}}(\R_+^k))}\\
& \le T^\frac 12  
 \| f - \varphi \|_{\l^2_{\pmb{n}}((\Z^3)^k; L_{t, \pmb{t}}^2([0, T]^{k+1}))}. 
\end{split}
\label{fub4}
\end{align}

\noi
Hence, it follows from \eqref{fub3}, \eqref{fub4} and the density\footnote{By identifying
a function $f \in \l^2_{\pmb{n}}((\Z^3)^k;L^2_{t, \pmb{t}}([0, T]^{k+1}))$ 
with a sequence $\{f_{\pmb{n}}\}_{\pmb{n} \in   
(\Z^3)^k}\subset  L^2_{t, \pmb{t}}([0, T]^{k+1})$, 
we can approximate each $f_{\pmb{n}}$ by a smooth function
$\varphi_{\pmb{n}}$ such that 
$\| f_{\pmb{n}} - \varphi_{\pmb{n}}\|_{L^2_{t, \pmb{t}}([0, T]^{k+1})} < \eps_{\pmb{n}}$
such that $\eps_{\pmb{n}}$ is symmetric in $\pmb{n}$
and  $\sum_{\pmb{n} \in (\Z^3)^k} \eps_{\pmb{n}} = \eps$.
Then, the function $\varphi \cong \{\varphi_{\pmb{n}}\}_{\pmb{n}\in (\Z^3)^k}$
approximates $f$ 
within distance $\eps$ in $\l^2_{\pmb{n}}((\Z^3)^k;L^2_{t, \pmb{t}}([0, T]^{k+1}))$.
Since $f$ is symmetric, we can choose $\varphi$ to be symmetric.
}
of 
$\l^2_{\pmb{n}}((\Z^3)^k;C^\infty_{t, \pmb{t}}([0, T]^{k+1}))$
in $\l^2_{\pmb{n}}((\Z^3)^k; L_{t, \pmb{t}}^2([0, T]^{k+1}))$
 that  we may assume that $f$ is symmetric
 and belongs to $\l^2_{\pmb{n}}((\Z^3)^k;C^\infty_{t, \pmb{t}}([0, T]^{k+1}))$.
Furthermore, we may assume that $f$ has a compact support in $\pmb{n}$.
Namely, there exists $K > 0$ such that 
if $\max(|n_1|, \dots, |n_k|) > K$, 
then $f(n_1, t_1, \dots, n_k, t_k, t) = 0$
for any $t_1, \dots, t_k, t \in [0, T]$.
Then, together with Lemma \ref{LEM:B1}\,(iv), we have

\begin{align}
\begin{split}
& \int_0^T   I_k   [f(\cdot,t)] dt \\
& = k! 
\int_0^T \sum_{\substack{n_1, \dots, n_k \in \Z^3\\\max(|n_1|, \dots, |n_k|) \le K}}\int_{0}^{T} \int_{0}^{t_1} \cdots \int_{0}^{t_{k-1}} 
f(z_1,  \dots,z_k,t) d B_{n_k}(t_k) \cdots d B_{n_1}(t_1)dt \\ 
& =  
k!  \sum_{\substack{n_1, \dots, n_k \in \Z^3\\\max(|n_1|, \dots, |n_k|) \le K}}
 \int_0^T \int_0^T  \int_0^{t_1} \cdots \int_{0}^{t_{k-1}} 
f(z_1,  \dots,z_k,t) d B_{n_k}(t_k) \cdots d B_{n_1}(t_1)dt,  
\end{split}
\label{fub5}
\end{align} 

\noi
since the summation is over a finite set of 
indices $\pmb{n} = (n_1, \dots, n_k)$
and $f$ is symmetric.
Hence, it remains to justifying the $t$-integration
with the stochastic integrals for each fixed 
$\pmb{n} = (n_1, \dots, n_k)$.
For this reason, we suppress the dependence of $f$ on $\pmb{n} = (n_1, \dots, n_k)$ in the following.

When $k = 1$, we can exploit the smoothness of $f$
and have 
\begin{align*}
\int_0^T    \int_{0}^{T} & f(t_1, t) d B_{n_1}(t_1) dt 
   =  \int_0^T f(T, t) B_{n_1}(T) dt  - 
  \int_0^T 
  \int_{0}^{T}  B_{n_1}(t_1) \dd_{t_1} f (t_1, t) dt_1 dt \\
&   = B_{n_1}(T) \int_0^T f(T, t)  dt  - 
  \int_{0}^{T}  B_{n_1}(t_1) \dd_{t_1} \bigg(\int_0^T f (t_1, t)dt \bigg) dt_1  \\
& = \int_0^T    \int_{0}^{T} f(t_1, t) dt d B_{n_1}(t_1), 
\end{align*}

\noi
where, at the second equality, we used the standard Fubini's theorem
in view of the almost sure boundedness of $B_{n_1}$ on $[0, T]$.
This proves \eqref{fub0} when $k = 1$.

For the general case, let us first consider the innermost integral in \eqref{fub5}.
For notational simplicity, 
let us suppress all the variables of $f$ except for $t_k$ and $t$.
Let $\Dl_m = \{ 
0 \leq \tau_0 < \tau_1 < \cdots < \tau_m \leq T\}$
be a partition of $[0, T]$
and define a step function $f_m(\cdot, t)$ by setting $f_m(\tau, t) = f(\tau_{j-1}, t)$
for $\tau_{j-1} < \tau \le \tau_j$.
Then, by defining $J_m$ by 
\begin{align} 
J_m(t)  := 
\int_0^{t_{k-1}}f_m(t_k, t) dB_{n_{k}}(t_k)
= \sum_{j = 1}^{m} (\ind_{[0, t_{k-1}]} f)(\tau_{j-1}, t) \big(B_{n_k}(\tau_j) - B_{n_k}(\tau_{j-1})\big), 
\label{fb1}
\end{align}

\noi 
it follows from  
the definition of the Wiener integral that 
\begin{align} 
J_m(t) \too   \int_{0}^{t_{k-1}}  f(t_k,t) d B_{n_k}(t_k)  \quad \text{ in } L^2(\O), 
\label{fb2}
\end{align}

\noi
as $m \to \infty$ (such that $|\Dl_m|\to 0$).
By integrating \eqref{fb1} in $t$, we have  
\begin{align}
\int_0^T J_m(t) dt = 
\sum_{j = 1}^{m} \bigg(\int_0^T (\ind_{[0, t_{k-1}]} f)(\tau_{j-1}, t) dt \bigg)
\big(B_{n_k}(\tau_j) - B_{n_k}(\tau_{j-1})\big).
\label{fb3}
\end{align}

\noi
By the definition of the Wiener integral once again, we have 
\begin{align}
 \text{RHS of } \eqref{fb3}
\too 
 \int_{0}^{t_{k-1}} 
\int_0^Tf(t_k,t)dt  d B_{n_k}(t_k)  \quad \text{ in } L^2(\O), 
\label{fb4}
\end{align}

\noi
while 
from Minkowski's integral inequality, \eqref{fb2}, 
and the bounded convergence theorem (recall that $f$ is smooth), we have 
\begin{align}
\begin{split}
\bigg\|\int_0^T & J_m(t) dt
- \int_0^T
 \int_{0}^{t_{k-1}} 
f(t_k,t)  d B_{n_k}(t_k)  dt \bigg\|_{L^2(\O)}\\
& \leq 
\int_0^T \Big\| J_m(t) 
- 
 \int_{0}^{t_{k-1}} 
f(t_k,t)  d B_{n_k}(t_k)  \Big\|_{L^2(\O)} dt
\too 0, 
\end{split}
\label{fb5}
\end{align}

\noi
as $m \to \infty$.
Hence, from \eqref{fb3}, \eqref{fb4}, and \eqref{fb5}, 
we conclude that 
\begin{align}
\int_0^T
 \int_{0}^{t_{k-1}} 
f(t_k,t)  d B_{n_k}(t_k)  dt 
= 
 \int_{0}^{t_{k-1}} 
\int_0^T f(t_k,t) dt  d B_{n_k}(t_k)\quad \text{ in } L^2(\O).
\label{fb6}
\end{align}

Next, we consider 
\begin{align}
\begin{split}
\int_0^{t_{k-2}} & \int_0^T F(t_{k-1}, t)  dt dB_{n_{k-1}}(t_{k-1}) \\
& : = 
\int_0^{t_{k-2}} \int_0^T\bigg(
 \int_{0}^{t_{k-1}} 
f(t_{k-1}, t_k,t)  d B_{n_k}(t_k)  \bigg)dt  dB_{n_{k-1}}(t_{k-1}) .
\end{split}
\label{fb7}
\end{align}

\noi
Given the partition $\Dl_m$ of $[0, T]$ as above, 
we define an adaptive step function $F_m(\cdot, t)$ by setting 
$F_m(\tau, t;\o) = F(\tau_{j-1}, t;\o)$
for $\tau_{j-1} < \tau \le \tau_j$.
Then, we can simply repeat the previous computation 
(but with Ito integrals instead of Wiener integrals)
and obtain
\begin{align}
\int_0^T
 \int_{0}^{t_{k-2}} 
F(t_{k-1},t)  d B_{n_{k-1}}(t_{k-1})  dt 
= 
 \int_{0}^{t_{k-2}} 
\int_0^T F(t_{k-1},t) dt  d B_{n_{k-1}}(t_{k-1})
\label{fb8}
\end{align}

\noi
in $L^2(\O)$.
Combining \eqref{fb7} and \eqref{fb8} with \eqref{fb6}, we then obtain
\begin{align*}
\int_0^T
& \int_0^{t_{k-2}} 
 \int_{0}^{t_{k-1}} 
f(t_{k-1}, t_k,t)  d B_{n_k}(t_k)    dB_{n_{k-1}}(t_{k-1})dt \\
& = 
\int_0^{t_{k-2}} 
 \int_{0}^{t_{k-1}} 
\int_0^T
f(t_{k-1}, t_k,t)dt  d B_{n_k}(t_k)    dB_{n_{k-1}}(t_{k-1})
\end{align*}

\noi
in $L^2(\O)$.
By iterating this process, we conclude
\begin{align*}
 \int_0^T & \int_0^T  \int_0^{t_1} \cdots \int_{0}^{t_{k-1}} 
f(t_1,  \dots,t_k,t) d B_{n_k}(t_k) \cdots d B_{n_1}(t_1)dt\\
& =  \int_0^T  \int_0^{t_1} \cdots \int_{0}^{t_{k-1}} 
 \int_0^Tf(t_1,  \dots,t_k,t)dt d B_{n_k}(t_k) \cdots d B_{n_1}(t_1)
\end{align*} 

\noi
in $L^2(\O)$. Together with \eqref{fub5}, 
this proves \eqref{fub0}.
\end{proof}

We conclude this section by stating the product formula (Lemma \ref{LEM:prod}).
Before doing so, we first recall the contraction of two functions.

\begin{definition} \label{DEF:B4}\rm
Let $k, \ell \in \N$. 
Given an integer $0 \leq r \leq \min (k,l)$, 
 we define the contraction
$ f \otimes_r g $
  of $r$ indices of $f\in L^2(Z^k)$ and $g\in L^2(Z^\l)$ by 
\begin{align*}\label{contraction}
 (f \otimes_r g)(z_1, \dots, z_{k + \ell - 2 r})  
 & =   \sum_{m_1, \dots, m_r \in \Z^3} \int_{\R_+^r} 
f(z_1, \dots, z_{k-r}, \zeta_1, \dots, \zeta_r)\\
& \hphantom{XXXXX} \times 
g(z_{k+1-r}, \dots, z_{k+\ell -2r},
\wt \zeta_1, \dots, \wt \zeta_r ) ds_1 \cdots ds_r, 
\end{align*}

\noi
where
$\zeta_j = (m_j, s_j)$
and $\wt \zeta_j = (-m_j, s_j)$.

\end{definition}

Note that even if $f$ and $g$ are symmetric, 
their contraction $f\otimes_r g$ is not symmetric in general.
We now state the product formula.
See \cite[Proposition 1.1.3]{Nua}.

\noi
\begin{lemma}[product formula]\label{LEM:prod}
Let  $k, \ell \in \N$. 
Let  $f \in L^2(Z^k)$ and $g \in L^2(Z^\l)$
be symmetric functions. Then,  we have 
\begin{align*}
& I_k[f] \cdot I_{\ell}[g] = \sum_{r=0}^{\min(k,\ell)} r! \binom{k}{r} \binom{\ell}{r} I_{k + \ell - 2r}[f \otimes_r g].
\end{align*}
\end{lemma}

\section{Random tensors} \label{SEC:C}

In this section, we provide the basic definition and 
some lemmas on (random) tensors from~\cite{DNY2, Bring}.
See~\cite[Sections 2 and 4]{DNY2} and \cite[Section 4]{Bring}
for further discussion.

\begin{definition} \label{def_tensor} \rm
Let $A$ be a finite index set. We denote by $n_A$ the tuple $ (n_j : j \in A)$. 
 A tensor $h = h_{n_A}$ is a function: $(\Z^3)^{A} \to \mathbb{C} $ with the input variables $n_A$. Note that the tensor $h$ may also depend on $\o \in \O$. 
 The support of a tensor $h$ is the set of $n_A$ such that $h_{n_A} \neq 0$. 

Given a finite index set  $A$, 
let $(B, C)$ be a partition of $A$. We define the norms 
 $\| \cdot \|_{n_A}$ and 
$\| \cdot \|_{n_{B} \to n_{C}}$ by 
\[ \| h \|_{n_A}  = \|h\|_{\l^2_{n_A}} = \bigg(\sum_{n_A} |h_{n_A}|^2\bigg)^\frac{1}{2}\]
and
\begin{align}
  \| h \|^2_{n_{B} \to n_{C}} = \sup \bigg\{ 
\sum_{n_{C}} \Big| \sum_{n_{B}} h_{n_A} f_{n_{B}} \Big|^2 :  \| f \|_{\l^2_{n_{B}}} =1  \bigg\},  
\label{Z0a}
\end{align}

\noi
where  we used the short-hand notation $\sum_{n_Z}$ for $\sum_{n_Z \in (\Z^3)^Z}$ for a finite index set $Z$.
Note that, by duality, we have  $\| h \|_{n_{B} \to n_{C}} = \| h \|_{n_{C} \to n_{B}} 
= \| \cj h \|_{n_{B} \to n_{C}}$ for any tensor $h = h_{n_A}$. 
If $B = \varnothing$ or $C = \varnothing$,  then we have
$  \| h \|_{n_{B} \to n_{C}} = \| h \|_{n_A}$.
\end{definition}

For example, when $A = \{1, 2\}$, 
the norm  $\| h \|_{n_{1} \to n_{2}}$ denotes the usual operator norm
$\| h \|_{\l^2_{n_{1}} \to \l^2_{n_{2}}}$
for an infinite dimensional matrix operator $\{h_{n_1 n_2}\}_{n_1, n_2 \in \Z^3}$.
By bounding the matrix operator norm by the Hilbert-Schmidt norm (= the Frobenius norm), we have
\begin{align}
\| h \|_{\l^2_{n_{1}} \to \l^2_{n_{2}}} \le \| h\|_{\l^2_{n_1, n_2}}
\label{Z0}
\end{align}

Let $(B, C)$ be a partition of $A$.
Then, 
by duality, we can write \eqref{Z0a} as 
\begin{align*}
  \| h \|_{n_{B} \to n_{C}} = \sup \bigg\{ 
\sum_{n_{C}} \Big| \sum_{n_{B}, n_C} h_{n_A} f_{n_{B}} g_{n_C}\Big| : 
\| f \|_{\l^2_{n_{B}}} =  \| g \|_{\l^2_{n_{C}}} =1   \bigg\}, 
\end{align*}

\noi
from which we obtain
\begin{align}
\sup_{n_A}|h_{n_A}|
= \sup_{n_B, n_C}|h_{n_Bn_C}|
\le   \| h \|_{n_{B} \to n_{C}}.
\label{Z0b}
\end{align}

%
%

\medskip

Next, we recall a key deterministic tensor bound in the study of the random cubic NLW from \cite{Bring}.

\begin{lemma}[Lemma 4.33 in \cite{Bring}]\label{LEM:Bring}

Let $s < \frac 12+ \be$ for some $\be > 0$.
Given  $\eps_j \in \{-1, 1\}$ for $j = 0,1,2,3$,
let $\kk(\bar n)$
 be as in \eqref{X6a}.
For $m \in \Z$, define the tensor  $h^m$  by
\begin{align*}
h^m_{nn_1n_2n_3}  
 = 
\bigg(\prod_{j=1}^3\ind_{\substack{|n_j|\sim N_j\\|n_j|\le N}}\bigg)
\ind_{\{|\kk(\bar n) - m|\le 1\}}
\frac{ \jb{n}^{s-1}}{\jb{n_{12}}^\be\jb{n_1} \jb{n_2}
 \jb{n_3}^{\frac12}}.
\end{align*}

\noi
Then, there exists $\dl_0 > 0$ such that 
\begin{align*}
 \max\Big( & \| h^m \|_{n_1n_2n_3 \to n}, 
\| h^m \|_{n_3 \to nn_1n_2},
 \| h^m \|_{n_1n_3 \to nn_2}, 
\| h^m \|_{n_2n_3 \to nn_1} \Big) \\
& \les \max(N_1, N_2, N_3)^{-\dl_0}, 
\end{align*}

\noi
uniformly in $N \geq 1$, $m\in \Z$, dyadic $N_1,N_2,N_3 \ge 1$, 
and  $\eps_j \in \{-1, 1 \}$ for $j = 0,1,2,3$.

\end{lemma}

We conclude this section with the following random matrix estimate.
This lemma is essentially Propositions 2.8 and 4.14 in \cite{DNY2};
see also Proposition 4.50 in \cite{Bring}.
In our stochastic PDE setting, however, 
we need a slightly different formulation (in particular, adapted
to multiple stochastic integrals with general integrands)
and thus for readers' convenience, 
we present its proof.

Let $A$ be a finite index set.
As in \eqref{XX2} and \eqref{XX3}, 
we set $z_A = (k_A,t_A)$ for $(k_A, t_A) \in  (\Z^3)^A\times  \R^A$
and write $f_{z_A} = f(z_A) = f(n_A, t_A)$.

\begin{lemma}\label{LEM:DNY}
Let $A$ be a finite index set with $k = |A| \geq 1$. Let $h = h_{bcn_A}$ be a tensor such that $n_j \in \Z^3$ for each $j \in A$ and $(b,c) \in (\Z^3)^d$ for some integer $d \geq 2$. 
Given  $N \geq 1$, 
 assume that 
\begin{align}\label{cond_DNY}
\supp h \subset \big\{ |b|, |c|, |n_j| \les N
\text{ for each $j \in A$}
 \big\}.
\end{align} 

\noi
Given a \textup{(}deterministic\textup{)} tensor $h_{bcn_A} \in \l^2_{bcn_A}$, define the tensor $H = H_{bc}$ by
\begin{align}
H_{bc} =   I_k \big[ h_{bcn_A} f_{z_A} \big]
\label{Z1}
\end{align}

\noi
for   $f \in \l^{\infty}_{n_A}((\Z^3)^A; L^2_{t_A}(\R_+^A ))$, 
where $I_k$ denotes the multiple stochastic integral 
defined in Appendix~\ref{SEC:B}.
Then,  for any $\theta > 0$,  we have
\begin{align}
\big\| \| H_{bc} \|_{b \to c} \big\|_{L^p(\O)} 
\les p^{\frac k2} N^{\theta}\Big( \max_{(B,C)} \| h \|_{b n_B \to c n_C}\Big)
 \| f(n_A, t_A)\|_{\l^{\infty}_{n_A} L^2_{t_A}}, 
\label{Z1a}
\end{align}

\noi
where the maximum is taken over  all partitions $(B,C)$ of $A$.
\end{lemma}

\begin{remark}\rm
(i)  The assumption that  $h_{bcn_A} \in \l^2_{bcn_A}$
and  $f \in \l^{\infty}_{n_A}((\Z^3)^A; L^2_{t_A}(\R_+^A ))$
 ensures that the multiple stochastic integral $I_k \big[ h_{bcn_A} f_{z_A}\big]$ 
 in \eqref{Z1} 
 is well defined. 
 Note that if for instance we have a  stronger condition 
 $f \in \l^2 \big( (\Z^3)^A; L^{2}(\R_+^A) \big) $,  then the conclusion \eqref{Z1a} trivially holds without any loss in $N$. 
We also note that even if the tensor $h$ is random,
Lemma \ref{LEM:DNY} holds with the same proof as long as 
$h$ is independent of the Brownian motions $\{B_{n_A}\}$ defining multiple stochastic integrals.

\smallskip

\noi 
(ii) By translation invariance, 
we may replace the condition \eqref{cond_DNY}
in Lemma \ref{LEM:DNY}
by 
\begin{align*}
\supp h \subset \big\{ |b - b_*|, |c - c_*|, |n_j - n_{j, *}| \les N
\text{  for each $j \in A$} \big\}
\end{align*} 

\noi
\noi
for some $(b_*, c_* ) \in (\Z^3)^d$
and $n_{j, *} \in \Z^3$, $j \in A$.

\end{remark}

\begin{proof}[Proof of Lemma \ref{LEM:DNY}] 
We follow the proof of Proposition 4.14 in \cite{DNY2}
and 
use a higher order version of Bourgain's $TT^*$-argument \cite{BO96}.
%
Let $T : \ell^2_c \to \ell^2_b$ be the linear operator whose kernel is $H_{bc}$. 
Namely,  $T$ is defined by 
\begin{align}
(Tg)_b  = \sum_{c} H_{bc} g_c, \quad g \in \ell^2_c.
\label{Z2}
\end{align}

\noi
For  $j \in \N$,  we define the operator $T_j$ 
by $T_j = (T T^*)^m$ if $j = 2m$, and $T_j = (TT^*)^m T$ if $j = 2m +1$. 
We claim that  $T_j$ has a kernel which 
is given by  a linear combination of terms $\mathcal{T}_j$ of the form
\begin{align}
\mathcal{T}_j = \begin{cases} I_{\ell} \big[y_{bb'}(z_D)\big],
&   \text{when $j$ is even},   \\
I_{\ell} \big[y_{bc}(z_D)\big],&  \text{when $j$ is  odd},
\end{cases}
\label{Z3}
\end{align} 

\noi
for some finite index set $D$ and  $\ell = |D| \leq k j$, 
where 
$y_{bb'}(z_D)$ (or $y_{bc}(z_D)$)
satisfies the following bound:
\begin{align}
\begin{split}
&  \| y_{bb'}(z_D) \|_{\l^2_{bcn_D} L^2_{t_D}} 
\quad \text{\big(or $\| y_{bc}(z_D) \|_{\l^2_{bcn_D} L^2_{t_D}}$\big)}  \\
& \qquad \les \Big( \max_{(B,C)} \| h \|_{b n_B \to c n_C} \Big)^{j-1} \| h_{bcn_A} \|_{\l^2_{bcn_A}}
 \| f(n_A, t_A)\|_{\l^{\infty}_{n_A} L^2_{t_A}}^j.
\end{split}
\label{Z4}
\end{align}

\noi
where the maximum is taken over  all partitions $(B,C)$ of $A$.
Here,  the implicit constant depends on $k$, $\l$, 
and $j$.  While it grows with $j$ (and $\l$), 
this does not cause an issue since for a given small $\ta > 0$
in \eqref{Z1a}, we fix $j = j(\ta) \gg 1$.

Let  $j = 1$.
In this case,
comparing  
\eqref{Z3} with 
\eqref{Z2} and  \eqref{Z1}
and using 
Lemma \ref{LEM:B1}\,(ii), 
we have $y_{bc}(z_D) = \Sym(h_{bc n_A} f(z_A))$ with $D = A$
and thus the bound \eqref{Z4} follows from 
H\"older's inequality.
Note that, in this case, 
it follows from Lemma \ref{LEM:B1}\,(iii) that 
\[
\| y_{bc}(z_A) \|_{\l^2_{bcn_A} L^2_{t_A}}
= (k!)^{-1}
\big\|\|H_{bc}\|_{\l^2_{bc}}\big\|_{L^2(\O)}, \] 

\noi
where the right-hand side is 
the second moment of the Hilbert-Schmidt norm of the operator~$T$. 
By taking higher powers $T_j$, we control the operator norm of $T$. 

Now, assume that the claim with \eqref{Z3} and \eqref{Z4} hold true
for $j - 1$.
We assume that $j$ is odd.
The proof for even $j$ is analogous.
Noting that  $T_j = T_{j-1}T$, 
it follows from the inductive hypothesis \eqref{Z3} with \eqref{Z1}
and Lemma \ref{LEM:B1}\,(ii)
that  the kernel for $T_j$ is given by 
a linear combination of terms $\TT_j$ of the form
\begin{align*}
\begin{split}
(\mathcal{T}_j)_{bc} & = \sum_{b'} (\mathcal{T}_{j-1})_{bb'} H_{b'c}   \\
& = \sum_{b'} I_{\ell} \big[ y_{bb'}(z_D) \big] \cdot  I_k \big[ h_{b'c n_A} f(z_A) \big]\\
& = \sum_{b'} I_{\ell} \big[ \Sym(y_{bb'}(z_D)) \big] \cdot  I_k \big[ \Sym(h_{b'c n_A} f(z_A)) \big] .
\end{split} 
\end{align*}

\noi
 Then, from the product formula (Lemma \ref{LEM:prod}), we have 
\noi
\begin{align*}
(\TT_j)_{bc}  = \sum_{r=0}^{\min(k,\ell)} r! \binom{k}{r}  \binom{\ell}{r}  I_{k + \ell - 2r}
 \bigg[ \sum_{b'} \big( \Sym(y_{bb'}) \otimes_r  \Sym(h_{b'c} f) \big) \bigg]. 
\end{align*}

\noi
Hence, 
it suffices to show that 
$ \sum_{b'} ( \Sym(y_{bb'}) \otimes_r  \Sym(h_{b'c} f) ) $ 
satisfies  \eqref{Z4} for each $0 \leq r \leq \min(k,\ell)$.
For notational simplicity, we 
drop $\Sym$ in 
$\Sym(y_{bb'})$ and $\Sym(h_{b'c} f) $ in the following.
Note that this does not cause any issue since, 
in taking the $L^2(\O)$-norm, 
we can remove $\Sym$ by  Jensen's inequality \eqref{Jen}
as in Section \ref{SEC:sto2}.

Fix $0 \leq r \leq \min(k, \ell)$. 
From Definition \ref{DEF:B4} on the contraction, 
we have 
\begin{align}
\begin{split}
  \big( y_{bb'}&  \otimes_r h_{b'c} f \big)(z_B)  \\
& 
=  \sum_{n_C} \int_{\R_+^r}  
y_{bb'}(z_{B_1}, z_{C} )\cdot  (h_{b'c} f)(z_{B_2}, \wt z_C )dt_C,  
\end{split}
\label{pf_DNY3}
\end{align}

\noi
where $\wt z_C = (- n_C, t_C)$ for given $z_C = (n_C, t_C)$.
Here, $B_1$, $B_2$, and $C$ are pairwise disjoint sets such that 
$|B_1| = \ell - r$, $|B_2| = k - r$, $ |C| = r $, $B = B_1 \cup B_2$, 
and (by suitable relabeling of indices) 
\begin{align}
B_1 \cup C =  D
\qquad \text{and}\qquad  B_2 \cup C =   A.
\label{Z5} 
\end{align}

\noi
Then, from 
\eqref{pf_DNY3},  
Cauchy-Schwarz's inequality (in $t_C$), 
Minkowski's integral inequality (with $L^2_{t_B} = L^2_{t_{B_1}}L^2_{t_{B_2}}$, 
 \eqref{Z0a}, and the identification in~\eqref{Z5}, 
we have 
\begin{align}
\begin{split}
 \Big\| &   \sum_{b'} \big( y_{bb'} \otimes_r h_{b'c} f \big)(n_B, t_B) \Big\|_{\ell^2_{bc n_B}L^2_{t_B}} 
 \\
 &  \leq  
 \Big\|  \sum_{b', n_C}  
\| y_{bb'}(n_{B_1}, t_{B_1}, n_{C}, t_{C} )\|_{L^2_{t_{B_1}t_{C}}}\cdot  
\|(h_{b'c} f)(n_{B_2},t_{B_2},-n_C, t_C )\|_{L^2_{t_{B_2}t_C}}
   \Big\|_{\ell^2_{bc n_{B}} } \\
 &  
 \le   \|y_{bb'}(n_{B_1}, t_{B_1}, n_C, t_C ) 
 \|_{\ell^2_{bb'n_{B_1} n_C}L^2_{t_{B_1}t_C} }
 \Big\| \|(h_{b'c} f)(z_{B_2}, z_C)\|_{L^2_{t_{B_2}t_C}} \Big\|_{b' n_C \to cn_{B_2}} \\
 &  
=    \|y_{bb'}(z_D ) \|_{\ell^2_{bb'n_D} L^2_{t_D}} 
\Big\| h_{b'cn_A} \|f(z_{A} )\|_{L^2_{t_A}} \Big\|_{b' n_C \to cn_{A\setminus C}} .
\end{split}
 \label{pf_DNY4}
\end{align}

\noi
Moreover, from \eqref{Z0a}, 
we have 
\begin{align}
\begin{split}
\Big\| h_{b'cn_A} \|f(z_{A} )\|_{L^2_{t_A}} \Big\|_{b' n_C \to cn_{A\setminus C}} 
& \leq \| h_{b'c n_A } 
\|_{b' n_C \to cn_{A\setminus C}}   \| f(n_A, t_A)\|_{\l^{\infty}_{n_A} L^2_{t_A}} \\
& \leq \Big(\max_{(A_1,A_2)} \| h_{bc n_A } \|_{bn_{A_1} \to cn_{A_2}} \Big)
 \| f(n_A, t_A)\|_{\l^{\infty}_{n_A} L^2_{t_A}}, 
\end{split}
\label{pf_DNY5}
\end{align}

\noi
where the maximum is taken over  all partitions $(A_1,A_2)$ of $A$.
Hence, 
from \eqref{pf_DNY4}, \eqref{pf_DNY5}, and the inductive hypothesis
\eqref{Z4} (with $j -1$ in place of $j$),
we obtain \eqref{Z4} for $j$.
Therefore, by induction, the claim holds for any $j \in \N$.

We are now ready to prove \eqref{Z1a}.
Consider the product $T_{2m} = (TT^*)^m $ for $m \ge 1$. 
Let us denote by $\mathcal{R}_{2m}$ 
the kernel of $T_{2m}$, 
which consists of terms $\TT_j$, satisfying  \eqref{Z3} and~\eqref{Z4}. 
Namely, we have 
\begin{align}
(\RR_{2m})_{bb'}
= \sum_{j = 1}^J
I_{2k \l_j}
 \big[y_{bb'}^{(j)}(z_{D^{(j)}})\big]
\label{Z7}
\end{align} 

\noi
for some $J \geq 1$, $0 \le \l_j \le m$, 
and $y_{bb'}^{(j)}$, satisfying  \eqref{Z4}. 
Note that we have
 $\mathcal{R}_{2m} \in \H_{\le 2mk}$.
Then, by 
 the standard $TT^*$ argument, \eqref{Z0}, 
 Minkowski's integral inequality, \eqref{Z7}, 
the Wiener chaos estimate (Lemma \ref{LEM:hyp}), 
Lemma \ref{LEM:B1}\,(iii), and \eqref{Z4}, we obtain
\begin{align}
\begin{split}
\big\| \| H_{bc} \|_{b \to c} \big\|_{L^p(\O)} 
& = \big\| \| H_{bc} \|_{b \to c}^{2m} \big\|_{L^{\frac{p}{2m}}(\O)}^{\frac{1}{2m}} 
= \big\| \| T \|_{\ell^2_c \to \ell^2_b }^{2m} \big\|_{L^{\frac{p}{2m}}(\O)} ^{\frac{1}{2m}} \\
&   =  \big\| \| (T T^*)^m\|_{\ell^2_{b'} \to \ell^2_{b} }\big\|_{L^{\frac{p}{2m}}(\O)} ^{\frac{1}{2m}}  
 \le \big\| \| (\mathcal{R}_{2m})_{bb'} \|_{\l^2_{bb'}} \big\|_{L^{\frac{p}{2m}}(\O)} ^{\frac{1}{2m}} \\
& \le \big\| \| (\mathcal{R}_{2m})_{bb'} \|_{L^{\frac{p}{2m}}(\O)} \big\|_{\l^2_{bb'}} ^{\frac{1}{2m}} \\
&  \le p^\frac k2 
\bigg( \sum_{j = 1}^J 
 \Big\| \big\| I_{2k \l_j} \big[y_{bb'}^{(j)}(z_{D^{(j)}})\big] \big\|_{L^{2}(\O)} \Big\|_{\l^2_{bb'}} 
 \bigg)^{\frac{1}{2m}} \\
&  \les p^\frac k2 
\bigg( \sum_{j = 1}^J
 \| y_{bb'}^{(j)}(z_{D^{(j)}}) \|_{\l^2_{bcn_{D^{(j)}}} L^2_{t_{D^{(j)}}}} \bigg)^{\frac{1}{2m}} \\
& \les p^{\frac k2} \Big( \max_{(B,C)} \| h \|_{b n_B \to c n_C} \Big)^{1 - \frac{1}{2m}} 
\|h \|_{\l^2_{bcn_A}}^{\frac{1}{2m}}
 \| f(n_A, t_A)\|_{\l^{\infty}_{n_A} L^2_{t_A}}
\end{split}
\label{pf_DNY7}
\end{align}

\noi
for any $p \geq 4m$.
Moreover, from \eqref{cond_DNY} and  \eqref{Z0b}, 
we have 
\begin{align}\label{pf_DNY8}
\|h \|_{\l^2_{bcn_A}} \leq N^{\frac 3 2(d + k)} 
\sup_{b,c,n_A} |h_{bcn_A}| 
\leq N^{\frac 3 2(d + k)}  \max_{(B,C)} \| h \|_{b n_B \to c n_C}.
\end{align}

\noi
Therefore, by 
combining \eqref{pf_DNY7} and \eqref{pf_DNY8} 
and taking $m$ sufficiently large, 
we obtain the desired bound \eqref{Z1a}.
\end{proof}

\begin{ackno}\rm
T.O.~would like to thank Istv\'an Gy\"ongy for his kind continual support
since T.O.'s arrival in Edinburgh in 2013
 and also for joyful  chat
over the daily tea break.
Y.W.~would like to thank Istv\'an Gy\"ongy
for his kindness and teaching during his stay at Edinburgh in 2016-2017.

T.O.~and Y.Z.~were supported by the 
 European Research Council (grant no.~864138 ``SingStochDispDyn"). 
Y.W.~was supported by 
supported by 
 the EPSRC New Investigator Award 
 (grant no.~EP/V003178/1).
Lastly, 
the authors wish to thank the anonymous referee for the helpful comments.
\end{ackno}

\medskip

\noi
{\bf Data availability statement.}
Data sharing not applicable to this article as no datasets were generated or analysed during the current study.

\end{document}